\documentclass[oneside, 12pt]{amsart}

\usepackage{comment}

\usepackage{enumerate}
\usepackage{array}
\usepackage[utf8]{inputenc}
\usepackage[margin=1in]{geometry}
\usepackage{amsthm}
\usepackage{enumitem}  
\usepackage[OT2,T1]{fontenc}
\usepackage{amscd,amssymb, enumitem, amsmath,mathrsfs, amsfonts, amsaddr}
\usepackage{url}
\usepackage{tikz}
\usepackage{fullpage}
\usepackage{microtype}
\usetikzlibrary{decorations.pathreplacing}
\usepackage[backref=page]{hyperref}
\usepackage{hyperref}
\usepackage[utf8]{inputenc}
\usepackage{enumitem}
\usepackage[new]{old-arrows}
\usepackage{mathtools}

\usepackage{textcomp}
\usepackage{mathrsfs}
\usepackage{amsbsy}
\usepackage{amstext}
\usepackage{amsthm}
\usepackage{amssymb}

\DeclareSymbolFont{cyrletters}{OT2}{wncyr}{m}{n}
\DeclareMathSymbol{\Sha}{\mathalpha}{cyrletters}{"58}

\usepackage[utf8]{inputenc}

\newtheorem{theorem}{Theorem}[section]
\newtheorem{lemma}[theorem]{Lemma}
\newtheorem{proposition}[theorem]{Proposition}
\newtheorem*{proposition*}{Proposition}
\newtheorem{corollary}[theorem]{Corollary}
\newtheorem*{theorem*}{Theorem}
\newtheorem*{conjecture*}{Conjecture}
\newtheorem*{corollary*}{Corollary}

\theoremstyle{definition}
\newtheorem{definition}[theorem]{Definition}

\newtheorem{claim}[theorem]{Claim}
\newtheorem{remark}[theorem]{Remark}

\newtheorem{emp}[theorem]{\bf }{\kern -4pt}
\newtheorem*{acknowledgement}{Acknowledgements}

\theoremstyle{theorem}
\newtheorem{conjecture}[theorem]{Conjecture}
\theoremstyle{plain}

\newtheoremstyle{TheoremNum}
{\topsep}{\topsep}              
{\itshape}                      
{}                              
{\bfseries}                     
{.}                             
{ }                             
{\thmname{#1}\thmnote{ \bfseries #3}}
\theoremstyle{TheoremNum}
\newtheorem{conjrep}{Conjecture}
\newtheorem{cororep}{Corollary}
\newtheorem{prorep}{Proposition}
\newtheorem{thmrep}{Theorem}

\theoremstyle{remark}

\title{Rational Spherical Triangles}
\author{Haiyang Wang}
\address{}
\date{\today}

\usepackage{draftwatermark}
\SetWatermarkText{}
\SetWatermarkScale{3}

\begin{document}
	
	\maketitle
	

	\begin{abstract}
	A rational spherical triangle is a triangle on the unit sphere such that the lengths of its three sides and its area are rational multiples of $\pi$. Little and Coxeter have given examples of rational spherical triangles in 1980s. In this work, we are interested in determining all the rational spherical triangles. We introduce a conjecture on the solutions to a trigonometric Diophantine equation. An implication of the conjecture is that the only rational spherical triangles are the ones given by Little and Coxeter. We prove some partial results towards the conjecture.
	\end{abstract}

	\section{Introduction}
	
\begin{emp}
		A {\it spherical triangle} is a triangle on the unit sphere such that each side of the triangle is an arc on a great circle of the unit sphere. We call a spherical triangle {\it proper} if the length of each side of the triangle is less than $\pi$ and the area of the triangle is less than $2\pi$. 
		
\end{emp}

The following theorem gives a relation between the side lengths and the area of a spherical triangle. 
	
	\begin{theorem}[L'Huilier (\cite{Zwi} 4.9.2)]\label{th.lh}
		Let $\Delta$ be a spherical triangle on the unit sphere. Assume that $\Delta$ has side lengths $a, b, c$ and area $E$. Then
		\[\operatorname{tan}^2(\frac{1}{4}E)=\operatorname{tan}[\frac{1}{4}(-a+b+c)]\operatorname{tan}[\frac{1}{4}(a-b+c)]\operatorname{tan}[\frac{1}{4}(a+b-c)]\operatorname{tan}[\frac{1}{4}(a+b+c)]. \]
	\end{theorem}

In this paper, we are interested in understanding when the side lengths and area of a spherical triangle are all rational multiples of $\pi$. 

	Let G:=$\mathbb{Q}\pi=\{r\cdot\pi|r\in\mathbb{Q}\}.$ For a nonzero element $g$ in $G$, we define the {\it numerator} and {\it denominator} of $g$ to be that of $\frac{g}{\pi}$ in reduced form, respectively. We denote the numerator and denominator of $g$ by $\operatorname{num}(g)$ and $\operatorname{den}(g)$, respectively.

\begin{definition}
	A spherical triangle on the unit sphere is called {\it rational} if the lengths of its sides and its area are in $G$. 
\end{definition}

	\begin{definition}
		Let $\Delta$ be a proper spherical triangle on the unit sphere. Assume that $\Delta$ has side lengths $a\le b\le c$ and area $E$. We associate the tuple $(E,a,b,c)$ to $\Delta$ and call this tuple the {\it measurement} of $\Delta$. Note that $a, b$ and $c$ determine $E$ by Theorem \ref{th.lh}. We include $E$ in the tuple for completeness. Two spherical triangles are called {\it congruent} if they have the same measurement. Attached to a tuple $(E,a,b,c)$ is its $\operatorname{lcm}$, defined to be the least common multiple of the denominators of the entries in the tuple.
	\end{definition}
	

	\begin{emp}\label{emp.cl}
			In 1981, Little (\cite{Lit}) observed that the tuple $(E, a, b, c )=(\frac{1}{2}, \frac{2}{5}, \frac{1}{2},\frac{2}{3})\pi$ with $\operatorname{lcm}=30$ is the measurement of a proper rational spherical triangle. Notably, the lengths of the sides of this triangle are mutually distinct. Little asked whether this is the only  possible measurement for a proper rational spherical triangle with mutually distinct side lengths. Soon after, Coxeter \cite{Cox} provided six more measurements of such rational spherical triangles. We denote the set of these seven measurements by $\Lambda_2$: 
		\begin{align*}
			\Lambda_2:=\{&(\frac{1}{2},\frac{2}{5},\frac{1}{2},\frac{4}{5})\pi,&&(\frac{1}{4},\frac{1}{4}, \frac{1}{2}, \frac{2}{3})\pi,&&(\frac{1}{2},\frac{1}{4}, \frac{2}{3}, \frac{3}{4})\pi,&&(\frac{5}{4},\frac{1}{2}, \frac{2}{3}, \frac{3}{4})\pi,&&(1,\frac{2}{5}, \frac{2}{3}, \frac{4}{5})\pi,\\
			&(\frac{3}{2},\frac{1}{2}, \frac{2}{3}, \frac{4}{5})\pi,
			 &&(\frac{1}{2},\frac{2}{5}, \frac{1}{2},\frac{2}{3})\pi\}.\\
		\end{align*}
		
		Let 
		\begin{align*}
			\Lambda_1:=&\{(1,\frac{1}{2}, \frac{2}{3}, \frac{2}{3})\pi\}\cup\{(\frac{m}{d},\frac{m}{d},\frac{1}{2},\frac{1}{2})\pi\mid m, d\in\mathbb{N}\text{ and } 0<\frac{m}{d}\le\frac{1}{2}\}\\
			&\cup\{(\frac{m}{d},\frac{1}{2},\frac{1}{2},\frac{m}{d})\pi\mid m, d\in\mathbb{N}\text{ and } \frac{1}{2}<\frac{m}{d}<1\}.
		\end{align*}

		Coxeter remarked that tuples in $\Lambda_1$ are measurements of proper rational spherical triangles, albeit not satisfying the mutually distinct side lengths condition. A natural question to ask is the following: are there measurements of proper rational spherical triangles besides the ones given by Little and Coxeter?
		
	\end{emp}


	\begin{conjecture}\label{cj.lc}
		The tuples in the set $\Lambda_1\cup\Lambda_2$ defined in \ref{emp.cl} are the only measurements of proper rational spherical triangles. In particular, there are only finitely many measurements of proper rational spherical triangles with mutually distinct side lengths, namely, the seven measurements found by Little and Coxeter.
	\end{conjecture}

\begin{remark}
	In a recent work \cite{HLM}, Huang, Lalín, and Mila investigated the spherical Heron triangles. These spherical triangles satisfy certain rationality conditions distinct from the ones in this article. The authors are able to parameterize the spherical Heron triangles by the rational points of particular families of elliptic curves. In Section $8$ of \cite{HLM}, the authors discussed variations of their rationality definition concerning spherical triangles. In particular, they suggested defining a proper spherical triangle to be rational if its side lengths and angles are in $G$. Note that for a proper spherical triangle with angles $\alpha, \beta,\gamma$  on a unit sphere, its area is given by $\alpha+\beta+\gamma-\pi$. As a result, a spherical triangle whose side lengths and angles are in $G$ also has its area in $G$, satisfying the definition of rationality presented in this paper. The authors of \cite{HLM} observed that the family in $\Lambda_1$ in \ref{emp.cl} provides examples of measurements of spherical triangles whose side lengths and angles are in $G$. They raised the question about the existence of other such rational spherical triangles. If we relax the requirement of their rationality to solely including side lengths and areas being in $G$, then the seven measurements in the set $\Lambda_2$ and the single sporadic measurement in the set $\Lambda_1$, in \ref{emp.cl}, provide additional examples. Furthermore, if Conjecture \ref{cj.lc} holds true, then only these eight additional rational spherical triangles need to be considered to address the above question in Huang-Lalín-Mila \cite{HLM}. The angles of these eight spherical triangles can be found in \cite{Cox}.
\end{remark}

\begin{emp}

Let $a, b, c, E$ be variables, and consider {\it L'Huilier's equation}:
\begin{equation}\label{eq.lh}
	\operatorname{tan}^2(\frac{1}{4}E)=\operatorname{tan}[\frac{1}{4}(-a+b+c)]\operatorname{tan}[\frac{1}{4}(a-b+c)]\operatorname{tan}[\frac{1}{4}(a+b-c)]\operatorname{tan}[\frac{1}{4}(a+b+c)].
\end{equation}

As we will discuss in Proposition \ref{pro.bj}, measurements of rational spherical triangles correspond to solutions to Equation \eqref{eq.lh} in $G$ within certain ranges. To study the solutions to Equation \eqref{eq.lh}, it is convenient to consider the following {\it generalized L'Huilier's equation}:

\begin{equation}\label{eq.ne4}
	(\operatorname{tan}x_0)^2=(\operatorname{tan}x_1)(\operatorname{tan}x_2)(\operatorname{tan}x_3)(\operatorname{tan}x_4).
\end{equation}

A close examination of the solutions in $G$ to Equation \eqref{eq.ne4} leads us to the following conjecture.

\end{emp}

\begin{conjrep}[\ref{cj.fi}]
		$(x_0,x_1,x_2,x_3,x_4)$ is a solution to \eqref{eq.ne4} in $G$ with $0<x_i<\frac{\pi}{2}$ for $0\le i\le 4$ if and only if $(x_0,x_1,x_2,x_3,x_4)$ is either in the infinite set $\Phi$ given in \ref{emp.fam} or in the set $\Psi$ consisting of $2928$ elements given in \ref{emp.spo}.
\end{conjrep}

\medskip

\begin{prorep}[\ref{pro.coli}]
		If Conjecture \ref{cj.fi} is true, then Conjecture \ref{cj.lc} is true.
\end{prorep}

\begin{emp}\label{em.id}

In this work, we will present several partial results that confirm Conjecture \ref{cj.fi}. The main idea is the following. Suppose we have a solution to Equation \eqref{eq.ne4} in $G$. Then we can view both sides of the equation as elements in a finitely generated free abelian group $X^n$ for some integer $n$, see Remark \ref{rk.rep}. In \cite{Con1} and \cite{Con2}, Conrad constructed a basis for this abelian group. By representing both sides of the equation in Conrad's basis, and comparing the coefficients, we can deduce information about the given solution.

\end{emp}

For $n\in\mathbb{N}_{\ge 3}$. Let

\[\overline{\Phi}_{1,1}:=\{(s,s,s,t,\frac{1}{2}-t)\pi\mid s,t\in\mathbb{Q}, \; 0<s< \frac{1}{2}\text{ and } 0<t\le\frac{1}{4}\}.\]

Our first main result is the following.

\medskip

\begin{thmrep}[\ref{th.4pr2}]
		Let $n$ be an odd prime. Assume that $(x_0,x_1,x_2,x_3,x_4)$ is a solution to \eqref{eq.ne4} in $G$ with $0<x_i<\frac{\pi}{2}$ and $\operatorname{den}(x_i)\in\{n,2n,4n\}$ for each $0\le i\le 4$. Then up to reordering $x_1,x_2,x_3,x_4$, the tuple $(x_0,x_1,x_2,x_3,x_4)$ is in $\overline{\Phi}_{1,1}$.
		
\end{thmrep}
\medskip

By applying Theorem \ref{th.4pr2}, we derive the following result that verifies Conjecture \ref{cj.lc} in a special case.

\medskip

\begin{cororep}[\ref{cor.tri}]
		Suppose that $E,a,b, c\in G$ and $\operatorname{den}(E)=\operatorname{den}(a)=\operatorname{den}(b)=\operatorname{den}(c)$ is a prime. Then $(E,a,b,c)$ is the measurement of a proper rational spherical triangle if and only if $(E,a,b,c)=(\frac{1}{2},\frac{1}{2},\frac{1}{2},\frac{1}{2})\pi$.
\end{cororep}

\medskip

When $n$ in Theorem \ref{th.4pr2} is assumed to be square-free and not a prime, we obtain the following theorem.

\medskip

\begin{thmrep}[\ref{th.4msf}]
	Let $n\in\mathbb{N}_{\ge 3}$ be odd, square-free and not a prime. Assume that each prime factor of $n$ is  greater than $11$. Assume that $(x_0,x_1,x_2,x_3,x_4)$ is a solution to \eqref{eq.ne4} in $G$ with $0<x_i<\frac{\pi}{2}$ and $\operatorname{den}(x_i)=4n$ for each $0\le i\le 4$. Then up to reordering $x_1,x_2,x_3,x_4$, the tuple $(x_0,x_1,x_2,x_3,x_4)$ is in $\overline{\Phi}_{1,1}$.
\end{thmrep}
\medskip

\begin{remark}

	We cannot deduce an analogue statement as Corollary \ref{cor.tri} from Theorem \ref{th.4msf} due to the following reason. When the measurement of proper rational spherical triangle has identical non-prime denominators, the denominators of the corresponding solution to Equation \eqref{eq.ne4} can vary. To tackle this issue, it is necessary to investigate solutions in $G$ to \eqref{eq.ne4} allowing mixed denominators.

\end{remark}

Theorem \ref{th.sfn2n} addresses the case when the condition $\operatorname{den}(x_i)=4n$ in Theorem \ref{th.4msf} is replaced by $\operatorname{den}(x_i)\in\{n,2n\}$. Theorem \ref{th.4msf} and Theorem \ref{th.sfn2n} are the technical part of the article.

When $n$ in Theorem \ref{th.4pr2} is assumed to be a non square-free integer, we obtain the following theorem.
\medskip
\begin{thmrep}[\ref{th.4mnsf}]
	Let $n\in\mathbb{N}_{\ge 9 }$ be odd and non square-free.  Assume that every prime factor of $n$ is greater than $5$. Assume that $(x_0,x_1,x_2,x_3,x_4)$ is a solution to \eqref{eq.ne4} in $G$ with $0<x_i<\frac{\pi}{2}$ and $\operatorname{den}(x_i)=4n$ for each $0\le i\le 4$. Then up to reordering $x_1,x_2,x_3,x_4$, the tuple $(x_0,x_1,x_2,x_3,x_4)$ is in $\overline{\Phi}_{1,1}$.
\end{thmrep}
\medskip
Theorem \ref{th.mnsf} and \ref{th.8nsf} analyze the cases when the condition $\operatorname{den}(x_i)=4n$ in Theorem \ref{th.4mnsf} is replaced by $\operatorname{den}(x_i)\in\{n,2n\}$ and $8\mid \operatorname{den}(x_i)$, respectively.

\begin{emp}

This paper is structured as follows. Section \ref{sec.lin} examines the relation between rational spherical triangle and solution of L'Huilier's equation and the generalized L'Huilier's equation. Section \ref{sec.conj} focuses on Conjecture \ref{cj.fi}  and its implication to the possible measurements of rational spherical triangles. 
The representation of tangents of elements in $G$ using cyclotomic numbers is presented in Section \ref{sec.form}. Conrad's basis for the group generated by cyclotomic numbers is recalled in Section \ref{sec.bas}. Section \ref{sec.prime} studies the solutions to Equation \eqref{eq.ne4} in $G$ under the condition that the denominators of solutions are equal to the same prime. The basis representation for the square-free case is discussed in Section \ref{sec.bassf}. Sections \ref{sec.sf1} and \ref{sec.sf2} investigate the solutions to Equation \eqref{eq.ne4} in $G$ assuming that denominators of solutions are equal to the same square-free, non-prime number. Section \ref{sec.nonsf} addresses the same question under the condition that denominators of solutions are equal to the same non-square-free number. Solutions to Equation \eqref{eq.ne4} in $G$ that allow for certain small divisors are treated in Section \ref{sec.smal}. Finally, Section \ref{sec.gen} discusses the solutions of a generalization of Equation \eqref{eq.ne4}.

\end{emp}

\begin{acknowledgement}
This work is part of the author’s doctoral thesis at the University of Georgia. The author would like to thank his advisor, Dino Lorenzini, 
for many helpful suggestions, feedback, and constant support.
\end{acknowledgement}

	\section{Linear Transformations}\label{sec.lin}

In this section, we study the relation between measurements of rational spherical triangles and the solutions to Equation \eqref{eq.lh} and Equation \eqref{eq.ne4}.

\begin{lemma}\label{lem.ineq}
	
	Assume that $(E,a,b,c)\in\mathbb{R}^4$ is a solution to L'Huilier's Equation \eqref{eq.lh} with $0< a\le b\le c<\pi$ and $0<E<2\pi$. Then $0<a+b-c\le a-b+c\le-a+b+c<a+b+c<2\pi$.
	
\end{lemma}

\begin{proof}
	
		Since $0< a\le b\le c<\pi$, it is clear that $a+b-c\le a-b+c\le-a+b+c<a+b+c$, $0<-a+b+c,a-b+c<2\pi$, $0<a+b+c<3\pi$ and $-\pi<a+b-c<2\pi$. In particular, we know that $\operatorname{tan}(\frac{-a+b+c}{4}), \operatorname{tan}(\frac{a-b+c}{4})>0$.  It remains to show that $a+b+c<2\pi$ and $0<a+b-c$.
	
	Because $\operatorname{tan}(\frac{\pi}{2})$ is undefined, we get that $a+b+c\ne 2\pi$. We now show that $a+b+c<2\pi$. We prove it by contradiction. Assume that 
	\begin{equation}\label{eq.ine1}
		a+b+c>2\pi.
	\end{equation}
	
	Then $\frac{\pi}{2}<\frac{a+b+c}{4}<\frac{3\pi}{4}$. So $\operatorname{tan}(\frac{a+b+c}{4})<0$. Since $0<E<2\pi$, we know that $(\operatorname{tan}(\frac{E}{4}))^2>0$. Notice that
	
	\begin{equation}\label{eq.alt}
		\operatorname{tan}(\frac{a+b-c}{4})=\frac{(\operatorname{tan}(\frac{E}{4}))^2}{\operatorname{tan}(\frac{-a+b+c}{4})\operatorname{tan}(\frac{-a+b+c}{4})\operatorname{tan}(\frac{a+b+c}{4})}.
	\end{equation}
	
	Hence $\operatorname{tan}(\frac{a+b-c}{4})<0$. Because $-\pi<a+b-c<2\pi$, we know that 
	
		\begin{equation}\label{eq.ine2}
		a+b-c<0.
	\end{equation}
	
Combining inequalities \ref{eq.ine1} and \ref{eq.ine2} gives $c>\pi$, a contradiction. Hence $a+b+c<2\pi$. The inequality $0<a+b-c$ follows from Equation \eqref{eq.alt}.
\end{proof}

Let 
\begin{equation*}
	\begin{aligned}	
		\Omega_1:=\left\{
		(E, a, b, c) \;\middle|\;
		\begin{aligned}
			& \text{there is a rational spherical triangle with side lengths}\\
			& 0<a\leq b\leq c<\pi, \text{ and area } 0<E<2\pi.
		\end{aligned}
		\right\}.
	\end{aligned}
\end{equation*}
Notice that $\Omega_1$ is the set of the measurements of proper rational spherical triangles. Let
\begin{equation*}
	\begin{aligned}	
		\Omega_2:=\left\{
		( E,a, b, c) \;\middle|\;
		\begin{aligned}
			& ( E,a, b, c) \text{ is a solution to L'Huilier's Equation} \eqref{eq.lh}\text{ in } G\\
			&  \text{ such that } 0<a\leq b\leq c<\pi,\text{ and } 0<E<2\pi.
		\end{aligned}
		\right\}.
	\end{aligned}
\end{equation*}

\begin{proposition}\label{pro.bj}
	$\Omega_1=\Omega_2$.
\end{proposition}

\begin{proof}

	If $(E,a,b,c)\in\Omega_1$, then $(E,a,b,c)$ is a solution to L'Huilier's Equation \eqref{eq.lh} by Theorem \ref{th.lh}. So $(E,a,b,c)\in\Omega_2$.
	
	Assume that $(E,a,b,c)\in\Omega_2$. Then $a+b>c, a+c>b,$ and $b+c>a$ by Lemma \ref{lem.ineq}. So 
	there is a rational spherical triangle with side lengths $a,b,c$ and area $E$ by the converse of triangle inequality for spherical triangles. Thus $(E,a,b,c)\in\Omega_1$.
\end{proof}

\begin{emp}\label{em.map}
		We now look at the relation between Equations \eqref{eq.lh} and \eqref{eq.ne4}. Let 
	\begin{equation*}
		\begin{aligned}	
			\Omega_3:=\left\{
			( x_0,x_1, x_2, x_3,x_4) \;\middle|\;
			\begin{aligned}
				& ( x_0,x_1, x_2, x_3,x_4)\text{ is a solution to Equation } \eqref{eq.ne4}\\
				& \text{in } G \text{ such that } 0<x_1\leq x_2\le x_3<x_4<\frac{\pi}{2},\\
				& x_4=x_1+x_2+x_3\text{ and } 0<x_0<\frac{\pi}{2}.
			\end{aligned}
			\right\}.
		\end{aligned}
	\end{equation*}

	Let
	\begin{align*}	
		\phi:\Omega_2
		&\rightarrow  
		\Omega_3\\
		( E,a, b, c)\quad&\mapsto\quad ( \frac{E}{4},\frac{a+b-c}{4}, \frac{a-b+c}{4}, \frac{-a+b+c}{4},\frac{a+b+c}{4}),
		\intertext{and let}
		\psi:\Omega_3
		&\rightarrow  
		\Omega_2\\
		(x_0,x_1,x_2,x_3,x_4)\quad&\mapsto\quad (4x_0,2x_1+2x_2,2x_1+2x_3,2x_2+2x_3).
	\end{align*}
\end{emp}

\begin{proposition}\label{pro.bj2}

The maps $\phi$ and $\psi$ are well defined and are inverse to each other.

\end{proposition}

\begin{proof}
	
	Assume that $(E,a,b,c)\in\Omega_2$. By Lemma \ref{lem.ineq}, we know that $0<\frac{a+b-c}{4}\le\frac{a-b+c}{4}\le \frac{-a+b+c}{4}<\frac{a+b+c}{4}<\frac{\pi}{2}$. It is clear that $\frac{a+b+c}{4}=\frac{a+b-c}{4}+\frac{a-b+c}{4}+\frac{-a+b+c}{4}$ and $0<\frac{E}{4}<\frac{\pi}{2}$. Hence $\phi$ is well defined.

	Assume that $(x_0,x_1,x_2,x_3,x_4)\in\Omega_3$. Since $0<x_1\leq x_2\le x_3<x_4<\frac{\pi}{2}$ and $0<x_0<\frac{\pi}{2}$, it is clear that $0<2x_1+2x_2\le 2x_1+2x_3\le 2x_2+2x_3<2\pi$ and $0<E<2\pi$. Because $0<x_4<\frac{\pi}{2}$ and $x_4=x_1+x_2+x_3$, we know that $0<2x_1+2x_2+2x_3<\pi$. Since $x_1,x_2,x_3>0$, it follows that $0<2x_1+2x_2\le 2x_1+2x_3\le 2x_2+2x_3<\pi$. So $\psi$ is well defined. 
	
	It is easy to verify that $\phi\circ\psi=\operatorname{id}$ and $\psi\circ\phi=\operatorname{id}$. So the claims follow.
\end{proof}

\begin{remark}
	Equation \eqref{eq.ne4} is naturally associated to the following twisted equation:
	
	\begin{equation}\label{eq.ne5}
		(\operatorname{tan}x_0)^2=-(\operatorname{tan}x_1)(\operatorname{tan}x_2)(\operatorname{tan}x_3)(\operatorname{tan}x_4).
	\end{equation}

Note that the tuple $(x_0,x_1,x_2,x_3,x_4)\in\mathbb{R}^5$ is a solution to Equation \eqref{eq.ne4} if and only if $(x_0,\eta_1 x_1,\eta_2 x_2,\eta_3 x_3,\eta_4 x_4)\in\mathbb{R}^5$ is a solution to Equation \eqref{eq.ne5} for all $\eta_1,\eta_2,\eta_3,\eta_4\in\{1,-1\}$ with $\prod\limits_{i=1}^4\eta_i=-1$. This is because $\operatorname{tan}(x)$ is an odd function.

\end{remark}

\section{A conjecture and its application to L'Huilier's Equation}\label{sec.conj}

In this section, we provide several families of solutions in $G$ to Equation \eqref{eq.ne4}. We then introduce Conjecture \ref{cj.fi} on all possible solutions in $G$ to Equation \eqref{eq.ne4}. Proposition \ref{pro.coli} discusses the implication of Conjecture \ref{cj.fi}  on the possible measurements of rational spherical triangles.

\begin{emp}\label{emp.fam}
	Let
	
	\begin{equation*}
		\aligned
	&\overline{\Phi}_{1,1}:=\{(s,s,s,t,\frac{1}{2}-t)\pi\mid s,t\in\mathbb{Q}, \; 0<s< \frac{1}{2}\text{ and } 0<t\le\frac{1}{4}\},\\
		&\overline{\Phi}_{1,2}:=\{(\frac{1}{4},s,\frac{1}{2}-s,t,\frac{1}{2}-t)\pi\mid s,t\in\mathbb{Q}, \;0<s\le t\le \frac{1}{4}\}.\\
		\endaligned
	\end{equation*}

	Let

	\begin{equation*}
		\aligned
		&\overline{\Phi}_{2,1}:=\{(\frac{1}{4},s,\frac{1}{3}-s,\frac{1}{3}+s,\frac{1}{2}-3s)\pi\mid s\in\mathbb{Q}\text{ and } 0<s< \frac{1}{6}\},\\
		&\overline{\Phi}_{2,2}:=\{(\frac{1}{2}-s,\frac{1}{2}-s,\frac{1}{3}-s,\frac{1}{3}+s,\frac{1}{2}-3s)\pi \mid s\in\mathbb{Q}\text{ and } 0<s< \frac{1}{6}\},
		\\
		&\overline{\Phi}_{2,3}:=\{(\frac{1}{6}+s,s,\frac{1}{6}+s,\frac{1}{3}+s,\frac{1}{2}-3s)\pi\mid s\in\mathbb{Q}\text{ and } 0<s< \frac{1}{6}\},
		\\
		&\overline{\Phi}_{2,4}:=\{(\frac{1}{6}-s,s,\frac{1}{3}-s,\frac{1}{6}-s,\frac{1}{2}-3s)\pi\mid s\in\mathbb{Q}\text{ and } 0<s< \frac{1}{6}\},
		\\
		&\overline{\Phi}_{2,5}:=\{(3s,s,\frac{1}{3}-s,\frac{1}{3}+s,3s)\pi\mid s\in\mathbb{Q}\text{ and } 0<s< \frac{1}{6}\}.
		\\
		\endaligned
	\end{equation*}

	Let 
	\begin{equation*}
		\aligned
		&\overline{\Phi}_{3,1}:=\{(\frac{1}{8},\frac{1}{24},\frac{7}{24},s,\frac{1}{2}-s)\pi\mid s\in\mathbb{Q}\text{ and } 0<s\le \frac{1}{4}\},\\
		&\overline{\Phi}_{3,2}:=\{(\frac{3}{8},\frac{5}{24},\frac{11}{24},s,\frac{1}{2}-s)\pi\mid s\in\mathbb{Q}\text{ and } 0<s\le \frac{1}{4}\}.\\
		\endaligned
	\end{equation*}

	Let $S_4$ be the symmetric group on $\{1,2,3,4\}$. Consider the following action of $S_4$ on the set $G^5$. For each $\sigma\in S_4$ and $(x_0,x_1,x_2,x_3,x_4)\in G^5$, let $$\sigma\cdot(x_0,x_1,x_2,x_3,x_4):=(x_0,x_{\sigma(1)},x_{\sigma(2)},x_{\sigma(3)},x_{\sigma(4)}).$$

	Let $I:=\{(1,1),(1,2),(2,1),(2,2),(2,3),(2,4),(2,5),(3,1),(3,2)\}$.
	For each  $(i,j)\in I$, let $\Phi_{i,j}:=S_4\cdot \overline{\Phi}_{i,j}$.  Let
	
	 $$\Phi:=\bigcup\limits_{(i,j)\in I}\Phi_{i,j}.$$

\end{emp}

\begin{conjecture}\label{cj.fi}
		The tuple $(x_0,x_1,x_2,x_3,x_4)$ is a solution to \eqref{eq.ne4} in $G$ with $0<x_i<\frac{\pi}{2}$ for $0\le i\le 4$ if and only if $(x_0,x_1,x_2,x_3,x_4)$ is either in the infinite set $\Phi$ given in \ref{emp.fam} or in the set $\Psi$ consisting of $2928$ elements given in \ref{emp.spo}.
\end{conjecture}

\begin{remark}\label{re.oned}
	It can be verified directly that each element in $\Psi$ and $\Phi$ is a solution to Equation \eqref{eq.ne4}.
\end{remark}

\begin{lemma}\label{le.int}
	The following is true.
	
	\begin{enumerate}
		\item $\Phi_{1,1}\cap\Omega_3=\{(s,s,s,\frac{1}{4}-s,\frac{1}{4}+s)\pi\mid 0<s<\frac{1}{8}\text{ and } s\in\mathbb{Q}\}\cup\{(s,\frac{1}{4}-s,s,s,\frac{1}{4}+s)\pi\mid \frac{1}{8}\le s<\frac{1}{4}\text{ and } s\in\mathbb{Q}\}$.
		
		\item $\Phi_{1,2}\cap\Omega_3=\emptyset$.
		
		\item $\Phi_{2,1}\cap\Omega_3=\{(\frac{1}{4},\frac{1}{8},\frac{1}{8},\frac{5}{24},\frac{11}{24})\pi\}$
		\item $\Phi_{2,i}\cap\Omega_3=\emptyset$ for $2\le i\le 5$.
		\item $\Phi_{3,1}\cap\Omega_3=\{(\frac{1}{8},\frac{1}{24},\frac{1}{12},\frac{7}{24},\frac{5}{12} )\pi\}$.
		\item $\Phi_{3,2}\cap\Omega_3=\emptyset$.
		
	\end{enumerate}
\end{lemma}

\begin{proof}\hfill
		\begin{enumerate}
		\item Assume that $\sigma\cdot(s,s,s,t,\frac{1}{2}-t)\pi\in\Omega_3$ for some $\sigma\in S_4$ and $0<s< \frac{1}{2}$ and $0<t\le\frac{1}{4}$. By the assumption, we know that $\frac{1}{2}-t\ge t$. So  $m:=\operatorname{max}\{s,t,\frac{1}{2}-t\}=s$ or $\frac{1}{2}-t$.
		
		If $\frac{1}{2}-t<s$, then $m=s$. So $s+\frac{1}{2}=s+t+\frac{1}{2}-t=s$, a contradiction.
		
		If $\frac{1}{2}-t\ge s$, then $m=\frac{1}{2}-t$. So $s+s+t=\frac{1}{2}-t$. Thus $t=\frac{1}{4}-s$ and $\frac{1}{2}-t=\frac{1}{4}+s$. Because $t>0$, we know that $0<s<\frac{1}{4}$. So
		$\Phi_{1,1}\cap\Omega_3=\{(s,s,s,\frac{1}{4}-s,\frac{1}{4}+s)\pi\mid 0<s<\frac{1}{8}\text{ and } s\in\mathbb{Q}\}\cup\{(s,\frac{1}{4}-s,s,s,\frac{1}{4}+s)\pi\mid \frac{1}{8}\le s<\frac{1}{4}\text{ and } s\in\mathbb{Q}\}$.

		\item Assume that $\sigma\cdot(\frac{1}{4},s,\frac{1}{2}-s,t,\frac{1}{2}-t)\pi\in\Omega_3$ for some $\sigma\in S_4$ and $0<s\le t\le \frac{1}{4}$. Then $\operatorname{max}\{s,\frac{1}{2}-s,t,\frac{1}{2}-t\}=\frac{1}{2}-s$. Thus $s+t+(\frac{1}{2}-t)=\frac{1}{2}-s$. So $s=0$, a contradiction. Thus $\Phi_{1,2}\cap\Omega_3=\emptyset$.

		\item Assume that $\sigma\cdot(\frac{1}{4},s,\frac{1}{3}-s,\frac{1}{3}+s,\frac{1}{2}-3s)\pi\in\Omega_3$ for some $\sigma\in S_4$ and $0<s<\frac{1}{6}$. Since $0<s<\frac{1}{6}$, we get that $s<\frac{1}{3}-s<\frac{1}{3}+s$. So $m:=\operatorname{max}\{s,\frac{1}{3}-s,\frac{1}{3}+s,\frac{1}{2}-3s\}=\frac{1}{3}+s$ or $\frac{1}{2}-3s$. 
		
		If $s\in(0,\frac{1}{24})$, then $m=\frac{1}{2}-3s$. Thus $s+(\frac{1}{3}-s)+(\frac{1}{3}+s)=\frac{1}{2}-3s$. So $s=-\frac{1}{24}$, a contradiction.
		
		If $s\in[\frac{1}{24},\frac{1}{6})$, then $m=\frac{1}{3}+s$. Thus $s+(\frac{1}{3}-s)+(\frac{1}{2}-3s)=\frac{1}{3}+s$. So $s=\frac{1}{8}$. Therefore, $\Phi_{2,1}\cap\Omega_3=\{(\frac{1}{4},\frac{1}{8},\frac{1}{8},\frac{5}{24},\frac{11}{24})\pi\}$. 
		
		\item Assume that $\sigma\cdot(\frac{1}{2}-s,\frac{1}{2}-s,\frac{1}{3}-s,\frac{1}{3}+s,\frac{1}{2}-3s)\pi\in\Omega_3$ for some $\sigma\in S_4$ and $0<s<\frac{1}{6}$. Because $0<s<\frac{1}{6}$, we get that $\frac{1}{3}-3s<\frac{1}{3}-s<\frac{1}{2}-s$. So $m:=\operatorname{max}\{\frac{1}{2}-s,\frac{1}{3}-s,\frac{1}{3}+s,\frac{1}{3}-3s\}=\frac{1}{2}-s$ or $\frac{1}{3}+s$.

		If $s\in(0,\frac{1}{12})$, then $m=\frac{1}{2}-s$. Thus $\frac{1}{3}-s+\frac{1}{3}+s+\frac{1}{3}-3s=\frac{1}{2}-s$. Then $s=\frac{1}{3}$, a contradiction.
		
		If $s\in[\frac{1}{12},\frac{1}{6})$, then $m=\frac{1}{3}+s$. Thus $\frac{1}{2}-s+\frac{1}{3}-s+\frac{1}{3}-3s=\frac{1}{3}+s$. Then $s=\frac{1}{6}$, a contradiction. 
		
		Therefore $\Phi_{2,2}\cap\Omega_3=\emptyset$. By similar argument, one can show that $\Phi_{2,i}\cap\Omega_3=\emptyset$ for $3\le i\le 5$.

		\item Assume that $\sigma\cdot(\frac{1}{8},\frac{1}{24},\frac{7}{24},s,\frac{1}{2}-s)\pi\in\Omega_3$ for some $\sigma\in S_4$ and $0<s\le\frac{1}{4}$. From $0<s\le\frac{1}{4}$, we know that $\frac{1}{2}-s\ge s$. So
		$m:=\operatorname{max}\{\frac{1}{24},\frac{7}{24},s,\frac{1}{2}-s\}=\frac{7}{24}$ or $\frac{1}{2}-s$.

		If $s\in(0,\frac{5}{24})$, then $m=\frac{1}{2}-s$. It follows that $\frac{1}{24}+\frac{7}{24}+s=\frac{1}{2}-s$. So $s=\frac{1}{12}$. 
		
		If $s\in[\frac{5}{24},\frac{1}{4}]$, then $m=\frac{7}{24}$. It follows that $\frac{1}{24}+s+(\frac{1}{2}-s)=\frac{13}{24}=\frac{7}{24}$, a contradiction.
		
		Therefore, $\Phi_{3,1}\cap\Omega_3=\{(\frac{1}{8},\frac{1}{24},\frac{1}{12},\frac{7}{24},\frac{5}{12} )\pi\}$. 
		
		\item Assume that $\sigma\cdot(\frac{3}{8},\frac{5}{24},\frac{11}{24},s,\frac{1}{2}-s)\pi\in\Omega_3$ for some $\sigma\in S_4$ and $0<s\le\frac{1}{4}$. Then $m:=\operatorname{max}\{\frac{5}{24},\frac{11}{24},s,\frac{1}{2}-s\}=\frac{11}{24}$ or $\frac{1}{2}-s$.

		If $s\in(0,\frac{1}{24})$, then $m=\frac{1}{2}-s$. It follows that $\frac{5}{24}+\frac{11}{24}+s=\frac{1}{2}-s$. Then $s=-\frac{1}{12}$, a contradiction.
		
		If $s\in[\frac{1}{24},\frac{1}{4}]$, then $m=\frac{11}{24}$. It follows that $\frac{5}{24}+s+(\frac{1}{2}-s)=\frac{11}{24}$. This is a contradiction.
		
		Therefore, $\Phi_{3,2}\cap\Omega_3=\emptyset$.
		
	\end{enumerate}
\end{proof}

\begin{lemma}\label{le.int2}
	
Let $\Psi$ be the set given in \ref{emp.spo}. Then $\Psi\cap\Omega_3$ consists of the following six elements.
	
		\begin{align*}
		&(\frac{1}{8},\frac{1}{40},\frac{7}{40}, \frac{9}{40},\frac{17}{40})\pi,&&(\frac{1}{16},\frac{1}{48},\frac{5}{48},\frac{11}{48}, \frac{17}{48})\pi,&&(\frac{5}{16},\frac{5}{48},\frac{7}{48},\frac{11}{48},\frac{23}{48} )\pi,\\
		&(\frac{1}{4},\frac{1}{15},\frac{2}{15},\frac{4}{15}, \frac{7}{15})\pi,&&(\frac{3}{8},\frac{11}{120},\frac{19}{120},\frac{29}{120},\frac{59}{120})\pi,&&
		(\frac{1}{8},\frac{7}{120},\frac{17}{120},\frac{23}{120},\frac{47}{120})\pi.\\
	\end{align*}
\end{lemma}

\begin{proof}

	Let $(x_0,x_1,x_2,x_3,x_4)\in\Psi$. Then $(x_0,x_1,x_2,x_3,x_4)\in \Omega_3$ if and only if (i) $x_1\le x_2\le x_3\le x_4$ and (ii) $x_1+x_2+x_3=x_4$ (see Remark \ref{re.oned}). By checking each element in $\Psi$ given in \ref{emp.spo}, the six elements listed above are the only ones that satisfy the conditions (i) and (ii).
\end{proof}

\begin{proposition}\label{pro.coli}
	
	If Conjecture \ref{cj.fi} is true, then Conjecture \ref{cj.lc} is true.
\end{proposition}

\begin{proof}

	By Proposition \ref{pro.bj}, Conjecture \ref{cj.lc} is equivalent to $\Omega_2=\Lambda_1\cup\Lambda_2$. By \ref{emp.cl} and Proposition \ref{pro.bj}, we know that $\Lambda_1\cup\Lambda_2\subset\Omega_2$. Assume that Conjecture \ref{cj.fi} is true. We now show that $\Omega_2\subset \Lambda_1\cup\Lambda_2$. Assume that $(E,a,b,c)\in\Omega_2$.

	By Proposition \ref{pro.bj2}, there exists a unique $x=(x_0,x_1,x_2,x_3,x_4)\in\Omega_3$ such that $\psi(x)=(E,a,b,c)$.
	Because we assumed that Conjecture \ref{cj.fi} is true, it follows that $x\in(\Phi\cup\Psi)\cap\Omega_3$. By Lemma \ref{le.int} and Lemma \ref{le.int2}, the following are the possible cases.

	\begin{enumerate}
		\item $x\in\Phi_{1,1}\cap\Omega_3$. If $x=(s,s,s,\frac{1}{4}-s,\frac{1}{4}+s)\pi$ for some $0<s<\frac{1}{8}$ and $s\in\mathbb{Q}$, then $\psi(x)=(4s,4s,\frac{1}{2},\frac{1}{2})\pi\in \Lambda_2$. Similarly, if $x=(s,\frac{1}{4}-s,s,s,\frac{1}{4}+s)\pi$ for some $\frac{1}{8}\le s<\frac{1}{4}$ and $s\in\mathbb{Q}$, then $\psi(x)=(4s,\frac{1}{2},\frac{1}{2},4s)\pi\in \Lambda_2$.
		\item $x\in\Phi_{2,1}\cap\Omega_3$. Then $x=(\frac{1}{4},\frac{1}{8},\frac{1}{8},\frac{5}{24},\frac{11}{24})\pi$. Thus $(E,a,b,c)=\psi(x)=(1,\frac{1}{2}, \frac{2}{3}, \frac{2}{3})\pi\in \Lambda_1$.
		\item $x\in\Phi_{3,1}\cap\Omega_3$. Then $x=(\frac{1}{8},\frac{1}{24},\frac{1}{12},\frac{7}{24},\frac{5}{12} )\pi$. Thus $\psi(x)=(\frac{1}{2},\frac{1}{4}, \frac{2}{3}, \frac{3}{4})\pi\in \Lambda_2$.
		\item $x\in\Psi\cap\Omega_3$. Then $x$ is one of the six elements given in Lemma \ref{le.int2}. The following computation shows that if $x$ is one of the six elements, then $\psi(x)\in \Lambda_2$.
		
		\small
		
		\begin{align*}
			\psi((\frac{1}{8},\frac{1}{40},\frac{7}{40}, \frac{9}{40},\frac{17}{40})\pi)&=(\frac{1}{2},\frac{2}{5},\frac{1}{2},\frac{4}{5})\pi,&&&\psi((\frac{1}{16},\frac{1}{48},\frac{5}{48},\frac{11}{48}, \frac{17}{48})\pi)&=(\frac{1}{4},\frac{1}{4}, \frac{1}{2}, \frac{2}{3})\pi,\\
			\psi((\frac{5}{16},\frac{5}{48},\frac{7}{48},\frac{11}{48},\frac{23}{48} )\pi)&=(\frac{5}{4},\frac{1}{2}, \frac{2}{3}, \frac{3}{4})\pi,&&&\psi((\frac{1}{4},\frac{1}{15},\frac{2}{15},\frac{4}{15}, \frac{7}{15})\pi)&=(1,\frac{2}{5}, \frac{2}{3}, \frac{4}{5})\pi,\\
			\psi((\frac{3}{8},\frac{11}{120},\frac{19}{120},\frac{29}{120},\frac{59}{120})\pi)&=(\frac{3}{2},\frac{1}{2}, \frac{2}{3}, \frac{4}{5})\pi,&&&\psi((\frac{1}{8},\frac{7}{120},\frac{17}{120},\frac{23}{120},\frac{47}{120})\pi)&=(\frac{1}{2},\frac{2}{5}, \frac{1}{2},\frac{2}{3})\pi.\\
		\end{align*}
		
		\normalsize

	\end{enumerate}
\end{proof}

\begin{emp}\label{em.inv}
	Consider the following action of $\mathbb{Z}/2\mathbb{Z}$ on the set $G^5$. Let $\theta\in\mathbb{Z}/2\mathbb{Z}$ be the non-identity element. For $(x_0,\dots,x_4)\in G^5$, let 
	
	$$\theta\cdot(x_0,\dots,x_4)=(\frac{\pi}{2}-x_0,\dots,\frac{\pi}{2}-x_4).$$ 
	
	To be used later in \ref{emp.spo}, we introduce the following group action. Since the $\mathbb{Z}/2\mathbb{Z}$ action commutes with the $S_4$ action given in \ref{emp.fam}, there is an induced action of $\mathbb{Z}/2\mathbb{Z}\times S_4$ on $G^5$. Let $\sigma\in S_4$ and $\theta\in\mathbb{Z}/2\mathbb{Z}$. Define $(\sigma,\theta)\cdot(x_0,\dots,x_4):=\sigma\cdot(\theta\cdot(x_0,\dots,x_4))$.
\end{emp}

\begin{proposition}
	$\theta:G^5\rightarrow G^5$ restricts to an involution on $\Phi$. More precisely,
	\begin{enumerate}[label=(\roman*)]
		\item\label{it.inv1} $\theta^2=\operatorname{id}$.
		\item\label{it.inv2} The following subsets are closed under the above $\mathbb{Z}/2\mathbb{Z}$ action: $\Phi_{1,2}$, $\Phi_{1,1}, \Phi_{2,1}, \Phi_{2,3}$ and  $\Phi_{2,5}$. 
		\item\label{it.inv3} $\theta\cdot \Phi_{2,2}=\Phi_{2,4}$.
		\item\label{it.inv4} $\theta\cdot \Phi_{3,1}=\Phi_{3,2}$.
	\end{enumerate}
	
\end{proposition}

\begin{proof}
	\ref{it.inv1} is immediate. For the remaining statements, we only verify \ref{it.inv3}. The others can be checked similarly. Let $(\frac{1}{2}-s,\frac{1}{2}-s,\frac{1}{3}-s,\frac{1}{3}+s,\frac{1}{2}-3s)\pi\in \Phi_{2,2}$ with $s\in\mathbb{Q}$ and $0<s<\frac{1}{6}$. Let $t:=\frac{1}{6}-s$. Then $0<t<\frac{1}{6}$ and $\theta\cdot(\frac{1}{2}-s,\frac{1}{2}-s,\frac{1}{3}-s,\frac{1}{3}+s,\frac{1}{2}-3s)\pi=(\frac{1}{6}+t,\frac{1}{6}+t,\frac{1}{3}-t,t,\frac{1}{2}-3t)\pi\in \Phi_{2,4}$. Therefore $\theta\cdot \Phi_{2,2}\subset \Phi_{2,4}$. Similarly, one gets $\theta\cdot \Phi_{2,4}\subset \Phi_{2,2}$. So $\theta^2\cdot \Phi_{2,4}= \Phi_{2,4}\subset\theta \cdot \Phi_{2,2}$. Hence $\theta\cdot \Phi_{2,2}=\Phi_{2,4}$.
\end{proof}

\section{The Sporadic solutions to Equation \eqref{eq.ne4}}\label{sec.spo}

	In section \ref{sec.conj}, we introduced certain families of solutions in $G$ to Equation \eqref{eq.ne4}. In the following, we list some sporadic solutions in $G$ to Equation \eqref{eq.ne4} that do not fall within these families. We also provide a computational result that supports Conjecture \ref{cj.fi}.

\begin{emp}\label{emp.spo}

	Let $\overline{\Psi}$ be the set consisting of the following $61$ tuples in $G^5$. We define the $\operatorname{lcm}$ of a tuple to be the least common multiple of the denominators of the entries in the tuple. Notice that each element $(x_0,x_1,x_2,x_3,x_4)\in\overline{\Psi}$ satisfies one of the following two conditions: either (i) $x_0<\frac{\pi}{4}$ and $x_1\le x_2\le x_3\le x_4$, or (ii) $x_0=\frac{\pi}{4},\; x_1\le x_2\le x_3\le x_4$ and $x_1+x_4<\frac{\pi}{2}$.
\begin{align*}
		\intertext{$\operatorname{lcm}=30$}
	&(\frac{1}{30},\frac{1}{30},\frac{1}{15},\frac{2}{15},\frac{4}{15})\pi,&&(\frac{1}{15},\frac{1}{30},\frac{1}{15},\frac{7}{30},\frac{11}{30})\pi,&&(\frac{2}{15},\frac{1}{30},\frac{2}{15},\frac{7}{30},\frac{13}{30})\pi,\\
	&(\frac{7}{30},\frac{1}{15},\frac{2}{15},\frac{7}{30},\frac{7}{15})\pi,\\
	\intertext{$\operatorname{lcm}=40$}
	&(\frac{1}{8},\frac{1}{40},\frac{7}{40},\frac{9}{40},\frac{17}{40})\pi,&&\\
		\intertext{$\operatorname{lcm}=48$}
	&(\frac{1}{16},\frac{1}{48},\frac{5}{48},\frac{11}{48},\frac{17}{48})\pi,&& (\frac{3}{16},\frac{1}{48},\frac{13}{48},\frac{17}{48},\frac{19}{48})\pi,&&\\
		\intertext{$\operatorname{lcm}=60$}
		& (\frac{1}{60},\frac{1}{60}, \frac{1}{20},\frac{1}{12},\frac{17}{60})\pi,
	&&(\frac{1}{60},\frac{1}{60}, \frac{1}{12},\frac{7}{60},\frac{3}{20})\pi,
	&&(\frac{1}{20},\frac{1}{60}, \frac{1}{20},\frac{13}{60},\frac{5}{12})\pi,\\
	& (\frac{1}{20},\frac{1}{60}, \frac{7}{60},\frac{13}{60},\frac{19}{60})\pi,
	&&(\frac{1}{20},\frac{1}{20}, \frac{1}{12},\frac{7}{60},\frac{19}{60})\pi,
	&&(\frac{1}{12},\frac{1}{60}, \frac{1}{12},\frac{13}{60},\frac{9}{20})\pi,\\
	&(\frac{1}{12},\frac{1}{60}, \frac{1}{12},\frac{7}{20},\frac{23}{60})\pi,
	&&(\frac{1}{12},\frac{1}{60}, \frac{11}{60},\frac{13}{60},\frac{23}{60})\pi,
	&&(\frac{1}{12},\frac{1}{20}, \frac{1}{12},\frac{11}{60},\frac{23}{60})\pi,\\
	&(\frac{1}{12},\frac{1}{12}, \frac{3}{20},\frac{11}{60},\frac{13}{60})\pi,
	&&(\frac{7}{60},\frac{1}{60}, \frac{7}{60},\frac{7}{20},\frac{5}{12})\pi,
	&&(\frac{7}{60},\frac{1}{20}, \frac{7}{60},\frac{11}{60},\frac{5}{12})\pi,\\
	&(\frac{3}{20},\frac{1}{60}, \frac{3}{20},\frac{23}{60},\frac{5}{12})\pi,
	&&(\frac{3}{20},\frac{1}{60}, \frac{17}{60},\frac{19}{60},\frac{23}{60})\pi,
	&&(\frac{3}{20},\frac{1}{12}, \frac{3}{20},\frac{17}{60},\frac{19}{60})\pi,\\
	&(\frac{11}{60},\frac{1}{12}, \frac{7}{60},\frac{11}{60},\frac{9}{20})\pi,
	&&(\frac{11}{60},\frac{1}{12}, \frac{11}{60},\frac{17}{60},\frac{7}{20})\pi,
	&&(\frac{13}{60},\frac{1}{20}, \frac{1}{12},\frac{13}{60},\frac{29}{60})\pi,\\
	&(\frac{13}{60},\frac{1}{12}, \frac{13}{60},\frac{19}{60},\frac{7}{20})\pi,
	&&(\frac{1}{4},\frac{1}{60},\frac{13}{60},\frac{5}{12},\frac{9}{20})\pi,
	&&(\frac{1}{4},\frac{1}{60},\frac{7}{20},\frac{23}{60},\frac{5}{12})\pi,\\ 
	&(\frac{1}{4},\frac{1}{30},\frac{7}{30},\frac{11}{30},\frac{13}{30})\pi, 
	&&(\frac{1}{4},\frac{1}{20},\frac{11}{60},\frac{23}{60},\frac{5}{12})\pi, 
	&&(\frac{1}{4},\frac{1}{12},\frac{17}{60},\frac{19}{60},\frac{7}{20})\pi,\\
		\intertext{$\operatorname{lcm}=72$}
	&(\frac{1}{8},\frac{1}{72},\frac{7}{72},\frac{23}{72},\frac{25}{72})\pi,&& (\frac{1}{8},\frac{1}{24},\frac{7}{72},\frac{17}{72},\frac{31}{72})\pi,&&\\
	\intertext{$\operatorname{lcm}=84$}
	&(\frac{1}{84},\frac{1}{84}, \frac{5}{84},\frac{1}{12},\frac{17}{84}),&&(\frac{5}{84},\frac{1}{84}, \frac{5}{84}, \frac{25}{84},\frac{5}{12}),&&(\frac{1}{12},\frac{1}{84},\frac{1}{12},\frac{25}{84},\frac{37}{84} )\\&(\frac{1}{12},\frac{1}{12},\frac{11}{84},\frac{13}{84},\frac{23}{12})\pi,&&(\frac{11}{84},\frac{1}{12},\frac{11}{84},\frac{19}{84},\frac{29}{84})\pi,&&(\frac{13}{84},\frac{1}{12},\frac{13}{84},\frac{19}{84},\frac{31}{84})\pi,\\
	&(\frac{17}{84},\frac{1}{84},\frac{17}{84},\frac{5}{12},\frac{37}{84})\pi,&&(\frac{19}{84},\frac{11}{84},\frac{13}{84},\frac{19}{84},\frac{5}{12})\pi,&&(\frac{1}{4},\frac{1}{84},\frac{25}{84},\frac{5}{12},\frac{37}{84})\pi,\\
	&(\frac{1}{4},\frac{1}{12},\frac{19}{84},\frac{29}{84},\frac{31}{84})\pi,\\
	\intertext{$\operatorname{lcm}=120$}
	&(\frac{1}{120},\frac{1}{120},\frac{7}{120},\frac{11}{120},\frac{17}{120})\pi,&&(\frac{7}{120},\frac{1}{120},\frac{7}{120},\frac{43}{120},\frac{49}{120})\pi,
	&&(\frac{11}{120},\frac{1}{120},\frac{11}{120},\frac{43}{120},\frac{53}{120})\pi,\\&(\frac{13}{120},\frac{13}{120},\frac{19}{120},\frac{23}{120},\frac{29}{120})\pi,
	&&(\frac{1}{8},\frac{1}{120},\frac{23}{120},\frac{47}{120},\frac{49}{120})\pi,&& (\frac{1}{8},\frac{1}{120},\frac{9}{40},\frac{41}{120},\frac{17}{40})\pi,\\ &(\frac{1}{8},\frac{1}{120},\frac{31}{120},\frac{41}{120},\frac{49}{120})\pi,&&
	(\frac{1}{8},\frac{1}{40},\frac{7}{120},\frac{47}{120},\frac{17}{40})\pi,
	&&(\frac{1}{8},\frac{1}{40},\frac{7}{40},\frac{31}{120},\frac{49}{120})\pi,\\&(\frac{1}{8},\frac{7}{120},\frac{17}{120},\frac{23}{120},\frac{47}{120})\pi,
	&&(\frac{1}{8},\frac{7}{120},\frac{17}{120},\frac{31}{120},\frac{41}{120})\pi,&& (\frac{1}{8},\frac{17}{120},\frac{7}{40},\frac{23}{120},\frac{9}{40})\pi,\\
	&(\frac{17}{120},\frac{1}{120},\frac{17}{120},\frac{49}{120},\frac{53}{120})\pi,&&(\frac{19}{120},\frac{13}{120},\frac{19}{120},\frac{31}{120},\frac{37}{120})\pi,
	&&(\frac{23}{120},\frac{13}{120},\frac{23}{120},\frac{31}{120},\frac{41}{120})\pi,\\&(\frac{29}{120},\frac{13}{120},\frac{29}{120},\frac{37}{120},\frac{41}{120})\pi,
	&&(\frac{1}{4},\frac{1}{120},\frac{43}{120},\frac{49}{120},\frac{53}{120})\pi,&&(\frac{1}{4},\frac{13}{120},\frac{31}{120},\frac{37}{120},\frac{41}{120})\pi.\\
\end{align*}

\normalsize

	Recall the $\mathbb{Z}/2\mathbb{Z}\times S_4$ action on $G^5$ discussed in section \ref{em.inv}. Let $\Psi:=(\mathbb{Z}/2\mathbb{Z}\times S_4)\cdot\overline{\Psi}$. We have that $|\overline{\Psi}|=61$ and $|\Psi|=48\cdot61=2928$.
	
	\begin{remark}
		Each orbit of the $\mathbb{Z}/2\mathbb{Z}\times S_4$ action on $\Psi$ has size $48$. The set $\overline{\Psi}$ consists of a representative $(x_0,x_1,x_2,x_3,x_4)$ of each orbit satisfying one of the two conditions stated in \ref{emp.spo}.
	\end{remark}

\end{emp}

	Fix a positive integer $D$ and define the set
\begin{equation*}
	\aligned
	L_D:=&\{(x_0,x_1,x_2,x_3,x_4)\in G^5\mid 0<x_i<\frac{\pi}{2}\text{ for each } 0\le i\le 4\text{ and }\\ &\operatorname{lcm}(\operatorname{den}(x_0),\dots,\operatorname{den}(x_4))\le D\}.
	\endaligned
\end{equation*}

It is clear that $|L_D| <\infty$.

\begin{proposition}\label{le.comp}
	The tuple $(x_0,x_1,x_2,x_3,x_4)$ is a solution to Equation \eqref{eq.ne4} in $L_{300}$ if and only if $(x_0,x_1,x_2,x_3,x_4)$ is either in the set $\Phi\cap L_{300}$ given in \ref{emp.fam} or in the set $\Psi$ consisting of $2928$ elements given in \ref{emp.spo}.
\end{proposition}

\begin{proof}
	The verification of the statement requires only a finite number of computations. We did the calculation with Magma \cite{Magma}.
\end{proof}

	\section{Basic Formulas}\label{sec.form}

Let $n\in\mathbb{N}_{\ge 2}$ and $\zeta_n:=e^{\frac{2\pi i}{n}}\in\mathbb{C}^*$. Let $a\in\mathbb{Z}$. In this section, we recall some basic relations satisfied by elements of the form $1-\zeta_n^a$. We then represent tangent of a rational multiples of $\pi$ in elements of this form in Proposition \ref{pro.tar}.

\begin{emp}

	Fix $n\in\mathbb{N}_{\ge 2}$. Let $\zeta_n:=e^{\frac{2\pi i}{n}}\in\mathbb{C}^*$. For $a\in\mathbb{Z}$, let 
	$$v(n,a):=1-\zeta_n^a.$$ 
	More generally, for $a\in\mathbb{Q}$ written in reduced form as $a=\frac{a'}{a''}$ such that $\operatorname{gcd}(n,a'')=1$, let $b$ be the unique integer such that $a''b\equiv 1\operatorname{mod} n$ and $1\le b< n.$ Let
	$$v(n,a):=v(n,a'b).$$

	Recall the following two basic relations. Let $a\in\mathbb{Z}$ and $n\in\mathbb{N}_{\ge 2}$. Then

	\begin{equation}\label{eq.sym}
		v(n,-a)=-\zeta_n^{-a}v(n,a).
	\end{equation}
	
	Let $a\in\mathbb{Z}$ and $m,n\in\mathbb{N}_{\ge 2}$ with $m\mid n$. Then

	\begin{equation}\label{eq.nor}
	v(m,a)=\prod_{j=0}^{n/m-1}v(n,a+mj).  
	\end{equation}
	
	See e.g. \cite{Wash} p. 150.

\end{emp}

\begin{proposition}\label{pro.tar}
	Let $n\in\mathbb{N}_{\ge 2}$. Let $a\in\mathbb{Z}$ be such that $\operatorname{gcd}(n,a)=1$. Then the following equations hold in $\mathbb{C}^*$:
	\begin{enumerate}

		\item If $n$ is odd, then
		\[\operatorname{tan}\frac{a}{n}\pi=iv(n,a)^2v(n,2a)^{-1}.\]
		
		\item If $n$ and $a$ are odd, then
		\[\operatorname{tan}\frac{a}{2n}\pi=iv(n,a)v(n,2^{-1}a)^{-2}.\]
		
		\item If $n$ and $a$ are odd, then
		\[\operatorname{tan}\frac{a}{4n}\pi=iv(4n,a)^2v(n,a)v(n,2^{-1}a)^{-1}.\]

		\item If $8\mid n$, then 
		\[\operatorname{tan}\frac{a}{n}\pi=iv(n,a)^2v(\frac{n}{2},a)^{-1}.\]

	\end{enumerate}
\end{proposition}

\begin{proof}\hfill

	\begin{enumerate}
		\item 	Assume that $n$ is odd. Then
		\begin{equation*}\label{eq.1tan}
			\aligned
			\operatorname{tan}\frac{a}{n}\pi=&i\frac{1-e^{\frac{a}{n}2\pi i}}{1-(-e^{\frac{a}{n}2\pi i})}\\
			=& i(1-e^{\frac{a}{n}2\pi i})(1-e^{\frac{2a+n}{2n}2\pi i})^{-1}\\
			=&i(1-e^{\frac{a}{n}2\pi i})^2(1-e^{\frac{2a}{n}2\pi i})^{-1}\\
			=& iv(n,a)^2v(n,2a)^{-1}.
			\endaligned
		\end{equation*}
		
		\item 	Assume that both $a$ and $n$ are odd. Then
		\begin{equation*}\label{eq.2tan}
			\aligned
			\operatorname{tan}\frac{a}{2n}\pi=&i\frac{1-e^{\frac{a}{2n}2\pi i}}{1-(-e^{\frac{a}{2n}2\pi i})}\\
			=& i(1-e^{\frac{a}{2n}2\pi i})(1-e^{\frac{a+n}{2n}2\pi i})^{-1}\\
			=& iv(2n,a)v(n,\frac{a+n}{2})^{-1}\\
			=& iv(n,a)v(n,2^{-1}a)^{-1}v(n,2^{-1}a)^{-1}\\
			=&iv(n,2^{-1}a)^{-2}v(n,a).\\
			\endaligned
		\end{equation*}

		\item 	Assume that both $a$ and $n$ are odd. Then
		\begin{equation*}
			\aligned
			\operatorname{tan}\frac{a}{4n}\pi=&i\frac{1-e^{\frac{a}{4n}2\pi i}}{1-(-e^{\frac{a}{4n}2\pi i})}\\
			=& i(1-e^{\frac{a}{4n}2\pi i})(1-e^{\frac{a+2n}{4n}2\pi i})^{-1}\\
			=&iv(4n,a)v(4n,a+2n)^{-1}\\
			=&iv(4n,a)^2v(n,a)v(n,2^{-1}a)^{-1}.\\
			\endaligned
		\end{equation*}

		\item 	Assume that $8\mid n$. Then
		
		\begin{equation*}
			\aligned
			\operatorname{tan}\frac{a}{n}\pi=&i\frac{1-e^{\frac{a}{n}2\pi i}}{1-(-e^{\frac{a}{n}2\pi i})}\\
			=& i(1-e^{\frac{a}{n}2\pi i})(1-e^{\frac{a+\frac{n}{2}}{n}2\pi i})^{-1}\\
			=&iv(n,a)v(n,a+\frac{n}{2})^{-1}\\
			=& iv(n,a)^2v(\frac{n}{2},a)^{-1}.
			\endaligned
		\end{equation*}

	\end{enumerate}
\end{proof}

	\section{The Basis}\label{sec.bas}

	 Fix $n\in\mathbb{N}_{\ge 2}$. In this section, we will consider elements of form $v(n,a)$ discussed in section \ref{sec.form} in certain finitely generated free abelian groups.
	 Furthermore, we recall in Theorem \ref{th.ba2} a basis for these free abelian groups constructed by Conrad (\cite{Con2}).

	\begin{emp}{\bf The cyclotomic numbers.}\label{emp.grp}
		For $n\in\mathbb{N}_{\ge 2}$ and $a\in\mathbb{Z}$, recall that
		$$v(n,a):=1-\zeta_n^a.$$ 
		
		Let $T$ be the torsion subgroup of $\mathbb{C}^*$. For $n\in\mathbb{N}_{\ge 2}$, let $Y^n$ be the subgroup of $\mathbb{C}^*$ generated by the $n-1$ elements in the set $\{v(n,a)\in\mathbb{C}^*\mid 1\le a\le n-1\text{ and }a\in\mathbb{N}\}$. Let $T^n$ be the torsion subgroup of $Y^n$. Notice that there is an embedding 
		\[Y^n/T^n \longhookrightarrow \mathbb{C}^*/T.\]
		We denote the image of this embedding by $X^n$.

		It is clear that $Y^n/T^n$, and therefore $X^n$, is a free abelian group of finite rank. An element in the group $X^n$ is called a {\it cyclotomic number}. By abuse of notation, we will again denote the class of $1-\zeta_n^a$ in $Y^n/T^n$ and the image of this class under the above embedding by $v(n,a)$. Furthermore, if $a\in\mathbb{Z}$ and $\operatorname{gcd}(n,a)=1$, we call an element of the form $v(n,a)\in\mathbb{C}^*/T$ a {\it cyclotomic number of level n}. 
		
		Fix $n\in\mathbb{N}_{\ge 2}$. We now define a subgroup $Z^n$ of $\mathbb{C}^*/T$ and a quotient group $\widehat{X^n}$ of $X^n$. If $n$ is not prime, 
		let $Z^n$ be the subgroup of $\mathbb{C}^*/T$ generated by the elements in the set $\bigcup\limits_{\substack{d\mid n, \\ d\ge 2, \,d\ne n}}X^d$. If $n$ is prime, let $Z^n$ be the trivial subgroup $\{1\}$ in $\mathbb{C}^*/T$. Let $\widehat{X^n}:=X^n/Z^n$. An element in the group $\widehat{X^n}$ is called a {\it relative cyclotomic number of level $n$}.

	\end{emp}

\begin{corollary}\label{cor.tarq}
	Let $n\in\mathbb{N}_{\ge 2}$. Let $a\in\mathbb{Z}$ be such that $\operatorname{gcd}(n,a)=1$. Then,
	\begin{enumerate}

		\item If $n$ is odd, then \[\operatorname{tan}\frac{a}{n}\pi=v(n,a)^2v(n,2a)^{-1}\]
		in $\widehat{X^n}$.
		
		\item If $n$ and $a$ are odd, then
		\[\operatorname{tan}\frac{a}{2n}\pi=v(n,a)v(n,2^{-1}a)^{-2}\]
		in $\widehat{X^n}$.

		\item If $4\mid n$, then 
		\[\operatorname{tan}\frac{a}{n}\pi=v(n,a)^2\]
		in $\widehat{X^{n}}$.

	\end{enumerate}
\end{corollary}

\begin{proof}
	The formulas follow from Proposition \ref{pro.tar}.
\end{proof}

\begin{remark}\label{rk.rep}
		Let $g\in G$ with $\operatorname{den}(g)>2$. Then the class of $\operatorname{tan}g$ in $\mathbb{C}^*/T$ belongs to $X^{\operatorname{den}(g)}$, and consequently belongs to $\widehat{X^{\operatorname{den}(g)}}$. Assume that $(x_0,x_1,x_2,x_3,x_4)\in G^5$ is a solution to either Equation \eqref{eq.ne4} or Equation \eqref{eq.ne5}. Then we get that $	(\operatorname{tan}x_0)^2=(\operatorname{tan}x_1)(\operatorname{tan}x_2)(\operatorname{tan}x_3)(\operatorname{tan}x_4),$ 
		or $	(\operatorname{tan}x_0)^2=-(\operatorname{tan}x_1)(\operatorname{tan}x_2)(\operatorname{tan}x_3)(\operatorname{tan}x_4)$. Let $n:=\operatorname{lcm}(\operatorname{den}(x_0),\dots,\operatorname{den}(x_4))$. Then both sides of either equality can be viewed as elements in $X^n$ and $\widehat{X^n}$. 
		
\end{remark}

\begin{lemma}
	Suppose that $x_i\in G$ for $0\le i\le 4$. Assume that $$(\operatorname{tan}x_0)^2=(\operatorname{tan}x_1)(\operatorname{tan}x_2)(\operatorname{tan}x_3)(\operatorname{tan}x_4)$$ in $X^n$ where $n:=\operatorname{lcm}(\operatorname{den}(x_0),\dots,\operatorname{den}(x_4))$. Then $(x_0,x_1,x_2,x_3,x_4)$ is a solution to either Equation \eqref{eq.ne4} or Equation \eqref{eq.ne5}.
\end{lemma}
	
\begin{proof}
	By the assumption, we get that $(\operatorname{tan}x_0)^2=\zeta(\operatorname{tan}x_1)(\operatorname{tan}x_2)(\operatorname{tan}x_3)(\operatorname{tan}x_4)$ where $\zeta$ is some root of unity. Since $\operatorname{tan}x_i\in\mathbb{R}$ for $0\le i\le 4$, we know that $\zeta\in\mathbb{R}$. Hence $\zeta=1$ or $-1$. The claim follows.
\end{proof}

	\begin{proposition}[\cite{Con1}, 2.3.6, see also \cite{Con2} Lemma 4.4]
		Assume that $n\in\mathbb{N}_{\ge 2}$ and $n\ne 4$. Then $\widehat{X^n}$ is a free abelian group of finite rank.
	\end{proposition}
	
	In \cite{Con1}, Conrad constructed a basis of this abelian group which we recall here. To state the basis, it is convenient to use the following notation.

	\begin{emp}\label{ss.rf}

		For $n\in\mathbb{N}_{\ge 2}$, let $n=p_1^{e_1}\cdots p_\ell^{e_\ell}$ be the prime factorization of $n$ with $p_1<p_2<\dots<p_\ell$ prime and $e_1,\dots,e_\ell\ge 1$. Let $a\in\mathbb{Z}$ be such that $n\nmid a$. We define the following numbers associated with $a$ with respect to $n$. 
		
		\begin{enumerate}
			\item If $e_i=1,$ let $\overline{a}_i$ be the integer satisfying $\overline{a}_i\equiv a \operatorname{mod} p_i$ and $0\le a_i\le p_i-1$.
			\item If $e_i>1$, let $\widehat{a}_i$ be the integer satisfying $\widehat{a}_i\equiv a \operatorname{mod} p_i^{e_i-1}$ and $0 \le \widehat{a}_i\le p_i^{e_i-1}-1$. Let $\widetilde{a}_i$ be the integer satisfying $\widetilde{a}_i\equiv a \operatorname{mod} p_i^{e_i}$ and $0\le \widetilde{a}_i\le p_i^{e_i}-1$. Let $\overline{a}_i:=\frac{\widetilde{a}_i-\widehat{a}_i}{{p_i}^{e_i-1}}$. Note that $\overline{a}_i\in\mathbb{Z}_{\ge 0}$, $0\le \overline{a}_i\le p_i-1$ and $\widetilde{a}_i=\overline{a}_i p_i^{e_i-1}+\widehat{a}_i$.
		\end{enumerate}

		We associate to $v(n,a)\in\mathbb{C}$ the following $\ell$-tuple formal symbol: $(a_1,\dots,a_\ell)_n$, where $a_i=\overline{a}_i$ if $e_i=1$, and $a_i$ is the pair $(\overline{a}_i,\widehat{a}_i)$ if $e_i\ge 2$, for $1\le i\le \ell$. This association is well defined because any $b\in\mathbb{Z}$ with $\operatorname{gcd}(n,b)=1$ such that $v(n,b)=v(n,a)$ in $\mathbb{C}$ satisfies $b=a+tn$ for some $t\in\mathbb{Z}$. Then $(b_1,\dots,b_\ell)_n=(a_1,\dots,a_\ell)_n$.  We call the symbol $(a_1,\dots,a_\ell)_n$ the {\it residue form} of $v(n,a)$. The following remark shows that the elements $n$ and $a_i$ with $1\le i\le \ell$ uniquely determine $v(n,a)\in\mathbb{C}$.

	\end{emp}

	\begin{lemma}\label{le.rf}
		Let $n\in\mathbb{N}_{\ge 2}$ and let $n=p_1^{e_1}\cdots p_\ell^{e_\ell}$ be the prime factorization of $n$ with $p_1<p_2<\dots<p_\ell$ and $e_1,\dots,e_\ell\ge 1$. For each $1\le i\le \ell$, if $e_i=1,$ let $b_i\in\mathbb{Z}$ be such that $0\le b_i\le p_i-1$. If $e_i>1$, let $b_i=(\overline{b}_i,\widehat{b}_i)\in\mathbb{Z}^2$ be such that $0\le \overline{b}_i\le p_i-1$ and $0 \le \widehat{b}_i\le p_i^{e_i-1}-1$. Then there exists a unique $a\in\mathbb{Z}$ with $0\le a\le n-1$, such that $v(n,a)\in\mathbb{C}$ has the residue form $(a_1,\dots,a_\ell)_n$ where $a_i=b_i$ for $1\le i\le \ell$.
	\end{lemma}
	
	\begin{proof}
		This follows from the Chinese Remainder Theorem.
	\end{proof}

	\begin{remark}\label{re.rf}
		Let $b_1,b_2,\dots,b_\ell$ be as in Lemma \ref{le.rf}. In the following, we will use the notation $(b_1,b_2,\dots,b_\ell)_n$ to refer the corresponding element $v(n,a)\in\mathbb{C}$ as given in Lemma \ref{le.rf}. 
	\end{remark}

		\begin{emp}\label{em.norm2}
			
			Let $n\in\mathbb{N}_{\ge 4}$ be non-prime and let $a\in\mathbb{Z}$ be such that $\operatorname{gcd}(n,a)=1$. Factor $n=p_1^{e_1}\cdots p_\ell^{e_\ell}$ with $p_1<p_2<\dots<p_\ell$ and $e_1,\dots,e_\ell\ge 1$. Let $(a_1,\dots,a_\ell)_{n}$ be the residue form of $v(n,a)$. Fix $1\le r\le\ell$.

			If $e_r=1$, let 
			\begin{equation*}
				\begin{aligned}
					\Gamma_r:=\left\{(b_1,\dots,b_\ell)_n\in X^n \;\middle|\;\begin{aligned}
						& 1\le b_r\le p_r-1 \text{ and } b_r\ne a_r;\\
						&\text{if } 1\le s\le\ell \text{ and } s\ne r, \text{ then }b_s=a_s
					\end{aligned}
					\right\}.
				\end{aligned}
			\end{equation*}

			If $e_r\ge 2$, let 
			\begin{equation*}
				\begin{aligned}
					\Gamma_r:=\left\{(b_1,\dots,b_\ell)_n\in X^n \;\middle|\;\begin{aligned}
						&b_r=(\overline{b}_r,\widehat{b}_r)\in\mathbb{Z}^2\text{ where } 0\le \overline{b}_r\le p_r-1, \overline{b}_r\ne  \overline{a}_r\\
						&\text{and }\widehat{b}_r=\widehat{a}_r; \text{ if } 1\le s\le\ell\text{ and } s\ne r,\text{ then } b_s=a_s\\ 
					\end{aligned}
					\right\}.
				\end{aligned}
			\end{equation*}

		\end{emp}

		\begin{lemma}\label{le.nor1}
			
			Keep the assumption in \ref{em.norm2}. Then
			
			$$v(n,a)=\prod\limits_{v\in \Gamma_r} v^{-1},$$
			where both sides are viewed as elements in $\widehat{X^{n}}$ under the quotient map.
		\end{lemma}
		
		\begin{proof}
			This follows from the Formula \eqref{eq.nor}.
		\end{proof}

	Let $R$ be a commutative ring and let $N$ be a free $R$ module. Let $B$ be a set and $\phi: B\rightarrow N$ be an injective map. We say that $B$ {\it induces a basis} of $N$ if the set \{$\phi(b)\mid b\in B$\} forms a basis of $N$.
	
	\begin{theorem}[\cite{Con1}, A.1]\label{th.ba1}
		Let $n\in\mathbb{N}_{\ge 2}$ and $n\ne 4$. Let $n=p_1^{e_1}\cdots p_\ell^{e_\ell}$ be the prime factorization of $n$ with $p_1<\dots<p_\ell$, and $e_i>0$ for $1\le i\le \ell$. Depending on what type of number $n$ is, the following set $B_n$ contained in $X^n$ induces a basis of $\widehat{X^n}$ under the quotient map $X^n \rightarrow\widehat{X}^n$.
		\begin{enumerate}
			\item Assume that $n=2$.
			Then $B_2:=\{(1)_2\in X^2\}$.
			\item Assume that $n$ is an odd prime.
			Then $B_n:=\{(b_1)_n\in X^n\mid b_1\in\mathbb{N} \text{ and } 1\le b_1\le \frac{n-1}{2}\}.$
			\item Assume that $n\ne 2,\; 2\mid n,$ and $4\nmid n$. Then $B_n$ is empty, i.e. $\widehat{X^n}$ is trivial.
			\item Assume that $n$ is odd, square-free and $\ell\ge 2$. Let\\  $$C_{n,\ell}:=\{(\overbrace{1,\dots,1}^{\ell-1}, b_\ell)_n\in X^n\mid\frac{p_\ell+1}{2}\le b_\ell\le p_\ell-2\text{ with } b_\ell\in\mathbb{N}\}.$$ 
			
			For $1\le k\le \ell-1$, 
			if $p_k\ne 3$, let 
			
			\begin{align*}
				C_{n,k}:=&\{(\overbrace{1,\dots,1}^{k-1}, b_k,\dots, b_\ell)_n\in X^n\mid\frac{p_k+1}{2}\le b_k\le p_k-2\text{ with } b_k\in\mathbb{N}. \\
				&\text{ If } k+1\le j\le \ell,  \text{ then }1\le b_j\le p_j-2 \text{ with } b_j\in\mathbb{N} \}.
			\end{align*}

			If $p_k=3$, then $k=1$. Let $C_{n,k}:=\emptyset$. Let $C_n:=\bigcup\limits_{k=1}^{\ell}C_{n,k}.$\\
			\begin{enumerate}
				\item When $\ell$ is even, then $B_n:=C_n\bigcup\{(\overbrace{1,\dots,1}^{\ell})_n\}.$\\
				\item When $\ell$ is odd, then $B_n:=C_n$.\\
			\end{enumerate}
			
			\item Assume that $n=4m$ where m is odd and square-free, and $\ell\ge 2$. If $p_\ell=3$, then $\ell=2$.\\ 
			If $p_\ell\ne 3$, let $$C_{n,\ell}:=\{((0,1),\overbrace{1,\dots,1}^{\ell-2}, b_\ell)_n\in X^n\mid\frac{p_\ell+1}{2}\le b_\ell\le p_\ell-2\text{ with } b_\ell\in\mathbb{N}\}.$$ If $p_\ell=3$, let $C_{n,\ell}:=\emptyset$. Assume that $\ell\ge 3$. For $2\le k\le \ell-1$, 
			if $p_k\ne 3$, let
			
			\begin{align*}
				C_{n,k}:=&\{((0,1),\overbrace{1,\dots,1}^{k-2}, b_k,\dots, b_\ell)_n\in X^n\mid\frac{p_k+1}{2}\le b_k\le p_k-2\text{ with } b_k\in\mathbb{N}, \\&\text{ and for } k+1\le j\le \ell,  1\le b_j\le p_j-2 \text{ with } b_j\in\mathbb{N} \}.
			\end{align*}

			If $p_k=3$, let $C_{n,k}:=\emptyset$. Let $C_n:=\bigcup\limits_{k=2}^{\ell}C_{n,k}.$\\

			\begin{enumerate}
				\item When $\ell$ is even, then $B_n:=C_n\bigcup\{((0,1),\overbrace{1,\dots,1}^{\ell-1})_n\}$\\
				\item When $\ell$ is odd, then $B_n:=C_n$.\\
				
			\end{enumerate}
			\item Assume that $n=m$ or $4m$, where $m$ is odd and non-square-free, or $n$ satisfies $8\mid n$. If $8\mid n$, let $\mu:=1$. If $8\nmid n$, let $\mu:=\operatorname{min}\{1\le j\le\ell\mid p_j\ne 2\text{ and } e_j\ge 2\}$. Then

			\begin{align*}
				B_n:=&\{(b_1,\dots,b_\ell)_n\in X^n\mid \text{ for } 1\le i\le \ell, \text{ if } e_i=1,  \text{ then }1\le b_i\le p_i-2.\\
				& \,\text{ If } e_i\ge 2,  \text{ then }b_i=(\overline{b}_i,\widehat{b}_i) \text{ with } 0\le \overline{b}_i\le p_i-2\text{ and } 1\le \widehat{b}_i< p_i^{e_i-1}.\\
				&\text{ Furthermore } 1\le \widehat{b}_\mu< \frac{p_\mu^{e_\mu-1}}{2}\}.
			\end{align*}

		\end{enumerate}
	\end{theorem}

	\begin{theorem}[\cite{Con1}, 2.3.8,  see also \cite{Con2} Theorem 4.6]\label{th.ba2}
		Let $n\in\mathbb{N}_{\ge 2}$. For $d\in\mathbb{N}$ and $d\mid n$, let $B_d\subseteq X^d$ be the set that induces a basis of $\widehat{X^d}$ given in Theorem \ref{th.ba1}. Then
		\begin{enumerate}
			\item $\bigcup\limits_{\substack{d\,\mid\, n\\d\ge 2}} B_d$ is a basis of $X^n$ if $4\nmid n$.
			\item $\{v(4,1)\}\cup\bigcup\limits_{\substack{d\,\mid\, n\\d>2,\; d\ne 4} }B_d$ is a basis of $X^n$ if $4\mid n$.
		\end{enumerate}
	\end{theorem}

	\begin{definition}
		For $n\in\mathbb{N}_{\ge 2}$ and $n\ne 4$, let $\{v_i\}_{i=1}^{m}$ be the basis of $\widehat{X^n}$ given in Theorem \ref{th.ba1}. Let $v\in\widehat{X^n}$, and let $v=\prod\limits_{i=1}^{m}v_i^{f_i}$ be its representation in the basis. Let $\operatorname{supp}(v):=\{v_i\mid f_i\ne 0\}$. We call $\operatorname{supp}(v)$ the {\it relative support} of $v$. 
	\end{definition}

	\begin{definition}\label{de.su}
		For $n\in\mathbb{N}_{\ge 2}$, let $\{v_i\}_{i=1}^{m}$ be the basis of $X^n$ given in Theorem \ref{th.ba2}. Let $v\in X^n$, and let  $v=\prod\limits_{i=1}^{m}v_i^{f_i}$ be its representation in the basis. For $1\le i\le m$, we call $f_i$ the {\it multiplicity} of $v$ at the basis element $v_i$. We denote the multiplicity of $v$ at $v_i$ by $\operatorname{multi}_{v_i}(v)$. 
	\end{definition}

	\section{The Prime Case}\label{sec.prime}
	
	Let $n$ be an odd prime. In this section, we analyze the solution to Equation \eqref{eq.ne4} under the condition that the denominator of the solution is $n, 2n$, or $4n$. In particular, we prove the following Theorem.

	\begin{theorem}\label{th.4pr2}
		Let $n$ be an odd prime. Assume that $(x_0,x_1,x_2,x_3,x_4)$ is a solution to \eqref{eq.ne4} in $G$ with $0<x_i<\frac{\pi}{2}$ and $\operatorname{den}(x_i)\in\{n,2n,4n\}$ for each $0\le i\le 4$. Then the tuple $(x_0,x_1,x_2,x_3,x_4)$ is in $\Phi_{1,1}$.
	\end{theorem}
	
	Before giving the proof of Theorem \ref{th.4pr2}, we need some preparations.

	\begin{emp}\label{em.bapr}
			Let $n$ be an odd prime. First, we recall the basis of $X^n, X^{2n}$, and $X^{4n}$ given in Theorem \ref{th.ba2}. Let $B_n:=\{(b_1)_n\in X^n\mid b_1\in\mathbb{N} \text{ and } 1\le b_1\le \frac{n-1}{2}\}$. The set $B_n$ forms a basis of $X^n$. The set $\{(1)_2\}\cup B_n$ forms a basis of $X^{2n}$. Let $b_1=(0,1)$. Assume that $n\ne 3$. Then the set $\{(b_1)_4\}\cup\{(b_1, b_2)_{4n}\mid\frac{n+1}{2}\le b_2\le n-2\text{ where } b_2\in\mathbb{N}\}$ forms a basis of $X^{4n}$. The element $(b_1)_4$ forms a basis of $X^{12}$.
	\end{emp}

	\begin{emp}
		Let $n$ be an odd prime and let $a\in\mathbb{Z}$ be such that $\operatorname{gcd}(n,a)=1$. Let $(a_1)_n$ be the residue form of $v(n,a)$ (see \ref{ss.rf}). If $1\le a_1\le \frac{n-1}{2}$, let $\epsilon_a:=1$. If $\frac{n+1}{2}\le a_1\le n-1$, let $\epsilon_a:=-1$.
	\end{emp}

\begin{emp}\label{em.4pr}
	
	Let $n$ be an odd prime and let $a\in\mathbb{Z}$ be such that $\operatorname{gcd}(4n,a)=1$. Let $(a_1,a_2)_{4n}$ be the residue form of $v(4n,a)$.
	
	If $\frac{n+1}{2}\le a_2\le n-2$ or $a_2=1$, let $\epsilon_{a}:=1$. If $2\le a_2\le \frac{n-1}{2}$ or $a_2=n-1$, let $\epsilon_{a}:=-1$.  
	
	If $(\epsilon_{a} a)_1=(0,1)$, let $d:=1$. If $(\epsilon_{a} a)_1=(1,1)$, let $d:=-1$.
	
	Let $b_1:=(0,1), b_2:=(\epsilon_{a} a)_2$ and $v:=(b_1,b_2)_n$. As explained in Remark \ref{re.rf}, we can view $v$ as an element in $X^n$.

\end{emp}

\begin{remark}\label{re.ep}
	 Let $n$ be an odd prime and let $a\in\mathbb{Z}$ be such that $\operatorname{gcd}(n,a)=1$. When considering a representative $v(n,a)$ of an element in $X^{n}$, we assume that $\epsilon_a=1$, unless otherwise specified. Similar conventions will be adopted in other cases as well. See \ref{ss.sf} and \ref{def.nonsfepsilon} for the definitions of $\epsilon_a$ when $n$ non-prime squarefree and non-squarefree, respectively. We have already used this convention when stating Theorem \ref{th.ba1}. 
\end{remark}

\begin{lemma}\label{le.ba4p}
	Keep the assumptions in \ref{em.4pr}. The following is true.
	\begin{enumerate}
		\item $v\in B_{4n}$, where $B_{4n}$ is the subset of $X^{4n}$ that induces a basis of $\widehat{X^{4n}}$ under the quotient map given in Theorem \ref{th.ba1}.
		\item $v(4n,a)=v^d$ in $\widehat{X^{4n}}$.
	\end{enumerate}
\end{lemma}

\begin{proof}\hfill

	\begin{enumerate}
		\item This is clear from the residue from of elements in $B_{4n}$.
		\item If $(\epsilon_{a} a)_1=(0,1)$, this is clear. Assume that $(\epsilon_{a} a)_1=(1,1)$. 
		By formula \eqref{eq.nor}, we have that
		$v(4n,\epsilon_{a}a)=v(4n,\epsilon_{a}a+2n)^{-1}$ in $\widehat{X^{4n}}$. Note that $v(4n,\epsilon_{a}a+2n)$ and $v$ have the same residue form, we get that $v=v(4n,a+2n)$ in $X^{4n}$. Hence 
		$v(4n,a)=v(4n,\epsilon_{a}a)=v^{-1}$ in $\widehat{X^{4n}}$.
		
	\end{enumerate}
\end{proof}

	\begin{lemma}\label{le.congp}
		Let $n\in\mathbb{N}_{\ge 3}$ be an odd prime. Let $a, b\in\mathbb{Z}$ such that $0< a,b<\frac{n}{2}$ and $\operatorname{gcd}(n,ab)=1$. Then $v(n,a)=v(n,b)$ in $X^{n}$ if and only if $a=b$.
	\end{lemma}

	\begin{proof}
		
		We first show that $v(n,a)=v(n,b)$ in $X^{n}$ if and only if $a\equiv  b\operatorname{mod} n$ or $a\equiv -b\operatorname{mod} n$. If $a\not\equiv  b\operatorname{mod} n$ and $a\not\equiv -b\operatorname{mod} n$, then $(\epsilon_a a)_1\ne (\epsilon_b b)_1$. So $v(n,a)\ne v(n,b)$ in $X^{n}$. See \ref{em.bapr}. The other direction is immediate.
		
		Under the additional condition $0< a,b<\frac{n}{2}$, we have that  $a\equiv  b\operatorname{mod} n$ or $a\equiv -b\operatorname{mod} n$ if and only if $a=b$. So the claim holds.
	\end{proof}

	\begin{lemma}\label{le.cong4p}
		Let $n\in\mathbb{N}_{\ge 2}$ be a prime number. Let $a,b\in\mathbb{Z}$ be such that $0< a, b< 2n$ and $\operatorname{gcd}(4n,ab)=1$. Then

		\begin{enumerate}
			\item $v(4n,a)=v(4n,b)$ in $\widehat{X^{4n}}$ if and only if $a=b$.
			\item $v(4n,a)=v(4n,b)^{-1}$ in $\widehat{X^{4n}}$ if and only if $a+b=2n$.
		\end{enumerate}

	\end{lemma}
	
	\begin{proof}

		 (i). We first show that $v(4n,a)=v(4n,b)$ in $\widehat{X^{4n}}$ if and only if $a\equiv b\operatorname{mod} 4n$ or $a\equiv -b\operatorname{mod} 4n$. Assume that $v(4n,a)=v(4n,b)$ in $\widehat{X^{4n}}$. Assume for the sake of contradiction that $a\not\equiv  b\operatorname{mod} 4n$ and $a\not\equiv -b\operatorname{mod} 4n$. Then either $(\epsilon_{a} a)_1\ne (\epsilon_{b} b)_1$ or $(\epsilon_{a} a)_2\ne (\epsilon_{b} b)_2$. 
		
		Let $v_a, v_b\in B_{4n}$ be the associated elements to $v(4n,a)$ and $v(4n,b)$, respectively as in \ref{em.4pr}. So $v(4n,a)=v_a^{\pm 1}$ in $\widehat{X^{4n}}$ and $v(4n,b)=v_b^{\pm 1}$ in $\widehat{X^{4n}}$.
		
		If $(\epsilon_{a} a)_2\ne (\epsilon_{b} b)_2$, then $v_a\ne b_b$. Hence $v(4n,a)\ne v(4n,b)$ in $\widehat{X^{4n}}$. If $(\epsilon_{a} a)_2= (\epsilon_{b} b)_2$ and $(\epsilon_{a} a)_1\ne (\epsilon_{b} b)_1$, then $v_a=v_b$. Assume that $(\epsilon_{a} a)_1=(0,1)$. Then $(\epsilon_{b} b)_1=(1,1)$. So $v(4n,a)=v_a\ne v_b^{-1}=v(4n,b)$ in $\widehat{X^{4n}}$. Thus $v(4n,a)=v(4n,b)$ in $\widehat{X^{4n}}$ if and only if $a\equiv b\operatorname{mod} 4n$ or $a\equiv -b\operatorname{mod} 4n$. Since we assumed that $0<a,b<2n$, it follows that $a=b$. The other direct is immediate.

	A similar argument shows that $v(4n,a)=v(4n,b)^{-1}$ in $\widehat{X^{4n}}$ if and only if $a\equiv b+2n\operatorname{mod} 4n$ or $a\equiv -b+2n\operatorname{mod} 4n$. Part (ii) then follows. We leave the details to the reader.
	\end{proof}

	\begin{lemma}\label{le.mulp}
		Let $n$ be an odd prime. Let $0< a< \frac{n}{2}$ and $v:=v(n,a)$. Let $0<b<n$. Then $\operatorname{multi}_v(\operatorname{tan}\frac{b}{2n}\pi)\in\{0,\pm 1,\pm 2\}$. Moreover, the following is true.

		\begin{enumerate}
			\item If $\operatorname{multi}_v(\operatorname{tan}\frac{b}{2n}\pi)=1$, then $b$ is odd. Furthermore, if $a$ is even, then $b=n-a$; if $a$ is odd, then $b=a$.
			\item If $\operatorname{multi}_v(\operatorname{tan}\frac{b}{2n}\pi)=-1$, then $b$ is even. Furthermore, if $a$ is even, then $b=a$; if $a$ is odd, then $b=n-a$. 
			\item If $\operatorname{multi}_v(\operatorname{tan}\frac{b}{2n}\pi)=2$,then $b$ is even and $b=2a$.
			\item If $\operatorname{multi}_v(\operatorname{tan}\frac{b}{2n}\pi)=-2$, then $b$ is odd and  $2a+b=n$.
		\end{enumerate}
	\end{lemma}
	
	\begin{proof}
		By Corollary \ref{cor.tarq}, it is clear that $\operatorname{multi}_v(\operatorname{tan}\frac{b}{2n}\pi)\in\{0,\pm 1,\pm 2\}$.
		
		If $\operatorname{multi}_v(\operatorname{tan}\frac{b}{2n}\pi)=1$, then $b$ is odd and $v(n,a)=v(n,b)$ by Corollary \ref{cor.tarq}. By Lemma \ref{le.congp}, $a\equiv  b\operatorname{mod} n$ or $a\equiv -b\operatorname{mod} n$. Since $0< a< \frac{n}{2}$ and $0<b<n$, we know that $0<a+b<\frac{3n}{2}$ and $-n<a-b<\frac{n}{2}$. So $a+b=n$ or $a-b=0$. If $a$ is even, then $b=n-a$, since $b$ is odd; similarly, if $a$ is odd, then $b=a$.
		
		If $\operatorname{multi}_v(\operatorname{tan}\frac{b}{2n}\pi)=2$, then $b$ is even and $v(n,\frac{b}{2})=v(n,a)$ by Corollary \ref{cor.tarq}. So $\frac{b}{2}\equiv a\operatorname{mod} n$ or $\frac{b}{2}\equiv -a\operatorname{mod} n$ by Lemma \ref{le.congp}. Since $0< a< \frac{n}{2}$ and $0<b<n$, we know that $-n<2a-b<n$ and $0<2a+b<2n$. So either $2a-b=0$ or $2a+b=n$. Since $b$ is even, $b=2a$.
		
		The cases $\operatorname{multi}_v(\operatorname{tan}\frac{b}{2n}\pi)=-1$ and $\operatorname{multi}_v(\operatorname{tan}\frac{b}{2n}\pi)=-2$ can be proved similarly.
	\end{proof}

	\begin{lemma}\label{le.02}
		Let $n$ be an odd prime. Assume that $(x_0,x_1,x_2,x_3,x_4)$ is a solution in $G$ to Equation \eqref{eq.ne4} with $0<x_i<\frac{\pi}{2}$. Assume that $\operatorname{den}(x_i)\in\{n,2n,4n\}$ for $1\le i\le 4$ and $\operatorname{den}(x_0)\in\{n,2n\}$. Then $\#\{1\le i\le 4\mid \operatorname{den}(x_i)=4n\}=0$ or $2$.
	\end{lemma}

	\begin{proof}
		We can write $x_i=\frac{x_i'}{4n}\pi$ where  $x_i'\in \mathbb{N}$, $0<x_i'<2n$, $n\nmid x_i'$, for $0\le i\le 4$. Since $\operatorname{den}(x_0)\in\{0,2n\}$,	for each $v\in B_{4n}$, we have that  $\operatorname{multi}_v(\operatorname{tan}\frac{x_0'}{4n}\pi)=0$.

		Assume that $\operatorname{den}(x_i)=4n$ for some $1\le i\le 4$. Let $v\in B_{4n}$ be the element associated to $v(4n,x_0')$ as in \ref{em.4pr}. Then $\operatorname{multi}_v(\operatorname{tan}\frac{x_i'}{4n}\pi)=\pm 2$. Assume that $\operatorname{multi}_v(\operatorname{tan}\frac{x_i'}{4n}\pi)=2$, i.e., $v(4n,x_0')=v^2$ in $\widehat{X^{4n}}$.
		Then there exists $1\le j\le 4$ such that $\operatorname{den}(x_j)=4n$ and $\operatorname{multi}_v(\operatorname{tan}\frac{x_j'}{4n}\pi)= -2$, i.e. $v(4n,x_i')=v^{-2}$ in $\widehat{X^{4n}}$. By Lemma \ref{le.cong4p}, $x_i+x_j=\frac{\pi}{2}$. Assume that $i=3$ and $j=4$. Then 
			\begin{equation*}\label{eq.ne7}
	(\operatorname{tan}x_0)^2=(\operatorname{tan}x_1)(\operatorname{tan}x_2).
\end{equation*}
in $\mathbb{C}^*$. If $\operatorname{den}(x_f)=4n$ for $1\le f\le 2$, by the same argument, we get that $\operatorname{den}(x_g)=4n$ for $1\le g\le 2$ and $g\ne f$. Furthermore, $x_f+x_g=\frac{\pi}{2}$. Then 
		
					\begin{equation*}\label{eq.ne8}
			(\operatorname{tan}x_0)^2=1
		\end{equation*}
	in $\mathbb{C}^*$. So $x_0=\frac{\pi}{4}$. This is a contradiction. So $\#\{1\le i\le 4\mid \operatorname{den}(x_i)=4n\}=0$ or $2$. The case	$\operatorname{multi}_v(\operatorname{tan}\frac{x_i'}{4n}\pi)=-2$ is similar.
	\end{proof}

	\begin{proof}[Proof of Theorem \ref{th.4pr2}]

		Assume that $(x_0,x_1,x_2,x_3,x_4)$ is a solution to \eqref{eq.ne4} in $G$ with $0<x_i<\frac{\pi}{2}$ and the denominator of $x_i$ equals $n$, $2n$ or $4n$ for $0\le i\le 4$.
		We can write $x_i=\frac{x_i'}{4n}\pi$ where  $x_i'\in \mathbb{N}$, $0<x_i'<2n$, and $n\nmid x_i'$ for $0\le i\le 4$. We analyze the cases $x_0'$ odd and even separately.\\

		(i) The case $x_0'$ is odd. In this case $\operatorname{tan}\frac{x_0'}{4n}\pi=v(4n,x_0')^2$ in $\widehat{X^{4n}}$. Let $v\in B_{4n}$ be the element associated to $v(4n,x_0')$ as in \ref{em.4pr}. Then $$\operatorname{multi}_v((\operatorname{tan}\frac{x_0'}{4n}\pi)^2)=\pm 4.$$ 
		
		Assume that $\operatorname{multi}_v((\operatorname{tan}\frac{x_0'}{4n}\pi)^2)=4$. Then $v(4n,x_0')=v$ in $\widehat{X^{4n}}$. So $\operatorname{multi}_v(\prod\limits_{i=1}^4\operatorname{tan}x_i)=4$. Hence, there exist $1\le i<j\le 4$ such that $x_i'$ and $x_j'$ are odd and $v(4n,x_i')=v(4n,x_j')=v$ in $\widehat{X^{4n}}$. By Lemma \ref{le.cong4p}, $x_i'=x_j'=x_0'$.

		Assume that $i=1$ and $j=2$. Then $(\operatorname{tan}x_0)^2=\operatorname{tan}x_1\operatorname{tan}x_2$ in $\mathbb{C}^*$ which implies that $\operatorname{tan}x_3\operatorname{tan}x_4=1$ in $\mathbb{C}^*$. Hence $x_3+x_4=\frac{\pi}{2}$. So $(x_0,x_1,x_2,x_3,x_4)\in\Phi_{1,1}$.\\
			
			(ii) The case $x_0'$ is even. By Lemma \ref{le.02}, we have that $N:=\#\{1\le i\le 4\mid \operatorname{den}(x_i)=4n\}=0$ or $2$. 
			
			The case $N=2$. By the proof of Lemma \ref{le.02}, we can assume that $\operatorname{den}(x_3)=\operatorname{den}(x_4)=4n$ and $(\operatorname{tan}x_3)(\operatorname{tan}x_4)=1$ in $\mathbb{C}^*$.	Then we get the following relation:
			\begin{equation*}\label{eq.ne6}
				(\operatorname{tan}x_0)^2=(\operatorname{tan}x_1)(\operatorname{tan}x_2)
			\end{equation*}
		in $\mathbb{C}^*$. By the assumption, we know that $\operatorname{den}(x_i)\in\{n,2n\}$ for $0\le i\le 2$. Assume that $4\mid x_0'$. The case $4\nmid x_0'$ is similar. Let  $v:=(n,\frac{x_0'}{4})$, then $\operatorname{multi}_v(\operatorname{tan}^2x_0)=4$. So
		 $$\operatorname{multi}_v(\operatorname{tan}x_1\operatorname{tan}x_2)=\operatorname{multi}_v(\operatorname{tan}x_1)+\operatorname{multi}_v(\operatorname{tan}x_2)=4.$$
		  Since $\operatorname{multi}_v(\operatorname{tan}x_i)\le 2$, So $\operatorname{multi}_v(\operatorname{tan}x_1)=\operatorname{multi}_v(\operatorname{tan}x_2)=2$. By Lemma \ref{le.mulp}, we get that $x_1=x_2=x_0$. Hence $(x_0,x_1,x_2,x_3,x_4)\in\Phi_{1,1}$. 
		
		The case $N=0$. Assume that $4\mid x_0'$. The case $4\nmid x_0'$ is similar. Let $v:=v(n,\frac{x_0'}{4})$. Then $\operatorname{multi}_v(\operatorname{tan}^2x_0)=4$. So 
		$$\operatorname{multi}_v(\prod\limits_{i=1}^4\operatorname{tan}x_i)=\sum\limits_{i=1}^4\operatorname{multi}_v(\operatorname{tan}x_i) =4.$$ By Lemma \ref{le.mulp}, there are the following three possible cases.

		\begin{enumerate}[label=(\alph*)]
			\item $\operatorname{multi}_v(\operatorname{tan}x_i)=\operatorname{multi}_v(\operatorname{tan}x_j)=2$ for $1\le i<j\le 4$. By Lemma \ref{le.mulp}, we get that $x_i=x_j=x_0$. So $(x_0,x_1,x_2,x_3,x_4)\in\Phi_{1,1}$.
			\item $\operatorname{multi}_v(\operatorname{tan}x_i)=2$, and $\operatorname{multi}_v(\operatorname{tan}x_j)=\operatorname{multi}_v(\operatorname{tan}x_k)=1$ where $1\le i,j,k\le 4$ and $j\ne k$. 
			By Lemma \ref{le.mulp}, we get that (i) $x_i=x_0$. (ii) if $\frac{x_0'}{4}$ is even, then $x_j'=x_k'=2n-\frac{x_0'}{2}$; if $\frac{x_0'}{4}$ is odd, then $x_j'=x_k'=\frac{x_0'}{2}$ for each $1\le i\le 4$. Take $w:=v(n,4^{-1}x_j')$, then $\operatorname{multi}_w(\operatorname{tan}x_j\operatorname{tan}x_k)=-4$. By Corollary \ref{cor.tarq}, if $\operatorname{multi}_w(\operatorname{tan}x_0)<0$,  then $\operatorname{multi}_w(\operatorname{tan}x_0)=-1$. By Lemma \ref{le.mulp}, we get that  $v(n,2^{-1}x_0')=v(n,4^{-1}x_j')=v(n,8^{-1}x_0')$. By Lemma \ref{le.congp}, we obtain $2^{-1}x_0'\equiv 8^{-1}x_0'\operatorname{mod} n$. So $4\equiv \pm 1\operatorname{mod} n$, i.e. $n=3$ or $5$. Assume that $n\ge 7$. Then  $$\operatorname{multi}_w(\operatorname{tan}x_0)\ge 0.$$ Since $\operatorname{multi}_v(\prod\limits_{\substack{m=1
				\\m\ne i}}^4\operatorname{tan}x_m)<0$, we get a contradiction.

			\item $\operatorname{multi}_v(\operatorname{tan}x_i)=1$ for each $1\le i\le 4$. By Lemma \ref{le.mulp}, we derive that $\frac{x_i'}{2}$ is odd. Furthermore, if $\frac{x_0'}{4}$ is even, then $x_i'=2n-\frac{x_0'}{2}$; if $\frac{x_0'}{4}$ is odd, then $x_i'=\frac{x_0'}{2}$ for each $1\le i\le 4$. Take $w:=v(n,4^{-1}x_i')$, then $\operatorname{multi}_w(\prod\limits_{i=1}^4\operatorname{tan}x_i)=-8<\operatorname{multi}_v(\operatorname{tan}^2x_0)$. This is a contradiction.
		\end{enumerate}

			If $n< 7$, one can directly verify the claim case by case. 
	\end{proof}

	\begin{remark}
		
It is worth noting that the approach applied in the proof of Theorem \ref{th.sfn2n} in section \ref{sec.sf2} offers an alternative method for proving the case where $x_0'$ is even.
		
	\end{remark}

	Now we can state the general solutions in $G$ to Equations \eqref{eq.ne4} and \eqref{eq.ne5}, assuming the denominator of $x_i$ is $n$, $2n$ or $4n$ for $0\le i\le 4$. Since tangent has period $\pi$, we can assume that $-\frac{\pi}{2}<x_i<\frac{\pi}{2}$, for $0\le i\le 4$. 
	
	\begin{corollary}\label{cor.gen}
		Let $n$ be an odd prime. Then the following is true.
		\begin{enumerate}
			\item $(x_0,x_1,x_2,x_3,x_4)$ is a solution in $G$ to \eqref{eq.ne4} with $-\frac{\pi}{2}<x_i<\frac{\pi}{2}$, $x_i\ne 0$ and $\operatorname{den}(x_i)\in\{n,2n,4n\}$ for $0\le i\le 4$ if and only if up to reordering $x_1,x_2,x_3,x_4$, the tuple $(x_0,x_1,x_2,x_3,x_4)$ is in the set
			\begin{equation*}
				\aligned
				T^{+}=&\{(\eta_0\frac{s}{4n},\eta_1\frac{s}{4n},\eta_2\frac{s}{4n},\eta_3\frac{t}{4n},\eta_4\frac{2n-t}{4n})\pi\,\, |\,\,  0< s,t<2n, \operatorname{gcd}(st,n)=1, \\&\eta_0,\eta_1,\eta_2,\eta_3,\eta_4\in\{1,-1\},\text{ and } \prod\limits_{i=1}^4\eta_i=1\}.
				\endaligned
			\end{equation*} 
			\item $(x_0,x_1,x_2,x_3,x_4)$ is a solution in $G$ to \eqref{eq.ne5} with $-\frac{\pi}{2}<x_i<\frac{\pi}{2}$, $x_i\ne 0$  and $\operatorname{den}(x_i)\in\{n,2n,4n\}$ for $0\le i\le 4$ if and only if up to reordering $x_1,x_2,x_3,x_4$, the tuple $(x_0,x_1,x_2,x_3,x_4)$ is in the set
			\begin{equation*}
				\aligned
				T^{-}=&\{(\eta_0\frac{s}{4n},\eta_1\frac{s}{4n},\eta_2\frac{s}{4n},\eta_3\frac{t}{4n},\eta_4\frac{2n-t}{4n})\pi\,\, |\,\,  0< s,t<2n, \operatorname{gcd}(st,n)=1, \\&\eta_0,\eta_1,\eta_2,\eta_3,\eta_4\in\{1,-1\},\text{ and } \prod\limits_{i=1}^4\eta_i=-1\}.
				\endaligned
			\end{equation*} 
		\end{enumerate}
	\end{corollary}
	
	\begin{proof}
			The conclusions follow from Theorem \ref{th.4pr2}.
	\end{proof}

	Notice that similar extensions as in Corollary \ref{cor.gen} can also be made to Theorem \ref{th.4msf}, Theorem \ref{th.sfn2n}, Theorem \ref{th.4mnsf}, Theorem \ref{th.mnsf}, and Theorem \ref{th.8nsf}.

	\begin{corollary}\label{cor.tri}
		Suppose that $E,a,b, c\in G$ and $\operatorname{den}(E)=\operatorname{den}(a)=\operatorname{den}(b)=\operatorname{den}(c)$ is a prime. Then $(E,a,b,c)$ is the measurement of a proper rational spherical triangle if and only if $(E,a,b,c)=(\frac{1}{2},\frac{1}{2},\frac{1}{2},\frac{1}{2})\pi$.
	\end{corollary}

	\begin{proof}
		One can verify directly that $(E,a,b,c)=(\frac{1}{2},\frac{1}{2},\frac{1}{2},\frac{1}{2})\pi$is the measurement of a proper rational spherical triangle. We now show the other direction. Let $x_1:=\frac{-a+b+c}{4},x_2:=\frac{a-b+c}{4},x_3:=\frac{a+b-c}{4}$, and $x_4:=\frac{a+b+c}{4}$. By Lemma \ref{lem.ineq}, we obtain that $0<x_i<\frac{\pi}{2}$ for each $1\le i\le 4$. Notice that $x_4>x_s$ for each $1\le s\le 3$. By Theorem \ref{th.4pr2}, we can assume that $x_3+x_4=\frac{\pi}{2}$ and $x_1=x_2$. Here we also used the fact that $a,b,c$ are symmetric. It follows that $a=b=\frac{\pi}{2}$. Hence $\operatorname{den}(c)=\operatorname{den}(E)=2$. By the properness, we know that $c=\frac{\pi}{2}$, and $E=\frac{\pi}{2}$ or $\frac{3\pi}{2}$. One can verify that $E\ne \frac{3\pi}{2}$. 
	\end{proof}

	\section{The basis representation in the square-free case}\label{sec.bassf}
	
Let $n\in\mathbb{N}_{\ge 3}$ be odd, square-free and not a prime. Let $a$ be an integer such that $\operatorname{gcd}(4n,a)=1$. This section provides the representation of element $v(4n,a)$ in Conrad's basis presented in Proposition \ref{th.ba1}. This section might be of independent interest.
	
	\begin{emp}{\bf Notation}\label{ss.sf}
		Let $n\in\mathbb{N}_{> 3}$ be odd, square-free and not a prime. Let $4n=p_1^{e_1}p_2\cdots p_\ell$ be the prime factorization of $n$ with $p_1<\dots<p_\ell$. Note that $p_1=2$ and $e_1=2$. 
		Let $$\delta(4n):=\#\{1\le s\le\ell\mid p_s=2\text{ or } p_s=3\}.$$
		 Let $a\in\mathbb{Z}$ be such that $\operatorname{gcd}(4n,a)=1$. Let $(a_1,\dots,a_\ell)_{4n}$ be the residue form of $v(n,a)$ as defined in \ref{ss.rf}. Let $$\operatorname{len}(4n,a):=\operatorname{max}\{1\le r\le \ell\mid \text{ for all } s\le r, \overline{a}_s=1\text{ or } \overline{a}_s=p_s-1\}.$$ 
		
We call $\operatorname{len}(4n,a)$ the \textit{cluster length} of $v(4n,a)$. If $\operatorname{len}(4n,a)<\ell$, let $$\tau(a):=\operatorname{len}(4n,a)+1.$$ 
	
	If $\operatorname{len}(4n,a)=\ell$, let 
	$$\tau(a):=\ell.$$

	Let $s:=\delta(4n)+1$. If $\frac{p_s+1}{2}\le a_s\le p_s-2$ or $a_s=1$, let $\epsilon_a:=1$. If $2\le a_s\le \frac{p_s-1}{2}$ or $a_s=p_s-1$, let $\epsilon_a:=-1$. Let $$\operatorname{pol}(4n,a):=\{1\le s\le \ell\mid (\epsilon_a a)_s=p_s-1\}.$$ 
		
		An element in $\operatorname{pol}(4n,a)$ is called a \textit{pole} of $v(4n,a)$. If $\operatorname{len}(4n,a)>0$, let $$\overline{\operatorname{pol}}(4n,a):=\{1\le s\le \operatorname{len}(4n,a)\mid (\epsilon_a a)_s=p_s-1 \}.$$ 
	
	If $\operatorname{len}(4n,a)=0$, let $\overline{\operatorname{pol}}(4n,a)$ be the empty set. If $0\le r\le \operatorname{len}(4n,a)-1$, let

	$$\overline{\operatorname{pol}}(4n,a,r):=\{s\in\overline{\operatorname{pol}}(4n,a)\mid s>r\}.$$

	If $\operatorname{len}(4n,a)<\ell$, let $$\widehat{\operatorname{pol}}(4n,a):=\{\operatorname{len}(4n,a)+1\le s\le \ell \mid (\epsilon_a a)_s=p_s-1 \}.$$ 
	
	If $\operatorname{len}(4n,a)=\ell$, let $\widehat{\operatorname{pol}}(4n,a)$ be the empty set. It is clear that
	
	 $$\operatorname{pol}(4n,a)=\overline{\operatorname{pol}}(4n,a)\cup\widehat{\operatorname{pol}}(4n,a).$$

	If $\operatorname{len}(4n,a)>1$, let $$\operatorname{pmin}(4n,a):=\operatorname{min}(\overline{\operatorname{pol}}(4n,a,\delta(4n))).$$

	If $\overline{\operatorname{pol}}(4n,a,\delta(4n))=\emptyset$, then let $\operatorname{pmin}(4n,a):=\ell+1$. Let 
	
	$$E(4n,a):=\{1\le s\le \ell\mid a_s=p_s-1\}$$ 
	
	and 
	$$F(4n,a):=\{1\le s\le \ell\mid a_s=1\}.$$ 
	
	Let $e(a):=|E(4n,a)|$, and $f(a):=|F(4n,a)|$. Observe that if $\epsilon_a=1$, then $E(4n,a)=\operatorname{pol}(4n,a)$. If $r\in E(4n,a)$ and $r>\delta(4n)$, let 
	
	$$\operatorname{ord}(r):=\#\{\delta(4n)<s\le r\mid a_s=p_s-1\}$$ 
	
	and let $$w_r(4n,a):=(b_1,\dots, b_\ell)_{4n}\in X^{4n}$$ 
	
	where if $\delta(4n)< s\le r$, then $b_s=1$; if $1\le r\le 	\delta(4n)$ or $r<s\le \ell$, then $b_s=a_s$.

	\end{emp}

	\begin{emp}\label{em.basic0}

	Keep the assumption on $n$ in \ref{ss.sf}. Let $a\in\mathbb{Z}$ such that $\operatorname{gcd}(4n,a)=1$ and $\epsilon_a=1$. Assume that $\delta(4n)\le \operatorname{len}(4n,a)\le \ell$ and if $\delta(4n)< s\le\operatorname{len}(4n,a)$, then $a_s=1$. We also assume that $\frac{p_{\tau(a)}+1}{2}\le a_{\tau(a)}\le p_{\tau(a)}-2$. Let

			\begin{equation*}
			\aligned
			\Gamma(4n,a):=&\{(b_1\dots,b_\ell)_{4n}\in X^{4n}\mid b_1=(0,1).  \text{ If } 2\le s<\tau(a), \text{ then } b_s=1. \text{ For } \tau\le s\le\ell,\\
			& \text{ if } s\in\operatorname{pol}(4n,a), \text{ then } 1\le b_s\le p_s-2. \text{ If } s\in\operatorname{pol}(4n,a), \text{ then } b_s=a_s\}.
			\endaligned
		\end{equation*}
	
	Let 
	
	$$K(4n,a):=\prod\limits_{v\in \Gamma(4n,a)}v^{(-1)^{e(a)}}.$$ 
	
	It is clear that $\Gamma(4n,a)\subset B_{4n}$.

	\end{emp}

	\begin{lemma}\label{lemma.basic0}
	Keep the assumptions in \ref{em.basic0}.
	Then
	\begin{equation*}
		v(4n,a)=K(4n,a)
	\end{equation*}
	in $\widehat{X^{4n}}$, where both sides are viewed as elements in $\widehat{X^{4n}}$ under the quotient map.

\end{lemma}
	
	\begin{proof}
		The equality follows immediately from \ref{le.nor1}.
	\end{proof}

	\begin{emp}\label{em.fl}
		Keep the assumption on $n$ in \ref{ss.sf}. Let $a\in\mathbb{Z}$ be such that $\operatorname{gcd}(4n,a)=1$. Assume that $\delta(4n)< \operatorname{len}(4n,a)<\ell$ and $a_s=1$ for each $\delta(4n)< s\le\operatorname{len}(4n,a)$. We also assume that $2\le a_{\tau(a)}\le \frac{p_{\tau(a)}-1}{2}$. It is convenient to introduce the following notations.

Let

	\begin{equation*}
	\aligned
	\overline{\Gamma}_1(4n,a):=&\{(b_1\dots,b_\ell)_{4n}\in X^{4n}\mid b_1=(0,1).  \text{ If } 3\mid n, \text{ then } b_2=1. \\
	&\text{ if } \delta(4n)< s<\tau(a), \text{ then } b_s=1 \text{ or } \frac{p_s+1}{2}\le b_s\le p_s-2.\\
	& 2\le b_{\tau(a)}\le \frac{p_{\tau(a)}-1}{2}.\text{ For } {\tau(a)}<s\le\ell, \text{ if } s\in \operatorname{pol}(4n,a),\\
	& \text{ then } 1\le b_s\le p_s-2; \text{ if } s\notin \operatorname{pol}(4n,a), \text{ then } b_s=(\epsilon_a a)_s\}.
	\endaligned
	\end{equation*}

Let

	\begin{equation*}
		\aligned
		\widehat{\Gamma}_1(4n,a):=&\{v(4n,b)\in \overline{\Gamma}_1(4n,a)\mid  \text{ if } \delta(4n)\le s<\tau(a), \text{ then } b_s=1\}.
		\endaligned
	\end{equation*}

See Remark \ref{re.ep}. Let
	
	 $$\Gamma_1(4n,a):=\overline{\Gamma}_1(4n,a)\backslash\widehat{\Gamma}_1(4n,a).$$

	Let

	\begin{equation*}
	\aligned
	\Gamma_2(4n,a):=&\{(b_1\dots,b_\ell)_{4n}\in X^{4n}\mid b_1=(0,1).  \text{ If } 3\mid n, \text{ then } b_2=1. \\
	&\text{ if } \delta(4n)< s<\tau(a), \text{ then } b_s=1 \text{ or } \frac{p_s+1}{2}\le b_s\le p_s-2.\\
	& \frac{p_{\tau(a)}+1}{2}\le b_{\tau(a)}\le p_{\tau(a)}-2.\text{ For } \tau(a)<s\le\ell, \text{ if } (\epsilon_a a)_s=1, \\
	&\text{ then } 1\le b_s\le p_s-2; \text{ if }  a_s\ne 1, 
	\text{ then } b_s=p_s- a_s\}.
	\endaligned
	\end{equation*}

	Let $$K_1(4n,a):=\prod\limits_{v\in \Gamma_1(4n,a)}v^{(-1)^{e(a)+1}}\text{ and } K_2(4n,a):=\prod\limits_{v\in \Gamma_2(4n,a)}v^{(-1)^{f(a)}}.$$

	\end{emp}

	\begin{lemma}
		Keep the assumptions in \ref{em.fl}. The following is true.
		\begin{enumerate}
			\item $\Gamma_1(4n,a)\cup \Gamma_2(4n,a)\subset B_{4n}$.
			\item $\Gamma_1(4n,a)\cap \Gamma_2(4n,a)=\emptyset$. 
		\end{enumerate}
	\end{lemma}
	
	\begin{proof}

		(i) is clear. (ii). Let $v(n,c)\in\Gamma_1(4n,a)$ and $v(n,d)\in\Gamma_2(4n,a)$. Note that by the convention stated in Remark \ref{re.ep}, we have assumed that $\epsilon_c=\epsilon_d=1$. Then $ c_{\tau(c)}=p_{\tau(a)}- d_{\tau(d)}$. So $v(n,c)\ne v(n,d)$ in $\widehat{X^{4n}}$.
	\end{proof}

	\begin{lemma}\label{lemma.basic}
		Keep the assumptions in \ref{em.fl}. Then
		\begin{equation*}
			v(4n,a)=K_1(4n,a)K_2(4n,a)
		\end{equation*}
		in $\widehat{X^{4n}}$, where both sides are viewed as elements in $\widehat{X^{4n}}$ under the quotient map.

	\end{lemma}

	\begin{proof}
		We prove the formula by induction on $\tau(a)$. Assume that $\tau(a)=\delta(4n)+1$. Then $\Gamma_1(4n,a)=\emptyset$. The equality $v(4n,a)=K_2(4n,a)$ follows from applying Lemma \ref{le.nor1} to $v(4n,-a)$. So the base case is true. 
		
		Assume that the formula is true for all $a\in\mathbb{Z}$ satisfying the conditions in the Proposition and $\tau(a)=t\le\ell-1$. Take $a\in\mathbb{Z}$ which satisfies the conditions in the Proposition and $\tau(a)=t+1$. We prove that the formula holds $v(4n,a)$.

		By the same argument as in the proof of Lemma \ref{lem.unr} applied to $v(4n,-a)$ for $t$, we get that

		\begin{align*}
			v(4n,-a)=&w_t(4n,-a)^{(-1)^{\operatorname{ord}(t)}}L_1(4n,-a) L_2(4n,-a) L_3(4n,-a)\\
			=&w_t(4n,-a)^{(-1)^{\operatorname{ord}(t)}}\prod\limits_{v\in \Delta_1(4n,-a)}v^{(-1)^{f(a)+1}} \prod\limits_{v\in \Delta_2(4n,-a)}v^{(-1)^{e(a)+1}} \prod\limits_{v\in \Delta_3(4n,-a)}v^{(-1)^{f(a)}}.
		\end{align*}

		Here we used the notations that are introduced in \ref{em.rep2}. The only change one needs to make in the proof of Lemma \ref{lem.unr} is applying the induction hypothesis, instead of Lemma \ref{lemma.basic}. We also applied the relations $f(-a)=e(a)$ and $e(-a)=f(a)$.

	By Proposition \ref{lemma.basic0}, we obtain that 
\[w_t(4n,-a)^{(-1)^{\operatorname{ord}(t)}}=\prod\limits_{v\in \operatorname{supp}(w_t(4n,-a))} v^{(-1)^{f(a)}}.\]

	Furthermore, we have the following relations.

		\begin{enumerate}[label=(\roman*)]

			\item $\Delta_2(4n,-a)=\Gamma_1(4n,a)$.
			
			$\Delta_2(4n,-a)\subset\Gamma_1(4n,a)$. Let $v(4n,c)\in\Delta_2 (4n,-a)$. Then, if $\delta(4n)< s\le \operatorname{len}(4n,a)$, then $\frac{p_s+1}{2}\le  c_s\le p_s-2$ or $c_s=1$. If $\tau(a)\le s\le\ell$ and $ a_s\ne p_s-1, 
			\text{ then } c_s= a_s$. So $v(4n,c)\in\Gamma_1(4n,a)$.

			$\Gamma_1(4n,a)\subset\Delta_2(4n,-a)$. Let $v(4n,c)\in\Gamma_1(4n,a)$. Let $r:=\operatorname{max}\{\delta(4n)<s\le\operatorname{len}(4n,a)\mid c_s\ne1\}$. Then $v(4n,c)\in\Delta_{2,r}(4n,-a)$. Then $v(4n,c)\in\Delta_{2,r}(4n,-a)$.

			\item $\Delta_1(4n,-a)\subset\Delta_3(4n,-a)$.
			
			Let $v(4n,c)\in\Delta_1(4n,-a)$. Let $r:=\operatorname{min}\{\delta(4n)< s< \tau(a)\mid c_s\ne 1\}$. Then $v(4n,c)\in\Delta_{3,r}(4n,-a)$.
			
			\item $(\Delta_3(4n,-a)\backslash\Delta_1(4n,-a))\cup\operatorname{supp}(w_t(4n,-a))=\Gamma_2(4n,a)$.
			 
			$\Delta_3(4n,-a)\backslash\Delta_1(4n,-a)\subset\Gamma_2(4n,a)$.

			Let $v(4n,c)\in\Delta_3(4n,-a)\backslash\Delta_1(4n,-a)$. Then there does not exists $s$ such that $2\le c_s\le\frac{p_s-1}{2}$. Assume that this is not true. Let $r:=\operatorname{min}\{s\mid 2\le c_s\le\frac{p_s-1}{2}\}$. Then $v(4n,a)\in\Delta_{1,r}(4n,-a)$. This is a contradiction. So $v(4n,c)\in\Gamma_2(4n,a)$. It is clear that $\operatorname{supp}(w_t(4n,-a))\subset\Gamma_2(4n,a)$.
			
			$\Gamma_2(4n,a)\subset(\Delta_3(4n,-a)\backslash\Delta_1(4n,-a))\cup\operatorname{supp}(w_t(4n,-a))$. 
			
			Let $v(4n,c)\in\Gamma_2(4n,a)$. (A) Assume that for each $\delta(4n)<s\le \operatorname{len}(4n,a)$, we have that  $c_s=1$. Then $v(4n,c)\in\operatorname{supp}(w_t(4n,-a))$. (B) Assume that for some $\delta(4n)<q\le \operatorname{len}(4n,a)$, we have that  $c_q\ne1$. Let
			$r:=\operatorname{min}\{\delta(4n)<s\le\operatorname{len}(4n,a)\mid c_s\ne 1\}$. Then $v(4n,c)\in\Delta_{3,r}(4n,-a)$. If $v(4n,c)\in\Delta_{1,u}(4n,-a)$ for some $\delta(4n)<u\le\operatorname{len}(4n,a)$, then $2\le c_u\le\frac{p_u-1}{2}$, this is a contradiction. So $v(4n,c)\notin\Delta_1(4n,-a)$.

			\item $\Delta_3(4n,-a)\cap\operatorname{supp}(w_t(4n,-a))=\emptyset$. 
			
			Let $v(4n,c)\in\Delta_{3,r}(4n,-a)$, for some $\delta(4n)<r<\tau(a)$ and let $v(4n,d)\in\operatorname{supp}(w_t(4n,-a))$. Then $\frac{p_r+1}{2}\le c_r\le p_r-s$ and $ d_r=1$. So $v(4n,c)\ne v(4n,d)$ in $\widehat{X^{4n}}$.

		\end{enumerate}
		
		Observe that if $v\in\Delta_1(4n,-a)$, then  $\operatorname{multi}_v(L_1(4n, -a))=-\operatorname{multi}_v(L_3(4n,-a))$.

		To sum up, we derive that $v(4n,a)=K_1(4n,a)K_2(4n,a)$. So the formula is true for $a\in\mathbb{Z}$ such that $\tau(a)=t+1$. This finishes the proof.
	\end{proof}

		\begin{emp}\label{em.rep2}
		Let $n$ be as in \ref{ss.sf}. Let $a\in\mathbb{Z}$ such that $\operatorname{gcd}(4n,a)=1$. In this section, we do not require that $\epsilon_a=1$. Let $(a_1,\dots,a_\ell)_{4n}$ be the residue form of $v(4n,a)$. Fix $t\in E(4n,a)$ such that $\delta(4n)<t\le\ell$. For each $r\in E(4n,a)$ such that $\delta(4n)<r\le t$, let

		\begin{equation*}
			\aligned
			\overline{\Delta}_{1,r}(4n,a):=&\{(b_1\dots,b_\ell)_{4n}\in X^{4n}\mid b_1=(0,1).  \text{ If } 3\mid n, \text{ then } b_2=1. \\
			&\text{ If } \delta(4n)<s<r, \text{ then } b_s=1 \text{ or } \frac{p_s+1}{2}\le b_s\le p_s-2.\\
			& 2\le b_r\le \frac{p_r-1}{2}.\text{ For } r<s\le\ell, \text{ if } s\in E(4n,a), \\
			&\text{ then } 1\le b_s\le p_s-2; \text{ if } s\notin E(4n,a), 
			\text{ then } b_s=a_s\},
			\endaligned
		\end{equation*}

		\begin{equation*}
			\aligned
			\widehat{\Delta}_{1,r}(4n,a):=&\{v(4n,b)\in\overline{\Delta}_{1,r}(4n,a)\mid b_1=(0,1).\text{ If } 2\le s<r, \\
			&\text{ then } b_s=1\},
			\endaligned
		\end{equation*}
		
		and
		
		 $$\Delta_{1,r}(4n,a):=\overline{\Delta}_{1,r}(4n,a)\backslash\widehat{\Delta}_{1,r}(4n,a).$$

		Let	\begin{equation*}
			\aligned
			\Delta_{2,r}(4n,a):=&\{(b_1,\dots,b_\ell)_{4n}\in X^{4n}\mid b_1=(0,1).  \text{ If } 3\mid n, \text{ then } b_2=1. \\
			&\text{ If } \delta(4n)<s<r, \text{ then } b_s=1 \text{ or } \frac{p_s+1}{2}\le b_s\le p_s-2.\,\,\\& \frac{p_r+1}{2}\le b_r\le p_r-2.\text{ For } r<s\le\ell, \text{ if } s\in F(4n,a),\\
			& \text{ then } 1\le b_s\le p_s-2; \text{ if } s\notin F(4n,a), 
			\text{ then } b_s=p_s-a_s\}.
			\endaligned
		\end{equation*}

		Let	\begin{equation*}
			\aligned
			\Delta_{3,r}(4n,a):=&\{(b_1,\dots,b_\ell)_{4n}\in X^{4n}\mid b_1=(0,1).  \text{ If } 3\mid n, \text{ then } b_2=1. \\
			&\text{ If } \delta(4n)<s<r, \text{ then } b_s=1.\,\,\,\,\frac{p_r+1}{2}\le b_r\le p_r-2.\\
			& \text{ For } r<s\le\ell, \text{ if } s\in E(4n,a), \text{ then } 1\le b_s\le p_s-2; \\
			&\text{ if } s\notin E(4n,a), 
			\text{ then } b_s=a_s
			\}.
			\endaligned
		\end{equation*}

		Let $$L_{1,r}(4n,a):=\prod\limits_{v\in \Delta_{1,r}(4n,a)}v^{(-1)^{e(a)+1}},\quad\quad L_{2,r}(4n,a):=\prod\limits_{v\in \Delta_{2,r}(4n,a)}v^{(-1)^{f(a)+1}},$$ 
		
		and $$L_{3,r}(4n,a):=\prod\limits_{v\in \Delta_{3,r}(4n,a)}v^{(-1)^{e(a)}}.$$

		Let
		
		$$\Delta_1(4n,a):=\bigcup\limits_{\substack{r\in E(4n,a) \\ \delta(4n)<r\le t}} \Delta_{1,r}(4n,a),\quad\quad\Delta_2(4n,a):=\bigcup\limits_{\substack{r\in E(4n,a) \\ \delta(4n)<r\le t}} \Delta_{2,r}(4n,a),$$ and $$\Delta_3(4n,a):=\bigcup\limits_{\substack{r\in E(4n,a) \\ \delta(4n)<r\le t}} \Delta_{3,r}(4n,a).$$

		Let $$L_1(4n,a):=\prod\limits_{v\in \Delta_1(4n,a)}v^{(-1)^{e(a)+1}},\quad\quad L_2(4n,a):=\prod\limits_{v\in \Delta_2(4n,a)}v^{(-1)^{f(a)+1}},$$ 
		
		and $$L_3(4n,a):=\prod\limits_{v\in \Delta_3(4n,a)}v^{(-1)^{e(a)}}.$$

	\end{emp}

\begin{lemma}\label{lem.unr}
	Keep the assumptions in \ref{em.rep2}. Then $$v(4n,a)=w_t(4n,a)^{(-1)^{\operatorname{ord}(t)}}L_1(4n,a) L_2(4n,a) L_3(4n,a)$$
	in $\widehat{X^{4n}}$.
\end{lemma}

\begin{proof}

	We prove it by induction on $\operatorname{ord}(r)$. For each $1\le k\le \#\{s\in E(4n,a)\mid \delta(4n)<s\le \operatorname{len}(4n,a)\}$, let $r_k\in E(4n,a)$ be such that $\operatorname{ord}(r_k)=k$. Note that $r_1=\operatorname{min}\{s\in E(4n,a)\mid \delta(4n)<s\le \operatorname{len}(4n,a)\}$. For each $1\le i\le p_{r_1}-2$, let $v_i:=(b_1,\dots,b_\ell)_{4n}$ where $b_{r_1}=i$ and if $s\ne r_1$, then $b_s=a_s$.
	By Lemma \ref{le.nor1}, we get that $$v(4n,a)=\prod\limits_{i=1}^{p_{r_1}-2} v_i^{-1}.$$
	
	Note that $v_1=w_{r_1}$. By Lemma \ref{lemma.basic}, we get that

	\begin{align*}
		\prod\limits_{i=2}^{\frac{p_{r_1}-1}{2}}v_i=L_{1,r_1}(4n,a)L_{1,r_2}(4n,a).
	\end{align*}

	By Lemma \ref{le.nor1}, we obtain that $$\prod\limits_{i=\frac{p_{r_1}+1}{2}}^{p_{r_1}-2}v_i=L_{3,r_1}(4n,a).$$  
	
	Thus $$v(4n,a)=w_{r_1}^{(-1)^{\operatorname{ord}(r_1)}}L_{1,r_1}(4n,a) L_{2,r_1}(4n,a) L_{3,r_1}(4n,a).$$ 
	So the base case is true. 
	
	Assume that $k<\#\{s\in E(4n,a)\mid \delta(4n)<s\le \operatorname{len}(4n,a)\}$ and the equality is true for $r_k$. We prove it for $r_{k+1}$.	By induction hypothesis, we have that  $$v(4n,a)=w_{r_k}^{(-1)^{\operatorname{ord}(r_k)}}L_{1,r_k}(4n,a) L_{2,r_k}(4n,a) L_{3,r_k}(4n,a).$$

	Similar to the proof of the base case, we get that

	\[w_{r_k}=w_{r_{k+1}}^{-1}\prod\limits_{v\in \Delta_{1,r_{k+1}}(4n,a)}v^{(-1)^{e(a)-k+1}}\prod\limits_{v\in \Delta_{2,r_{k+1}}(4n,a)}v^{(-1)^{f(a)+k+1}}\prod\limits_{v\in \Delta_{3,r_{k+1}}(4n,a)}v^{(-1)^{e(a)-k}}.\]
	
	So $$v(4n,a)=w_{r_{k+1}}^{(-1)^{\operatorname{ord}(r_{k+1})}}L_{1,r_{k+1}}(4n,a) L_{2,r_{k+1}}(4n,a) L_{3,r_{k+1}}(4n,a).$$	This conclude the proof of the formula.
\end{proof}

\begin{emp}{\bf The case $\delta(4n)<\operatorname{len}(4n,a)<\ell$ and $\frac{p_{\tau(a)}+1}{2}\le  a_{\tau(a)}\le p_{\tau(a)}-2$.}\label{ss.midgt}
	Keep the assumptions about $n$ in \ref{ss.sf}. Let $a\in\mathbb{Z}$ be such that $\operatorname{gcd}(4n,a)=1,\epsilon_a=1$, and $\frac{p_{\tau(a)}+1}{2}\le  a_{\tau(a)}\le p_{\tau(a)}-2$. Assume that $\operatorname{len}(4n,a)>\delta(4n)$ and $\overline{\operatorname{pol}}(4n,a,\delta(4n))\ne\emptyset$. See Lemma \ref{lemma.basic0} for the case $\overline{\operatorname{pol}}(4n,a,\delta(4n))=\emptyset$. 
	Assume that $2\le \operatorname{len}(4n,a)<\ell$. We introduce the following notations. For $r\in\overline{\operatorname{pol}}(4n,a,\delta(4n))$, let 
	\begin{equation*}
		\aligned
		\overline{\Gamma}_{1,r}(4n,a):=&\{(b_1,\dots,b_\ell)_{4n}\in X^{4n}\mid b_1=(0,1). \text{ If } 3\mid n, \text{ then } b_2=1. \\
		&\text{ If } \delta(4n)< s\le r-1, \text{ then } b_s=1\text{ or } \frac{p_s+1}{2}\le b_s\le p_s-2.\\
		&2\le b_r\le \frac{p_r-1}{2}.\, \text{ For } r<s\le\ell, \text{ if } s\in\operatorname{pol}(4n,a),\\
		& \text{ then } 1\le b_s\le p_s-2 ;\text{ if } s\notin\operatorname{pol}(4n,a),\text{ then } b_s= a_s\},
		\endaligned
	\end{equation*}
	\begin{equation*}
		\aligned
		\widehat{\Gamma}_{1,r}(4n,a):=&\{v(4n,b)\in \overline{\Gamma}_{1,r}(4n,a)\mid  \text{ for } \delta(4n)< s\le r-1, \\
		&\text{ if } s\in\operatorname{pol}(4n,a), \text{ then } b_s=1\text{ or } \frac{p_s+1}{2}\le b_s\le p_s-2;\\
		&\text{ if } s\notin\operatorname{pol}(4n,a), \text{ then } b_s=1\},
		\endaligned
	\end{equation*}
and $$\Gamma_{1,r}(4n,a):=\overline{\Gamma}_{1,r}(4n,a)\backslash\widehat{\Gamma}_{1,r}(4n,a).$$
	
	Let
	$$\widehat{\Gamma}_1(4n,a):=\bigcup\limits_{r\in\overline{\operatorname{pol}}(4n,a,\delta(4n))} \widehat{\Gamma}_{1,r}(4n,a),$$
	and 
	$$\Gamma_1(4n,a):=\bigcup\limits_{r\in\overline{\operatorname{pol}}(4n,a,\delta(4n))} \Gamma_{1,r}(4n,a).$$

	Let 	
	
	\begin{equation*}
		\aligned
		\Gamma_{2,r}(4n,a):=&\{(b_1,\dots,b_\ell)_{4n}\in X^{4n}\mid b_1=(0,1).  \text{ If } 3\mid n, \text{ then } b_2=1. \\
		&\text{ If } \delta(4n)< s\le r-1, \text{ then } b_s=1\text{ or } \frac{p_s+1}{2}\le b_s\le p_s-2. \\
		&\frac{p_r+1}{2}\le b_r\le p_r-2.\text{ For } r<s\le\ell, \text{ if }  a_s=1, \\
		&\text{ then } 1\le b_s\le p_s-2;\text{ if }  a_s\ne 1,\text{ then } b_s=p_s- a_s\}.
		\endaligned
	\end{equation*}

	Let $$\Gamma_2(4n,a):=\bigcup\limits_{r\in\overline{\operatorname{pol}}(4n,a,\delta(4n))} \Gamma_{2,r}(4n,a).$$

	Let
	\begin{equation*}
		\aligned
		\Gamma_3(4n,a):=&\{(b_1,\dots,b_\ell)_{4n}\in X^{4n}\mid b_1=(0,1).  \text{ If } 3\mid n, \text{ then } b_2=1. \\
		& \text{ For } \delta(4n)< s< \tau(a), \text{ if } s\in\operatorname{pol}(4n,a),\text{ then } b_s=1\text{ or }\\
		& \frac{p_s+1}{2}\le b_s\le p_s-2;\text{ if } s\notin\operatorname{pol}(4n,a), \text{ then } b_s=1. \\
		&\text{ For } \tau(a)\le s\le \ell, \text{ if } s\in\operatorname{pol}(4n,a), \text{ then } 1\le b_s\le p_s-2;\\
		&\text{ if } s\notin\operatorname{pol}(4n,a), \text{ then } b_s= a_s \}.
		\endaligned
	\end{equation*}

	Let $$K_1(4n,a):=\prod\limits_{v\in \Gamma_1(4n,a)} v^{(-1)^{e(a)+1}},\quad\quad K_2(4n,a):=\prod\limits_{v\in \Gamma_2(4n,a)} v^{(-1)^{f(a)+1}}$$
 and 
	$$K_3(4n,a):=\prod\limits_{v\in \Gamma_3(4n,a)} v^{(-1)^{e(a)}}.$$
	
\end{emp}

\begin{lemma}\label{le.em2}
	Keep the assumptions in \ref{ss.midgt}. Then the following is true.
	\begin{enumerate}
		\item $\Gamma_i(4n,a)\subset B_{4n}$ for each $1\le i\le 3$.
		\item $\Gamma_i(4n,a)\cap \Gamma_j(4n,a)=\emptyset$ for all $1\le i<j\le 3$.
	\end{enumerate}
\end{lemma}

\begin{proof}\hfill

	\begin{enumerate}
		\item This is clear.
		
		\item $\Gamma_1(4n,a)\cap \Gamma_2(4n,a)=\Gamma_2(4n,a)\cap \Gamma_3(4n,a)=\emptyset$. 
		
		Let $v(4n,c)\in \Gamma_{1,r}(4n,a)$ for some $r\in\overline{\operatorname{pol}}(4n,a,\delta(4n))$, $v(4n,d)\in \Gamma_2(4n,a)$ and $v(4n,e)\in \Gamma_3(4n,a)$. Then $c_{\tau(a)}= e_{\tau(a)}=p_{\tau(a)}- d_{\tau(a)}$. So $v(4n,d)\ne v(4n,c)$ and $v(4n,d)\ne v(4n,e)$ in $\widehat{X^{4n}}$. 
		
		$\Gamma_1(4n,a)\cap \Gamma_3(4n,a)=\emptyset$. 
		
		Let $v(4n,c)\in\Gamma_{1,r}(4n,a)$ and $v(4n,d)\in\Gamma_3(4n,a)$. Then there exists $\delta(4n)<t<r$ such that $t\notin\operatorname{pol}(4n,a)$ and $\frac{p_t+1}{2}\le c_t\le p_t-2$. Notice that $d_t=1$. So $v(4n,c)\ne v(4n,d)$ in $\widehat{X^{4n}}$.

	\end{enumerate}
\end{proof}

\begin{proposition}\label{prop.barep1}
	Keep the assumptions in \ref{ss.midgt}. Then
	\begin{equation*}
		v(4n,a)=K_1(4n,a) K_2(4n,a) K_3(4n,a),
	\end{equation*}
	in $\widehat{X^{4n}}$, where both sides are viewed as elements in $\widehat{X^{4n}}$ under the quotient map.
\end{proposition}

\begin{proof}

	By applying Lemma \ref{lem.unr} to $v(4n,a)$ with $t:=\operatorname{max}\{\delta(4n)< s\le\operatorname{len}(4n, a)\mid  a_s=p_s-1\}$, we get that

	\begin{align*}
		v(4n,a)=&w_t(4n,a)^{(-1)^{\operatorname{ord}(t)}}L_1(4n,a) L_2(4n,a) L_3(4n,a)\\
		=&w_t(4n,a)^{(-1)^{\operatorname{ord}(t)}}\prod\limits_{v\in \Delta_1(4n,a)}v^{(-1)^{e(a)+1}} \prod\limits_{v\in \Delta_2(4n,a)}v^{(-1)^{f(a)+1}} \prod\limits_{v\in \Delta_3(4n,a)}v^{(-1)^{e(a)}}.
	\end{align*}

	By Proposition \ref{lemma.basic0}, we obtain that 
	\[w_t(4n,a)^{(-1)^{\operatorname{ord}(t)}}=\prod\limits_{v\in \operatorname{supp}(w_t(4n,a))} v^{(-1)^{e(a)}}.\]

	It is clear that $\overline{\Gamma}_1(4n,a)=\Delta_1(4n,a)$ and $\Gamma_2(4n,a)=\Delta_2(4n,a)$. Furthermore, we have the following relations.
	
	\begin{enumerate}

		\item $\Delta_1(4n,a)\cap \Delta_3(4n,a)=\widehat{\Gamma}_1(4n,a)$.

		$\widehat{\Gamma}_1(4n,a)\subset\Delta_1(4n,a)\cap \Delta_3(4n,a)$. Let $v(4n,c)\in\widehat{\Gamma}_{1,r}(4n,a)$ for some $r\in\overline{\operatorname{pol}}(4n,a,\delta(4n))$. Assume that $=1$. Then $v(4n,c)\in\Delta_{1,r}(4n,a)$. Let $q:=\operatorname{min}\{\delta(4n)<s<r\mid c_s\ne 1\}$. Then $v(4n,c)\in\Delta_{3,q}(4n,a)$.
		
		$\Delta_1(4n,a)\cap \Delta_3(4n,a)\subset\widehat{\Gamma}_1(4n,a)$. Let $v(4n,c)\in\Delta_{1,r}(4n,a)\cap \Delta_3(4n,a)$ for some $r\in\overline{\operatorname{pol}}(4n,a,\delta(4n))$.  Then $2\le c_r\le \frac{p_r-1}{2}$ and for each $\delta(4n)<s<r$, we have that  $\frac{p_s+1}{2}\le c_s\le p_s-2$ or $c_s=1$.	By the assumption $v(4n,c)\in \Delta_3(4n,a)$, we know that for each $\delta(4n)<q<r$, if $ a_q=1$, then $c_q=1$ . Hence $v(4n,c)\in \widehat{\Gamma}_{1,r}(4n,a)$.

		\item $(\Delta_3(4n,a)\backslash\widehat{\Gamma}_1(4n,a))\cup\operatorname{supp}(w_t(4n,a))=\Gamma_3(4n,a)$.
		
		$\Gamma_3(4n,a)\subset (\Delta_3(4n,a)\backslash\widehat{\Gamma}_1(4n,a))\cup\operatorname{supp}(w_t(4n,a))$. 
		
		Let $v(4n,c)\in\Gamma_3(4n,a)$. Let $D:=\{\delta(4n)<s<\tau(a)\mid c_s\ne 1\}$. If $D=\emptyset$, then $v(4n,c)\in\operatorname{supp}(w_t(4n,a))$. If $D\ne\emptyset$, let $r:=\operatorname{min}(D)$. Then $v(4n,c)\in\Delta_(3,r)(4n,a)$. If $v(4n,c)\in\widehat{\Gamma}_{1,u}(4n,a)$, then $\frac{p_u+1}{2}\le c_u\le p_u-2$. This is a contradiction. So $v(4n,c)\in(\Delta_3(4n,a)\backslash\widehat{\Gamma}_1(4n,a))$.

		$(\Delta_3(4n,a)\backslash\widehat{\Gamma}_1(4n,a))\cup \operatorname{supp}(w_t(4n,a))\subset\Gamma_3(4n,a)$. 
		
		It is clear that $\operatorname{supp}(w_t(4n,a))\subset\Gamma_3(4n,a)$. Let $v(4n,c)\in(\Delta_{3,r}(4n,a)\backslash\widehat{\Gamma}_1(4n,a))$ for some $r\in\overline{\operatorname{pol}}(4n,a,\delta(4n))$. If $\delta(4n)<s<r$, then $c_s=1$. $\frac{p_r+1}{2}\le c_r\le p_r-2$. For $r<s\le \ell$, if $s\in\operatorname{pol}(4n,a)$, then $1\le c_s\le p_s-2$. If $s\notin\operatorname{pol}(4n,a)$, then $c_s=a_s$. Let $D:=\{r<s<\tau(a)\mid 2\le c_s\le \frac{p_s-1}{2}\}$. If $D=\emptyset$, then $v(4n,c)\in\Gamma_3(4n,a)$. If $D\ne\emptyset$. Let $u:=\operatorname{min}(D)$. Then $v(4n,c)\in{\Gamma}_{1,u}(4n,a))$. This is a contradiction. Hence the inclusion follows.

		\item $(\Delta_3(4n,a)\backslash\widehat{\Gamma}_1(4n,a))\cap\operatorname{supp}(w_t(4n,a))=\emptyset$.
		
		Let $v(4n,c)\in\Delta_{3,r}(4n,a)$ for some $r\in\overline{\operatorname{pol}}(4n,a,\delta(4n))$ and $v(4n,d)\in\operatorname{supp}(w_t(4n,a))$. Then $ \frac{p_r+1}{2}\le c_r\le p_r-2$ and $ d_r=1$. So $v(4n,c)\ne v(4n,d)$ in $\widehat{X^{4n}}$.

	\end{enumerate}

	Notice that if $v\in\widehat{\Gamma}_1(4n,a)$, then  $\operatorname{multi}_v(L_1(4n,a))=-\operatorname{multi}_v(L_3(4n,a))$. So, $v(4n,a)=K_1(4n,a) K_2(4n,a) K_3(4n,a)$ in $\widehat{X^{4n}}$.
\end{proof}

\begin{emp}{\bf The case $\delta(4n)<\operatorname{len}(4n,a)<\ell$ and $2\le  a_{\tau(a)}\le\frac{p_{\tau(a)}-1}{2}$.}\label{ss.mile}
	Keep the assumptions about $n$ in \ref{ss.sf}. Let $a\in\mathbb{Z}$ be such that $\operatorname{gcd}(4n,a)=1,\epsilon_a=1$ and $2\le  a_{\tau(a)}\le\frac{p_{\tau(a)}-1}{2}$. Assume that $\operatorname{len}(4n,a)>\delta(4n)$ and $\overline{\operatorname{pol}}(4n,a,\delta(4n))\ne\emptyset$. See Lemma \ref{lemma.basic} for $\overline{\operatorname{pol}}(4n,a,\delta(4n))=\emptyset$. 
	Assume that $2\le \operatorname{len}(4n,a)<\ell$.
	
	We introduce the following notations. For $r\in\overline{\operatorname{pol}}(4n,a,\delta(4n))$, let 
	\begin{equation*}
		\aligned
		\overline{\Gamma}_{1,r}(4n,a):=&\{(b_1,\dots,b_\ell)_{4n}\in X^{4n}\mid b_1=(0,1).  \text{ If } 3\mid n, \text{ then } b_2=1. \\
		&\text{ if } \delta(4n)< s\le r-1, \text{ then } b_s=1\text{ or } \frac{p_s+1}{2}\le b_s\le p_s-2.\\
		&2\le b_r\le \frac{p_r-1}{2}.\, \text{ For } r<s\le\ell, \text{ if } s\in\operatorname{pol}(4n,a), \\
		& \text{ then } 1\le b_s\le p_s-2 ;\text{ if } s\notin\operatorname{pol}(4n,a),\text{ then } b_s= a_s\},
		\endaligned
	\end{equation*}
	\begin{equation*}
		\aligned
		\widehat{\Gamma}_{1,r}(4n,a):=&\{v(4n,b)\in \overline{\Gamma}_{1,r}(4n,a)\mid \text{ for } \delta(4n)< s\le r-1, \\		
		&\text{ if } s\in\operatorname{pol}(4n,a), \text{ then } b_s=1\text{ or } \frac{p_s+1}{2}\le b_s\le p_s-2;\\
		&\text{ if } s\notin\operatorname{pol}(4n,a), \text{ then } b_s=1\},
		\endaligned
	\end{equation*}
and $$\Gamma_{1,r}(4n,a):=\overline{\Gamma}_{1,r}(4n,a)\backslash\widehat{\Gamma}_{1,r}(4n,a),$$

Let	
	$$\widehat{\Gamma}_1(4n,a):=\bigcup\limits_{r\in\overline{\operatorname{pol}}(4n,a,\delta(4n))} \widehat{\Gamma}_{1,r}(4n,a),$$
	and
	$$ \Gamma_1(4n,a):=\bigcup\limits_{r\in\overline{\operatorname{pol}}(4n,a,\delta(4n))} \Gamma_{1,r}(4n,a).$$
	
	Let 	
	\begin{equation*}
		\aligned
		\overline{\Gamma}_{2,r}(4n,a):=&\{(b_1,\dots,b_\ell)_{4n}\in X^{4n}\mid b_1=(0,1).  \text{ If } 3\mid n, \text{ then } b_2=1. \\
		&\text{ If } \delta(4n)< s\le r-1, \text{ then } b_s=1\text{ or } \frac{p_s+1}{2}\le b_s\le p_s-2. \\
		&\frac{p_r+1}{2}\le b_r\le p_r-2.\text{ For } r<s\le\ell, \text{ if }  a_s=1,\\
		& \text{ then } 1\le b_s\le p_s-2;\text{ if }  a_s\ne 1,\text{ then } b_s=p_s- a_s\}.
		\endaligned
	\end{equation*}
and
	\begin{equation*}
		\aligned
		\widehat{\Gamma}_{2,r}(4n,a):=&\{(b_1,\dots,b_\ell)_{4n}\in \overline{\Gamma}_{2,r}(4n,a)\mid \text{ for } r<s<\tau(a), \text{ if } a_s=1, \\
		&\text{ then } b_s=1\text{ or } \frac{p_s+1}{2}\le b_s\le p_s-2; \text{ if } a_s=p_s-1,\\
		& \text{ then } b_s=1\}.
		\endaligned
	\end{equation*}

	Here, if $r=\tau-1$, then we define $\widehat{\Gamma}_{2,r}(4n,a):=\overline{\Gamma}_{2,r}(4n,a)$. Let
	 $$\Gamma_{2,r}(4n,a):=\overline{\Gamma}_{2,r}(4n,a)\backslash\widehat{\Gamma}_{2,r}(4n,a).$$

	Let
	$$\widehat{\Gamma}_2(4n,a):=\bigcup\limits_{r\in\overline{\operatorname{pol}}(4n,a,\delta(4n))} \widehat{\Gamma}_{2,r}(4n,a),$$
	and
	$$\Gamma_2(4n,a):=\bigcup\limits_{r\in\overline{\operatorname{pol}}(4n,a,\delta(4n))} \Gamma_{2,r}(4n,a).$$

	Let
	
	\begin{equation*}
		\aligned
		\overline{\Gamma}_3(4n,a):=&\{(b_1,\dots,b_\ell)_{4n}\in X^{4n}\mid b_1=(0,1).  \text{ If } 3\mid n, \text{ then } b_2=1. \\
		&\text{ for } \delta(4n)< s< \tau(a), \frac{p_s+1}{2}\le b_s\le p_s-2 \text{ or } b_s=1.\\
		&
		 \text{ For } \tau(a)\le s\le \ell,\text{ if } s\in\operatorname{pol}(4n,a), \text{ then } 1\le b_s\le p_s-2; \\
		 & \text{ if } s\notin\operatorname{pol}(4n,a), \text{ then } b_s= a_s \},
		\endaligned
	\end{equation*}
	and
	\begin{equation*}
		\aligned
		\widehat{\Gamma}_3(4n,a):=&\{(b_1,\dots,b_\ell)_{4n}\in \overline{\Gamma}_3(4n,a)\mid  \text{ for } \delta(4n)< s< \tau(a),  \\
		&\text{ if } s\in\operatorname{pol}(4n,a),\text{ then } b_s=1\text{ or } \frac{p_s+1}{2}\le b_s\le p_s-2;\\
		&\text{ if } s\notin\operatorname{pol}(4n,a), \text{ then } b_s=1 \}.
		\endaligned
	\end{equation*}

	Let
	
	 $$\Gamma_3(4n,a):=\overline{\Gamma}_3(4n,a)\backslash\widehat{\Gamma}_3(4n,a).$$

	\begin{equation*}
	\aligned
	\overline{\Gamma}_4(4n,a):=&\{(b_1,\dots,b_\ell)_{4n}\in X^{4n}\mid b_1=(0,1).  \text{ If } 3\mid n, \text{ then } b_2=1. \\
	&\text{ if } \delta(4n)< s<\tau(a), \text{ then } b_s=1\text{ or } \frac{p_s+1}{2}\le b_s\le p_s-2. \\
	 &\text{ For } \tau(a)\le s\le\ell, \text{ if }  a_s=1, \text{ then } 1\le b_s\le p_s-2;\\
	&\text{ if }  a_s\ne 1,\text{ then } b_s=p_s- a_s\}.
	\endaligned
\end{equation*}
	
	Let
	
	 $$\Gamma_4(4n,a):=\overline{\Gamma}_4(4n,a)\backslash\widehat{\Gamma}_2(4n,a).$$
	
Let	

	\begin{align*}
		K_1(4n,a)&:=\prod\limits_{v\in \Gamma_1(4n,a)} v^{(-1)^{e(a)+1}},\quad\quad K_2(4n,a):=\prod\limits_{v\in \Gamma_2(4n,a)} v^{(-1)^{f(a)+1}},\\
		K_3(4n,a)&:=\prod\limits_{v\in \Gamma_3(4n,a)} v^{(-1)^{e(a)+1}},\quad\quad K_4(4n,a):=\prod\limits_{v\in \Gamma_4(4n,a)} v^{(-1)^{f(a)}}.
	\end{align*}

\end{emp}

\begin{lemma}\label{le.em2-2}
	Keep the assumptions in \ref{ss.mile}. Then the following is true.
	\begin{enumerate}
		\item $\Gamma_i(4n,a)\subset B_{4n}$ for each $1\le i\le 4$.
		\item $\Gamma_i(4n,a)\cap \Gamma_j(4n,a)=\emptyset$ for all $1\le i<j\le 4$.
	\end{enumerate}
\end{lemma}

\begin{proof}

	\begin{enumerate}
		\item This is clear.
		
		\item Let $v(4n,c)\in \Gamma_{1,r}(4n,a)$ for some $r\in\overline{\operatorname{pol}}(4n,a,\delta(4n))$, $v(4n,d)\in \Gamma_{2,u}(4n,a)$ for some $u\in\overline{\operatorname{pol}}(4n,a,\delta(4n))$, $v(4n,f)\in \Gamma_3(4n,a)$ and $v(4n,g)\in\Gamma_4(4n,a)$.
		
		$(\Gamma_1(4n,a)\cup\Gamma_3(4n,a))\cap(\Gamma_2(4n,a)\cup\Gamma_4(4n,a))=\emptyset$. 
		
		Observe that $c_{\tau(a)}= f_{\tau(a)}=p_{\tau(a)}- d_{\tau(a)}=p_{\tau(a)}- g_{\tau(a)}$. So $v(4n,d)\ne v(4n,c)$ and $v(4n,d)\ne v(4n,e)$ in $\widehat{X^{4n}}$.
		
		$\Gamma_1(4n,a)\cap\Gamma_3(4n,a)=\emptyset$. 
		
		Because $v(4n,c)\in\Gamma_{1,r}(4n,a)$, we know that $2\le c_r\le \frac{p_r-1}{2}$. Notice that $f_r=1$ or $\frac{p_r+1}{2}\le f_r\le p_r-2$. Hence $v(4n,c)\ne v(4n,f)$ in $\widehat{X^{4n}}$.
		
		$\Gamma_2(4n,a)\cap\Gamma_4(4n,a)=\emptyset$.

		Because $v(4n,d)\in\Gamma_2(4n,a)$, there exists $r<t<\tau(a)$ such that $2\le  d_t\le\frac{p_t-1}{2}$. Notice that $g_t=1$ or $\frac{p_t+1}{2}\le g_t\le p_t-2$. So $v(4n,d)\ne v(4n,g)$ in $\widehat{X^{4n}}$.  
		
	\end{enumerate}
\end{proof}

\begin{proposition}\label{prop.barep1-2}
	Keep the assumptions in \ref{ss.mile}. Then
	\begin{equation*}
		v(4n,a)=K_1(4n,a) K_2(4n,a) K_3(4n,a) K_4(4n,a),
	\end{equation*}
	in $\widehat{X^{4n}}$, where both sides are viewed as elements in $\widehat{X^{4n}}$ under the quotient map.
\end{proposition}

\begin{proof}
	The proof is similar to the proof of Proposition \ref{prop.barep1}. By applying Lemma \ref{lem.unr} to $v(4n,a)$ with $t:=\operatorname{max}\{\delta(4n)< s\le\operatorname{len}(4n, a)\mid a_s=p_s-1\}$, we get that

	\begin{align*}
		v(4n,a)=&w_t(4n,a)^{(-1)^{\operatorname{ord}(t)}}L_1(4n,a) L_2(4n,a) L_3(4n,a)\\
		=&w_t(4n,a)^{(-1)^{\operatorname{ord}(t)}}\prod\limits_{v\in \Delta_1(4n,a)}v^{(-1)^{e(a)+1}} \prod\limits_{v\in \Delta_2(4n,a)}v^{(-1)^{f(a)+1}} \prod\limits_{v\in \Delta_3(4n,a)}v^{(-1)^{e(a)}}.
	\end{align*}

	Assume that $w_t(4n,a)=v(4n,b)$ for some $b$ such that $\epsilon_b =1$. Let $\Gamma_1(4n,b)$ and $\Gamma_2(4n,b)$ be as introduced in \ref{em.fl}. By Lemma \ref{lemma.basic}, we obtain that \[W:=w_t(4n,a)^{(-1)^{\operatorname{ord}(t)}}=\prod\limits_{v\in \Gamma_1(4n,b)} v^{(-1)^{e(a)+1}}\prod\limits_{v\in \Gamma_2(4n,b)} v^{(-1)^{f(a)}}.\] Here we used the relations  $e(b)=e(a)-\operatorname{ord}(t)$ and $f(b)=f(a)+\operatorname{ord}(t)$.

	It is clear that the following equalities hold:  $\overline{\Gamma}_1(4n,a)=\Delta_1(4n,a)$,
	$\overline{\Gamma}_2(4n,a)=\Delta_2(4n,a)$, $\overline{\Gamma}_3(4n,a)=\Gamma_2(4n,b)$, and $\overline{\Gamma}_4(4n,a)=\Gamma_1(4n,b)$. Furthermore, we have the following relations.

	\begin{enumerate}

		\item $\Delta_1(4n,a)\cap \Delta_3(4n,a)=\widehat{\Gamma}_1(4n,a)$.
		
		$\Delta_1(4n,a)\cap \Delta_3(4n,a)\subset\widehat{\Gamma}_1(4n,a)$. 
		
		Let $v(4n,c)\in\Delta_{1,r}\cap\Delta_3(4n,a)$ for some $r\in\overline{\operatorname{pol}}(4n,a,\delta(4n))$. Because $v(4n,d)\in\Delta_{1,r}(4n,a)$, it follows that for each $\delta(4n)<s<r$, we have that  $c_s=1$ or $\frac{p_s+1}{2}\le c_s\le p_s-2$. Furthermore, since $v(4n,c)\in\Delta_3(4n,a)$, we have that  $c_s=1$ for each $\delta(4n)<s<r$ with $s\notin\operatorname{pol}(4n,a)$. Hence $v(4n,c)\in\widehat{\Gamma}_1(4n,a)$.
		
		$\widehat{\Gamma}_1(4n,a)\subset\Delta_1(4n,a)\cap \Delta_3(4n,a)$. Let $v(4n,c)\in\widehat{\Gamma}_{1,r}(4n,a)$ for some $r\in\overline{\operatorname{pol}}(4n,a,\delta(4n))$. Then $v(4n,c)\in\Delta_{1,r}(4n,a)$. Let $t:=\operatorname{min}\{\delta(4n)<s<r\mid c_s\ne 1\}$. Then $v(4n,c)\in\Delta{3,s}$.
		
		\item $\Delta_2(4n,a)\cap \Gamma_2(4n,b)=\widehat{\Gamma}_2(4n,a)$.
		
		$\Delta_2(4n,a)\cap \Gamma_2(4n,b)\subset\widehat{\Gamma}_2(4n,a)$. Let $v(4n,c)\in\Delta_{2,r}(4n,a)\cap \Gamma_2(4n,b)$ for some $r\in\overline{\operatorname{pol}}(4n,a,\delta(4n))$. Then if $\delta(4n)< s\le r-1, $ then $  c_s=1$ or $\frac{p_s+1}{2}\le c_s\le p_s-2$. $\frac{p_r+1}{2}\le c_r\le p_r-2$.  
		For $ r<s\le\ell, $ if $ a_s=1, $ then $ 1\le c_s\le p_s-2$. If $  a_s\ne 1,$ then $ c_s=p_s- a_s$. Since $v(4n,c)\in\Gamma_2(4n,b)$, so for each $r<s<\tau(a)$ such that $a_s=1$, we additional have $\frac{p_s+1}{2}\le c_s\le p_s-2$ or $c_s=1$. Thus $v(4n,c)\in\widehat{\Gamma}_2(4n,a)$.

		$\widehat{\Gamma}_2(4n,a)\subset\Delta_2(4n,a)\cap \Gamma_2(4n,b)$. Let $v(4n,c)\in\widehat{\Gamma}_{2,r}(4n,a)$ for some $r\in\overline{\operatorname{pol}}(4n,a,\delta(4n))$. Then $v(4n,c)\in\Delta_{2,r}(4n,a)$. It is clear that $v(4n,c)\in\Gamma_2(4n,b)$.

		\item $\Delta_3(4n,a)\cap\Gamma_1(4n,b)=\widehat{\Gamma}_3(4n,a)$. 
		
		$\Delta_3(4n,a)\cap\Gamma_1(4n,b)\subset\widehat{\Gamma}_3(4n,a)$. Let $v(4n,c)\in\Delta_{3,r}(4n,a)\cap\Gamma_1(4n,b)$ for some $r\in\overline{\operatorname{pol}}(4n,a,\delta(4n))$. Because $v(4n,c)\in\Delta_{3,r}(4n,a)$, if $\delta(4n)<s<r$, then $c_s=1$. For $r<s<\ell$, if $s\in\operatorname{pol}(4n,a)$, then $1\le c_s\le p_s-2$. If $s\notin\operatorname{pol}(4n,a)$, then $c_s=1$. Furthermore, because $v(4n,c)\in\Gamma_1(4n,b)$, if $r<s<\tau(a)$, we have that  $\frac{p_r+1}{2}\le c_s\le p_s-2$ or $c_s=1$. So $v(4n,c)\in\widehat{\Gamma}_3(4n,a)$.
		
		$\widehat{\Gamma}_3(4n,a)\subset\Delta_3(4n,a)\cap\Gamma_1(4n,b)$. Let $v(4n,c)\in\widehat{\Gamma}_3(4n,a)$. It is clear that $v(4n,c)\in\Gamma_1(4n,b)$. Let $r:=\operatorname{min}\{\delta(4n)<s<\tau(a)\mid c_s\ne 1\}$. Then $v(4n,c)\in\Delta_{3,r}(4n,a)$.

		\item $\widehat{\Gamma}_1(4n,a)\cap \widehat{\Gamma}_3(4n,a)=\emptyset$.
		
		Let $v(n,c)\in\widehat{\Gamma}_{1,r}(4n,a)$ for some $r\in\overline{\operatorname{pol}}(4n,a,\delta(4n))$ and $v(n,d)\in\widehat{\Gamma}_3(4n,a)$. Then $2\le c_r\le\frac{p_r-1}{2}$ and $ d_r=1$ or $\frac{p_r+1}{2}\le d_r\le p_r-2$. So $v(4n,c)\ne v(4n,d)$ in $\widehat{X^{4n}}$.

		\item $\widehat{\Gamma}_1(4n,a)\cup \widehat{\Gamma}_3(4n,a)=\Delta_3(4n,a)$.

		$\widehat{\Gamma}_1(4n,a)\cup \widehat{\Gamma}_3(4n,a)\subset\Delta_3(4n,a)$. Let $v(4n,c)\in\widehat{\Gamma}_{1,r}(4n,a)$ for some $r\in\overline{\operatorname{pol}}(4n,a,\delta(4n))$. Let $u:=\operatorname{min}\{\delta(4n)<s<r\mid c_s\ne 1\}$. Then $v(4n,a)\in\Delta_{3,u}(4n,a)$. Let $v(4n,f)\in\widehat{\Gamma}_3(4n,a)$. Let $q:=\operatorname{min}\{\delta(4n)<s<\tau(a)\mid c_s\ne 1\}$. Then $v(4n,f)\in\Delta_{3,q}(4n,a)$.
		
		$\Delta_3(4n,a)\subset\widehat{\Gamma}_1(4n,a)\cup \widehat{\Gamma}_3(4n,a)$. Let $v(4n,c)\in\Delta_{3,r}(4n,a)$ for some $r\in\overline{\operatorname{pol}}(4n,a,\delta(4n))$. Let $D:=\{r<s<\tau(a)\mid 2\le c_s\le\frac{p_s-1}{2}\}$. If $D=\emptyset$, then $v(4n,c)\in\widehat{\Gamma}_3(4n,a)$. If $D\ne\emptyset$, let $u:=\operatorname{min}(D)$. Then $v(4n,c)\in\widehat{\Gamma}_{1,u}(4n,a)$.

	\end{enumerate}

	Moreover, if $v\in\widehat{\Gamma}_1(4n,a)$, then $\operatorname{multi}_v(L_1(4n,a))=-\operatorname{multi}_v(L_3(4n,a))$. If $v\in\widehat{\Gamma}_2(4n,a)$, then $\operatorname{multi}_v(L_2(4n,a))=-\operatorname{multi}_v(W)$. If $v\in\widehat{\Gamma}_3(4n,a)$, then $\operatorname{multi}_v(L_3(4n,a))=-\operatorname{multi}_v(W)$.
	
	To sum up, we get that $v(4n,a)=K_1(4n,a) K_2(4n,a) K_3(4n,a) K_4(4n,a)$.
\end{proof}

\begin{emp}{\bf $\operatorname{len}(4n,a)=\ell$ and $\ell$ is even.}\label{ss.fuev}
	Keep the assumptions about $n$ in \ref{ss.sf}. Let $a\in\mathbb{Z}$ be such that $\operatorname{gcd}(4n,a)=1$ and $\epsilon_a=1$. Assume that $\operatorname{pol}(4n,a,\delta(4n))\ne\emptyset$. If $\operatorname{pol}(4n,a,\delta(4n))=\emptyset$, then $v(n,a)\in B_{4n}$.
	Assume that $\operatorname{len}(4n,a)=\ell$ and $\ell$ is even. We introduce the following notations. For $r\in\operatorname{pol}(4n,a,\delta(4n))$, let 
	\begin{equation*}
		\aligned
		\overline{\Gamma}_{1,r}(4n,a):=&\{(b_1\dots,b_\ell)_{4n}\in X^{4n}\mid b_1=(0,1).  \text{ If } 3\mid n, \text{ then } b_2=1. \\
		&\text{ If } \delta(4n)<s<r, \text{ then } b_s=1 \text{ or } \frac{p_s+1}{2}\le b_s\le p_s-2.\\
		& 2\le b_r\le \frac{p_r-1}{2}.\text{ For } s>r,\text{ if } s\in\operatorname{pol}(4n,a),\\
		&\text{ then } 1\le b_s\le p_s-2; \text{ if } s\notin\operatorname{pol}(4n,a), \text{ then } b_s=1\},\\
		\endaligned
	\end{equation*}
	\begin{equation*}
		\aligned
		\widehat{\Gamma}_{1,r}(4n,a):=&\{(b_1\dots,b_\ell)_{4n}\in \overline{\Gamma}_{1,r}(4n,a)\mid  \text{ for } \delta(4n)<s<r, \\
		&\text{ if } s\in\operatorname{pol}(4n,a),  \text{ then } b_s=1\text{ or } \frac{p_s+1}{2}\le b_s\le p_s-2;\\
		&\text{ if } s\notin\operatorname{pol}(4n,a), \text{ then } b_s=1\},
		\endaligned
	\end{equation*}
and $$\Gamma_{1,r}(4n,a):=\overline{\Gamma}_{1,r}(4n,a)\backslash\widehat{\Gamma}_{1,r}(4n,a).$$
	
	Let
	$$\widehat{\Gamma}_1(4n,a):=\bigcup\limits_{r\in\operatorname{pol}(4n,a,\delta(4n))} \widehat{\Gamma}_{1,r}(4n,a),$$
	and
	$$ \Gamma_1(4n,a):=\bigcup\limits_{r\in\operatorname{pol}(4n,a,\delta(4n))} \Gamma_{1,r}(4n,a).$$
	
	Let 
	
	\begin{equation*}
		\aligned
		\overline{\Gamma}_{2,r}(4n,a):=&\{(b_1,\dots,b_\ell)_{4n}\in X^{4n}\mid b_1=(0,1).  \text{ If } 3\mid n, \text{ then } b_2=1. \\
		&\text{ If } \delta(4n)< s<r, \text{ then } b_s=1 \text{ or } \frac{p_s+1}{2}\le b_s\le p_s-2.\\
		& \frac{p_r+1}{2}\le b_r\le p_r-2. \text{ For } s>r, \text{ if } s\in\operatorname{pol}(4n,a),\\
		& \text{ then } b_s=1;\text{ if } s\notin\operatorname{pol}(4n,a), \text{ then } 1\le b_s\le p_s-2\}.
		\endaligned
	\end{equation*}

	Let 
	\begin{equation*}
		\aligned
		\widehat{\Gamma}_2(4n,a):=&\{(b_1,\dots,b_\ell)_{4n}\in X^{4n}\mid b_1=(0,1).  \text{ If } 3\mid n, \text{ then } b_2=1. \\
		&\text{ for } \delta(4n)<s\le\ell\text{ if } s\in\operatorname{pol}(4n,a), \text{ then } b_s=1 \text{ or }\\
		& \frac{p_s+1}{2}\le b_s\le p_s-2; \text{ if } s\notin\operatorname{pol}(4n,a),\\
		&\text{ then } b_s=1\}\backslash \{((0,1),1,\dots,1)_{4n}\},
		\endaligned
	\end{equation*}
 $$\overline{\Gamma}_2(4n,a):=\bigcup\limits_{r\in\operatorname{pol}(4n,a,\delta(4n))} \Gamma_{2,r}(4n,a),$$
	and
	$$\Gamma_2(4n,a):=\overline{\Gamma}_2(4n,a)\backslash \widehat{\Gamma}_2(4n,a).$$

Let $$K_1(4n,a):=\prod\limits_{v\in \Gamma_1(4n,a)}v^{(-1)^{e(a)+1}},\quad\quad K_2(4n,a):=\prod\limits_{v\in \Gamma_2(4n,a)}v^{(-1)^{e(a)+1}},$$ and $$K_3(4n,a):=((0,1),1,\dots,1)_{4n}^{(-1)^{e(a)}}.$$
	
\end{emp}

\begin{lemma}\label{le.I12}
	Keep the assumptions in \ref{ss.fuev}. The following is true.
	\begin{enumerate}
		\item $\Gamma_1(4n,a), \Gamma_2(4n,a)\subset B_{4n}$ and $((0,1),1,\dots,1)_{4n}\in B_{4n}$.
		\item $((0,1),1,\dots,1)_{4n}\notin \Gamma_1(4n,a)\cup \Gamma_2(4n,a)$ and $\Gamma_1(4n,a)\cap \Gamma_2(4n,a)=\emptyset$. 
	\end{enumerate}
\end{lemma}

\begin{proof}

	\begin{enumerate}
		\item This is clear.
		
		\item  The first claim is clear. We show the second relation.
		
		Assume that $r<k$ where $r,k\in\operatorname{pol}(4n,a,\delta(4n))$.

		Let $v(4n,c)\in \Gamma_{1,k}(4n,c)$ and $v(4n,d)\in \Gamma_{2,r}(4n,a)$. Then $2\le c_k\le \frac{p_k-1}{2}$ and $ d_k=1$, so $v(4n,c)\ne v(4n,d)$ in $\widehat{X^{4n}}$. 
		
		Let $v(4n,c)\in \Gamma_{1,r}(4n,c)$ and $v(4n,d)\in \Gamma_{2,k}(4n,a)$. Then 
		$2\le c_r\le \frac{p_r-1}{2}$ and $\frac{p_r+1}{2} \le  d_r\le p_r-2$ or $ d_r=1$. So $v(4n,c)\ne v(4n,d)$ in $\widehat{X^{4n}}$. 
		
		Let $v(4n,c)\in \Gamma_{1,r}(4n,c)$ and $v(4n,d)\in \Gamma_{2,r}(4n,a)$. Then
		$2\le c_r\le \frac{p_r-1}{2}$ and $\frac{p_r+1}{2} \le  d_r\le p_r-2$. So $v(4n,c)\ne v(4n,d)$ in $\widehat{X^{4n}}$. 
		
		To summarize, we get that $\Gamma_1(4n,a)\cap \Gamma_2(4n,a)=\emptyset$.
		
	\end{enumerate}
\end{proof}

\begin{proposition}\label{prop.sfe}
	Keep the assumptions in \ref{ss.fuev}. Then
	
	\begin{equation*}
		v(4n,a)=K_1(4n,a) K_2(4n,a) K_3(4n,a),
	\end{equation*}
	in $\widehat{X^{4n}}$.
\end{proposition}

\begin{proof}
	
	By applying Lemma \ref{lem.unr} to $v(4n,a)$ with $t:=\operatorname{max}\{\delta(4n)< s\le\operatorname{len}(4n, a)\mid  a_s=p_s-1\}$, we get that

	\begin{align*}
	v(4n,a)=&w_t(4n,a)^{(-1)^{\operatorname{ord}(t)}}L_1(4n,a) L_2(4n,a) L_3(4n,a)\\
	=&w_t(4n,a)^{(-1)^{\operatorname{ord}(t)}}\prod\limits_{v\in \Delta_1(4n,a)}v^{(-1)^{e(a)+1}} \prod\limits_{v\in \Delta_2(4n,a)}v^{(-1)^{f(a)+1}} \prod\limits_{v\in \Delta_3(4n,a)}v^{(-1)^{e(a)}}.
	\end{align*}

	Note that 
	
	\[w_t(4n,a)^{(-1)^{\operatorname{ord}(t)}}=v(4n,1)^{e(a)}=K_3(4n,a).\]

	It is clear that the following equalities hold:  $\overline{\Gamma}_1(4n,a)=\Delta_1(4n,a)$, and
	$\overline{\Gamma}_2(4n,a)=\Delta_2(4n,a)$.	Furthermore, we have the following relations.

	\begin{enumerate}[label=(\roman*)]

		\item\label{it.sfe2} $\Delta_1(4n,a)\cap \Delta_3(4n,a)=\widehat{\Gamma}_1(4n,a)$.
		
		$\widehat{\Gamma}_1(4n,a)\subset \Delta_1(4n,a)\cap \Delta_3(4n,a)$. Let $v(4n,c)\in\widehat{\Gamma}_{1,r}(4n,a)$ for some $r\in\operatorname{pol}(4n,a,\delta(4n))$, then $v(4n,c)\in \Delta_{1,r}(4n,a)$. Let $u:=\operatorname{min}\{\delta(4n)<s<r\mid c_s=p_s-1\}$. Then $v(4n,c)\in \Delta_{3,u}(4n,a)$.
			
		$\Delta_1(4n,a)\cap \Delta_3(4n,a)\subset\widehat{\Gamma}_1(4n,a)$.  Let $v(4n,c)\in \Delta_1(4n,a)\cap \Delta_3(4n,a)$. Let $(c_1,\dots,c_\ell)_{4n}$ be the residue form of $v(4n,c)$. Assume that $v(4n,c)\in \Delta_{1,r}(4n,a)$ for some $r\in\operatorname{pol}(4n,a,\delta(4n))$. Then $2\le c_s\le \frac{p_s+1}{2}$ and for each $j<i$, we have that  $\frac{p_s-1}{2}\le c_j\le p_j-2$ or $c_j=1$.	Since $v(4n,c)\in \Delta_3(4n,a)$, so $c_s=1$ if $a_s=1$. So $v(4n,c)\in \widehat{\Gamma}_1(4n,a)$.

		\item\label{it.sfe3}	$\Delta_2(4n,a)\cap \Delta_3(4n,a)=\widehat{\Gamma}_2(4n,a)$.
			
			$\widehat{\Gamma}_2(4n,a)\subset \Delta_2(4n,a)\cap \Delta_3(4n,a)$. Let $v(4n,c)\in\widehat{\Gamma}_2(4n,a)$. Let $r:=\operatorname{min}\{\delta(4n)< s\le\ell\mid c_s\ne 1\}$. Then $v(4n,c)\in \Delta_{3,r}(4n,a)$. Let $u:=\operatorname{max}\{1\le s\le\ell\mid c_s\ne 1\}$. Then $v(4n,c)\in \Delta_{3,u}(4n,a)$.

			$\Delta_2(4n,a)\cap \Delta_3(4n,a)\subset\widehat{\Gamma}_2(4n,a)$. Assume that $v(4n,c)\in \Delta_2(4n,a)\cap \Delta_3(4n,a)$. Assume that $v(4n,c)\in \Delta_{3,r}(4n,a)$ for some $r\in\operatorname{pol}(4n,a,\delta(4n))$.  Then if $\delta(4n)<s<r$, then $c_s=1$. $\frac{p_r+1}{2}\le c_r\le p_r-2$. For $r<s\le \ell$, if $s\in\operatorname{pol}(4n,a)$, then $1\le c_s\le p_s-2$. If $s\notin\operatorname{pol}(4n,a)$, then $c_s=1$. Since $v(4n,c)\in \Delta_2(4n,a)$, then for each $s$ such that $ a_s=p_s-1$, we have that  $\frac{p_s+1}{2}\le c_s\le p_s-2$ or $c_s=1$. So $v(4n,c)\in \widehat{\Gamma}_2(4n,a)$.

		\item\label{it.sfe5} $\Delta_3(4n,a)=\widehat{\Gamma}_1(4n,a)\cup \widehat{\Gamma}_2(4n,a)$. 
		
		$\widehat{\Gamma}_1(4n,a)\cup \widehat{\Gamma}_2(4n,a)\subset \Delta_3(4n,a)$. This is clear.

		$\Delta_3(4n,a)\subset\widehat{\Gamma}_1(4n,a)\cup \widehat{\Gamma}_2(4n,a)$.	Let $v(4n,c)\in \Delta_{3,r}(4n,a)$ for some $r\in\operatorname{pol}(4n,a,\delta(4n))$. So if $2<s<r$, then $c_s=1$. For $r\le s<\ell$, if $s\in\operatorname{pol}(4n,a)$, then $1\le c_s\le p_s-2$. If $s\notin\operatorname{pol}(4n,a)$, then $c_s=1$. Let $D:=\{\delta(4n)< s\le\ell\mid 2\le c_j\le \frac{p_s-1}{2}\}$. If $D=\emptyset$, then $v(4n,c)\in\widehat{\Gamma}_2(4n,a)$. If $D\ne\emptyset$, let $u:=\operatorname{min}(D)$. Then $v(4n,c)\in \widehat{\Gamma}_{1,u}(4n,a)$.

		\item $\widehat{\Gamma}_1(4n,a)\cap \widehat{\Gamma}_2(4n,a)=\emptyset$. 
		
		Let $v(4n,c)\in\widehat{\Gamma}_{1,r}(4n,a)$ for some $r\in\operatorname{pol}(4n,a,\delta(4n))$ and let $v(4n,d)\in\widehat{\Gamma}_2(4n,a)$. Then $2\le c_r\le\frac{p_r-1}{2}$ and $d_r=1$ or $\frac{p_s+1}{2}\le d_r\le p_r-2$. So $v(4n,c)\ne v(4n,d)$ in $\widehat{X^{4n}}$. 
		
	\end{enumerate}

	Moreover, if $v\in \widehat{\Gamma}_1(4n,a)$, then $\operatorname{multi}_v(L_1(4n,a))=-\operatorname{multi}_v(L_3(4n,a))$. If $v\in \widehat{\Gamma}_2(4n,a)$, then $\operatorname{multi}_v(L_2(4n,a))=-\operatorname{multi}_v(L_3(4n,a))$. 
	
	To sum up, we obtain that $v(4n,a)=K_1(4n,a) K_2(4n,a) K_3(4n,a)$.
\end{proof}

Our next goal is to address the case $\operatorname{len}(4n,a)=\ell$ and $\ell$ is odd. To do so, we need the following lemma. Let
\begin{equation*}
	\aligned
	\Gamma(4n):=&\{(b_1,\dots,b_\ell)_{4n}\in X^{4n}\mid \text{ for } 1\le s\le \ell, b_s=1 \text{ or }\\
	& \frac{p_s+1}{2}\le b_s\le p_s-2\}\backslash \{((0,1),1,\dots,1)_{4n}\}.
	\endaligned
\end{equation*}

Notice that $\Gamma(4n)\subset B_{4n}$.

\begin{lemma}\label{le.odd1}
	Let $n$ be as in \ref{ss.sf}. Assume that $\ell$ is odd. Let $\Gamma(4n)$ be as above. Then $$v(4n,1)=\prod\limits_{v\in \Gamma(4n)}v^{-1}$$
	in $\widehat{X^{4n}}$.
\end{lemma}

\begin{proof}

	By applying Lemma \ref{lem.unr} to $v(4n,-1)$ with $t:=\ell$, we get that

	\begin{equation*}
		\aligned
		v(4n,n-1)=&w_t(4n,-1)^{(-1)^{\operatorname{ord}(t)}}L_1(4n,-1) L_2(4n,-1) L_3(4n,-1)\\
		=&v(4n,1)^{-1}\prod\limits_{v\in \Delta_1(4n,-1)}v\prod\limits_{v\in \Delta_2(4n,-1)}v^{-1}\prod\limits_{v\in \Delta_3(4n,-1)}v^{-1}.\\
		\endaligned
	\end{equation*}
	
	Furthermore, we have the following relations.
	
	\begin{enumerate}[label=(\roman*)]
		
		\item\label{it.sf1-1} $\Delta_2(4n,-1)=\Gamma(4n)$.
		
		It is clear that $\Delta_2(4n,-1)\subset \Gamma(4n)$. If $v(4n,c)\in \Gamma(4n)$, let $r:=\operatorname{max}\{\delta(4n)<s\le\ell\mid c_s\ne 1\}$. Then $v(4n,c)\in \Delta_{2,r}(4n,-1)$. So $\Gamma(4n)\subset\Delta_2(4n,-1)$.

		\item $\Delta_1(4n,-1)\subset\Delta_3(4n,-1)$.
		
		Assume that $v(4n,c)\in \Delta_{1,r}(4n,-1)$ for some $\delta(4n)+1<r\le \ell$. Let $u:=\operatorname{min}\{\delta(4n)<s<r\mid c_s\ne 1\}$. Then $v(4n,c)\in \Delta_{3,u}(4n,-1)$. 
		
		\item $\Delta_3(4n,-1)\backslash\Delta_1(4n,-1)=\Gamma(4n)$.

		$\Gamma(4n)\subset \Delta_3(4n,-1)$. Let $v(4n,c)\in \Gamma(4n)$. Let $r:=\operatorname{min}\{\delta(4n)<s\le \ell\mid c_s\ne 1\}$, it is clear that $v(4n,c)\in \Delta_{3,r}(4n,-1)$.

		$\Delta_1(4n,-1)\cap \Gamma(4n)=\emptyset$. Let $v(4n,c)\in \Delta_{1,r}(4n,-1)$ and $v(4n,d)\in \Delta_3(4n,-1)$, then $2\le c_r\le \frac{p_r-1}{2}$ and $\frac{p_r+1}{2}\le  d_r\le p_r-2$ or $ d_r=1$. So $\Delta_1(4n,-1)\cap \Gamma(4n)=\emptyset$. So $\Gamma(4n)\subset\Delta_3(4n,-1)\backslash\Delta_1(4n,-1)$.

		$\Delta_3(4n,-1)\backslash\Delta_1(4n,-1)\subset\Gamma(4n)$. This is equivalent to 	$\Delta_3(4n,-1)\subset\Delta_1(4n,-1)\cup \Gamma(4n)$. Assume that $v(4n,c)\in \Delta_3(4n,-1)$. Let $D:=\{\delta(4n)<s\le\ell\mid 2\le c_s\le \frac{p_s-1}{2} \}$. If $D=\emptyset$, then $v(4n,c)\in \Gamma(4n)$. If $D\ne\emptyset$, let $r:=\operatorname{min}(D)$. Then $v(4n,c)\in \Delta_{1,r}(4n,-1)$.

	\end{enumerate}

	Moreover, if $v\in\Delta_1(4n,a)$, then $\operatorname{multi}_v(L_1(4n,-1))=-\operatorname{multi}_v(L_3(4n,-1))$. Combining the above, we obtain that $v(4n,1)^2=\prod\limits_{v\in \Gamma(4n)}v^{-2}$. Therefore, $v(4n,1)=\prod\limits_{v\in \Gamma(4n)}v^{-1}$.
\end{proof}

\begin{emp}\label{em.sfoa}{\bf $\operatorname{len}(4n,a)=\ell$ and $\ell$ is odd.}
	Keep the assumptions about $n$ in \ref{ss.sf}. Let $a\in\mathbb{Z}$ be such that $\operatorname{gcd}(4n,a)=1$ and $\epsilon_a=1$. Assume that $\operatorname{pol}(4n,a,\delta(4n))\ne\emptyset$. See Lemma \ref{le.odd1} for $\operatorname{pol}(4n,a,\delta(4n))=\emptyset$. 
	Assume that $2\le \operatorname{len}(4n,a)=\ell$ and $\ell$ is odd. We introduce the following notations. For each $r\in\operatorname{pol}(4n,a,\delta(4n))$, let 
	\begin{equation*}
		\aligned
		\overline{\Gamma}_{1,r}(4n,a):=&\{(b_1\dots,b_\ell)_{4n}\in X^{4n}\mid b_1=(0,1).  \text{ If } 3\mid n, \text{ then } b_2=1. \\
		& \text{ If } \delta(4n)<s<r, \text{ then } b_s=1 \text{ or } \frac{p_s+1}{2}\le b_s\le p_s-2.\\
		& 2\le b_r\le \frac{p_r-1}{2}.\text{ For } s>r,\text{ if } s\in\operatorname{pol}(4n,a),\\
		&\text{ then } 1\le b_s\le p_s-2; \text{ if } s\notin\operatorname{pol}(4n,a), \text{ then } b_s=1\},
		\endaligned
	\end{equation*}
	\begin{equation*}
		\aligned
		\widehat{\Gamma}_{1,r}(4n,a):=&\{(b_1\dots,b_\ell)_{4n}\in \overline{\Gamma}_{1,r}(4n,a)\mid  \text{ for } \delta(4n)<s<r, \\
		&\text{ if } s\in\operatorname{pol}(4n,a),  \text{ then } b_s=1\text{ or }\frac{p_s+1}{2}\le b_s\le p_s-2;\\
		&  \text{ if } s\notin\operatorname{pol}(4n,a), \text{ then } b_s=1\},
		\endaligned
	\end{equation*}
and $$\Gamma_{1,r}(4n,a):=\overline{\Gamma}_{1,r}(4n,a)\backslash\widehat{\Gamma}_{1,r}(4n,a).$$
	
	Let
	$$\widehat{\Gamma}_1(4n,a):=\bigcup\limits_{r\in\operatorname{pol}(4n,a,\delta(4n))} \widehat{\Gamma}_{1,r}(4n,a),$$
	and
	$$\Gamma_1(4n,a):=\bigcup\limits_{r\in\operatorname{pol}(4n,a,\delta(4n))} \Gamma_{1,r}(4n,a).$$
	
	Let 
	\begin{equation*}
		\aligned
		\overline{\Gamma}_{2,r}(4n,a):=&\{(b_1,\dots,b_\ell)_{4n}\in X^{4n}\mid b_1=(0,1). \text{ If } 3\mid n, \text{ then } b_2=1. \\
		& \text{ If } \delta(4n)<s<r, \text{ then } b_s=1 \text{ or } \frac{p_s+1}{2}\le b_s\le p_s-2.\\
		& \frac{p_r+1}{2}\le b_r\le p_r-2.\text{ For } s>r, \text{ if } s\in\operatorname{pol}(4n,a)\\
		&\text{ then } b_s=1; \text{ if } s\notin\operatorname{pol}(4n,a),\text{ then } 1\le b_s\le p_s-2\},
		\endaligned
	\end{equation*}
	and
	\begin{equation*}
		\aligned
		\widehat{\Gamma}_{2,r}(4n,a):=&\{(b_1,\dots,b_\ell)_{4n}\in \overline{\Gamma}_{2,r}(4n,a)\mid \text{ for } s>r,\\
		& \text{ if } s\in\operatorname{pol}(4n,a), \text{ then } b_s=1; \text{ if } s\notin\operatorname{pol}(4n,a),\\
		&\text{ then } b_s=1\text{ or } \frac{p_s+1}{2}\le b_s\le p_s-2\}.
		\endaligned
	\end{equation*}

	Here, if $r=\ell$, then we define $\widehat{\Gamma}_{2,r}(4n,a):=\overline{\Gamma}_{2,r}(4n,a)$. Let
	 $$\Gamma_{2,r}(4n,a):=\overline{\Gamma}_{2,r}(4n,a)\backslash\widehat{\Gamma}_{2,r}(4n,a).$$
	
	Let
	$$\widehat{\Gamma}_2(4n,a):=\bigcup\limits_{r\in\operatorname{pol}(4n,a,\delta(4n))} \widehat{\Gamma}_{2,r}(4n,a),$$
	and
	$$\Gamma_2(4n,a):=\bigcup\limits_{r\in\operatorname{pol}(4n,a,\delta(4n))} \Gamma_{2,r}(4n,a).$$

	Let
	\begin{equation*}
		\aligned
		\Gamma_3(4n,a):=&\{(b_1,\dots,b_\ell)_{4n}\in X^{4n}\mid b_1=(0,1). \text{ If } 3\mid n, \\
		& \text{ then } b_2=1.  \text{ For } \delta(4n)< s\le \ell, \text{ if } s\in\operatorname{pol}(4n,a),\\
		&  \text{ then } b_s=1 \text{ or } \frac{p_s+1}{2}\le b_s\le p_s-2;\\
		& \text{ if } s\notin\operatorname{pol}(4n,a),\text{ then } b_s=1\}\backslash \{((0,1),1,\dots,1)_{4n}\},
		\endaligned
	\end{equation*}
	and 
	\begin{equation*}
		\aligned
		\Gamma_4(4n,a):=&\{(b_1,\dots,b_\ell)_{4n}\in X^{4n}\mid b_1=(0,1). \text{ If } 3\mid n, \\
		&\text{ then } b_2=1.   \text{ For } \delta(4n)< s\le \ell, \text{ if } s\notin\operatorname{pol}(4n,a),\\
		& \text{ then } b_s=1 \text{ or } \frac{p_s+1}{2}\le b_s\le p_s-2; \\
		&\text{ if } s\in\operatorname{pol}(4n,a),\text{ then } b_s=1\}\backslash \{((0,1),1,\dots,1)_{4n}\}.
		\endaligned
	\end{equation*}

	Let

	\begin{align*}
		K_1(4n,a)&:=\prod\limits_{v\in \Gamma_1(4n,a)}v^{(-1)^{e(a)+1}}, &&K_2(4n,a):=\prod\limits_{v\in \Gamma_2(4n,a)}v^{(-1)^{e(a)}},\\
		K_3(4n,a)&:=\prod\limits_{v\in \Gamma_3(4n,a)}v^{(-1)^{e(a)}}, &&K_4(4n,a):=\prod\limits_{v\in \Gamma_4(4n,a)}v^{(-1)^{e(a)+1}}.
	\end{align*}

\end{emp}

\begin{lemma}
	Keep the assumptions in \ref{em.sfoa}. Then the following is true.
	\begin{enumerate}
		\item $\bigcup\limits_{1\le i\le 4}\Gamma_i(4n,a)\subset B_{4n}$.
		\item $\Gamma_i(4n,a)\cap \Gamma_j(4n,a)=\emptyset$ for all $1\le i<j\le 4$.
	\end{enumerate}
\end{lemma}

\begin{proof}
	
	\begin{enumerate}
		\item This is clear.
		
		\item $\Gamma_1(4n,a)\cap \Gamma_2(4n,a)=\emptyset$.
		This is similar to the argument in Lemma \ref{le.I12}.
		
		$\Gamma_1(4n,a)\cap \Gamma_3(4n,a)=\Gamma_1(4n,a)\cap \Gamma_4(4n,a)=\emptyset$. 
		
		Let $v(4n,c)\in \Gamma_{1,r}(4n,a)$ for some $r\in\operatorname{pol}(4n,a,\delta(4n))$, $v(4n,d)\in \Gamma_3(4n,a)$ and $v(4n,e)\in \Gamma_4(4n,a)$ then $2\le c_r\le \frac{p_r-1}{2}$ and $\frac{p_r+1}{2}\le  d_r\le p_r-2$ or $ d_r=1$, $\frac{p_r+1}{2}\le  e_r\le p_r-2$ or $ d_r=1$ So $v(4n,c)\ne v(4n,d)$ and $v(4n,c)\ne v(4n,e)$ in $\widehat{X^{4n}}$.

		$\Gamma_3(4n,a)\cap \Gamma_4(4n,a)=\emptyset$.
		
		Since $v(4n,d)\ne ((0,1),1,\dots,1)_{4n}$, there exists $r\in\operatorname{pol}(4n,a)$ such that $\frac{p_r+1}{2}\le  d_r\le p_r-2$. Since $ e_r=1$, we get that $v(4n,d)=v(4n,e)$.

		$\Gamma_2(4n,a)\cap \Gamma_3(4n,a)=\emptyset$ and $\Gamma_2(4n,a)\cap \Gamma_4(4n,a)=\emptyset$.
		
		Let $v(4n,c)\in \Gamma_{2,r}(4n,a)$ and $v(4n,d)\in \Gamma_3(4n,a)$ and $v(4n,e)\in \Gamma_4(4n,a)$, then for some $k\notin\operatorname{pol}(4n,a)$ and $r< k\le \ell$, we have that  $2\le c_r\le \frac{p_r-1}{2}$, $\frac{p_r+1}{2}\le  d_r\le p_r-2$ or $ d_r=1$, $\frac{p_r+1}{2}\le  e_r\le p_r-2$ or $ e_r=1$. So $v(4n,c)\ne v(4n,d)$ and $v(4n,c)\ne v(4n,e)$ in $\widehat{X^{4n}}$.

	\end{enumerate}
\end{proof}

\begin{proposition}
	Keep the assumptions in \ref{em.sfoa}. Then
	
	\begin{equation*}
		v(4n,a)=K_1(4n,a) K_2(4n,a) K_3(4n,a) K_4(4n,a),
	\end{equation*}
	in $\widehat{X^{4n}}$, where both sides are viewed as elements in $\widehat{X^{4n}}$ under the quotient map.
\end{proposition}

\begin{proof}
	
	By applying Lemma \ref{lem.unr} to $v(4n,a)$ with $t:=\operatorname{max}\{\delta(4n)< s\le\ell\mid  a_s=p_s-1\}$, we get that

	\begin{align*}
		v(4n,a)=&w_t(4n,a)^{(-1)^{\operatorname{ord}(t)}}L_1(4n,a) L_2(4n,a) L_3(4n,a)\\
		=&w_t(4n,a)^{(-1)^{\operatorname{ord}(t)}}\prod\limits_{v\in \Delta_1(4n,a)}v^{(-1)^{e(a)+1}} \prod\limits_{v\in \Delta_2(4n,a)}v^{(-1)^{f(a)+1}} \prod\limits_{v\in \Delta_3(4n,a)}v^{(-1)^{e(a)}}.
	\end{align*}

	By Lemma \ref{le.odd1}, we derive that $$W:=w_t(4n,a)^{(-1)^{\operatorname{ord}(t)}}=v(4n,1)^{(-1)^{e(a)}}=\prod\limits_{v\in \Gamma(4n)}v^{(-1)^{e(a)+1}}.$$

	It is clear that the following equalities hold:  $\overline{\Gamma}_1(4n,a)=\Delta_1(4n,a)$, and
	$\overline{\Gamma}_2(4n,a)=\Delta_2(4n,a)$. Furthermore, we have the following relations.

	\begin{enumerate}[label=(\roman*)]

		\item\label{it.sfo2} $\Delta_1(4n,a)\cap \Delta_3(4n,a)=\widehat{\Gamma}_1(4n,a)$.
		
		The proof is similar to the proof of \ref{it.sfe2} in Proposition \ref{prop.sfe}.

		\item\label{it.sfo3} $\Delta_2(4n,a)\cap \Gamma(4n)=\widehat{\Gamma}_2(4n,a)$
		
		$\widehat{\Gamma}_2(4n,a)\subset \Delta_2(4n,a)\cap \Gamma(4n)$. Let $v(4n,c)\in\widehat{\Gamma}_{2,r}(4n,a)$, then $v(4n,c)\in \Delta_{2,r}(4n,a)$. It is clear that $v(4n,c)\in \Gamma(4n)$.
		
		$\Delta_2(4n,a)\cap \Gamma(4n)\subset\widehat{\Gamma}_2(4n,a)$. Let $v(4n,c)\in \Delta_{2,r}(4n,a)\cap \Gamma(4n)$ for some $r\in\operatorname{pol}(4n,a,\delta(4n))$. As $v(4n,c)\in \Delta_{2,r}$, we get that $\frac{p_s-1}{2}\le c_s\le p_s-2$ or $c_s=1$ if $1\le s<r$ and $\frac{p_r+1}{2}\le c_r\le p_r-2$. For $r<s\le\ell$, if $s\in\operatorname{pol}(4n,a)$, then $c_s=1$. If $s\notin\operatorname{pol}(4n,a)$, then $1\le c_s\le p_s-2$. Moreover, because $v(n,c)\in\Gamma(4n)$, we know that if $r<s\le\ell$ and $s\notin\operatorname{pol}(4n,a)$, we get that $\frac{p_s-1}{2}\le c_s\le p_s-2$ or $c_s=1$. So $v(4n,c)\in\widehat{\Gamma}_{2,r}(4n,a)$.
		
		\item $\Delta_3(4n,a)\backslash\widehat{\Gamma}_1(4n,a)=\Gamma_3(4n,a).$

		$\Gamma_3(4n,a)\subset\Delta_3(4n,a)\backslash\widehat{\Gamma}_1(4n,a).$		
		
		Let $v(4n,c)\in\Gamma_3(4n,a)$. Let $r:=\operatorname{min}\{\delta(4n)<s\le\ell\mid c_s\ne 1\}$. Then $v(4n,a)\in\Delta_{3,r}(4n,a)$. Let $v(4n,d)\in\widehat{\Gamma}_{1,t}(4n,a)$ for some $t\in\operatorname{pol}(4n,a,\delta(4n))$. Then $2\le  d_t\le\frac{p_t-1}{2}$, and $c_t=1$ or $\frac{p_t+1}{2}\le c_t\le p_t-2$. So $v(4n,c)\ne v(4n,d)$ in $\widehat{X^{4n}}$ and hence $v(4n,c)\notin\widehat{\Gamma}_1(4n,a)$. Thus $v(4n,c)\in\Delta_3(4n,a)\backslash\widehat{\Gamma}_1(4n,a)$.

		$\Delta_3(4n,a)\backslash\widehat{\Gamma}_1(4n,a)\subset\Gamma_3(4n,a).$
		
		Let $v(4n,c)\in\Delta_3(4n,a)\backslash\widehat{\Gamma}_1(4n,a)$. Because $v(4n,c)\in\Delta_3(4n,a)$, we know that if $s\in\operatorname{pol}(4n,a,\delta(4n))$, then $1\le c_s\le p_s-2$. If $s\notin\operatorname{pol}(4n,a,\delta(4n))$, then $c_s=1$. Let $D:=\{s\mid 2\le c_s\le\frac{p_s-1}{2}\}$. If $D=\emptyset$, then $v(4n,c)\in\Gamma_3(4n,a)$ and we are done. Assume $D\ne\emptyset$. Let $r:=\operatorname{min}(D)$. Then $v(4n,c)\in\widehat{\Gamma}_{1,r}(4n,a)$, a contradiction. 
		
		\item $\Gamma(4n)\backslash\widehat{\Gamma}_2(4n,a)=\Gamma_4(4n,a).$

		$\Gamma_4(4n,a)\subset\Gamma(4n)\backslash\widehat{\Gamma}_2(4n,a).$
		
		Let $v(4n,c)\in\Gamma_4(4n,a).$ It is clear that $v(4n,c)\in\Gamma(4n)$. Let $v(4n,d)\in\widehat{\Gamma}_{2,r}(4n,a)$ for some $r\in\operatorname{pol}(4n,a,\delta(4n))$. Then $c_r=1$ and $\frac{p_r+1}{2}\le  d_r\le p_r-2$. It follows that $v(4n,c)\ne v(4n,d)$ in $\widehat{X^{4n}}$. Hence $v(4n,c)\in\Gamma(4n)\backslash\widehat{\Gamma}_2(4n,a)$.
		
		$\Gamma(4n)\backslash\widehat{\Gamma}_2(4n,a)\subset\Gamma_4(4n,a).$
		
		Let $v(4n,c)\in\Gamma(4n)\backslash\widehat{\Gamma}_2(4n,a)$. Because $v(4n,c)\in\Gamma(4n)$, we know that $c_s=1$ or $\frac{p_s+1}{2}\le c_s\le p_s-2$ for each $\delta(4n)<s\le\ell$. Let $D:=\{s>\delta(4n)\mid s\in\operatorname{pol}(4n,a)\text{ and } c_s\ne 1\}.$ If $D=\emptyset$, then $v(4n,c)\in\Gamma_4(4n,a)$ and we are done. Assume that $D\ne\emptyset$. Let $r:=\operatorname{max}(D)$. Then $v(4n,c)\in\widehat{\Gamma}_{2,r}(4n,a)$, a contradiction.

	\end{enumerate}

	Moreover, if $v\in\widehat{\Gamma}_1(4n,a)$, then $\operatorname{multi}_v(L_1(4n,a))=-\operatorname{multi}_v(L_3(4n,a))$. If $v\in\widehat{\Gamma}_2(4n,a)$, then $\operatorname{multi}_v(L_2(4n,a))=-\operatorname{multi}_v(W)$. 
	
	To summarize, we get that $v(4n,a)=K_1(4n,a) K_2(4n,a) K_3(4n,a) K_4(4n,a)$.	
\end{proof}

\begin{remark}
	Let $n\in\mathbb{N}_{> 3}$ be odd, square-free and not a prime. If $3\mid n$, let $\delta(n):=1$ and if $3\nmid n$, let $\delta(n):=0$. Each statement of this section still holds if one replace $4n$ in $v(4n,a)$ by $n$. Only minor modifications are needed accordingly.
\end{remark}

\section{The Four times Squarefree case }\label{sec.sf1}

	Suppose that $n$ is odd, square-free and not a prime. The primary focus of this section is studying the solutions in $G$ to Equation \eqref{eq.ne4} while imposing the condition that the denominator of the solutions is $4n$. Our goal is to prove the following theorem.

\begin{theorem}\label{th.4msf}
	Let $n\in\mathbb{N}_{\ge 3}$ be odd, square-free and not a prime. Assume that each prime factor of $n$ is  greater than $11$. Assume that $(x_0,x_1,x_2,x_3,x_4)$ is a solution to \eqref{eq.ne4} in $G$ with $0<x_i<\frac{\pi}{2}$ and $\operatorname{den}(x_i)=4n$ for each $0\le i\le 4$. Then $(x_0,x_1,x_2,x_3,x_4)\in\Phi_{1,1}$.
\end{theorem}

We postpone the proof of Theorem \ref{th.4msf} to the end of this section.

\begin{lemma}\label{le.congsf4dn}
	Let $n\in\mathbb{N}$ satisfying $4\mid n$. Let $n=p_1^{e_1}\cdots p_\ell^{e_\ell}$ be the prime factorization of $n$. Let $a,b\in\mathbb{N}$ be such that $0< a, b< \frac{n}{2}$ and $\operatorname{gcd}(n,ab)=1$. Let $(\epsilon_a a)_s$ and $(\epsilon_a a)_s$ be the elements associated to $\epsilon_a a$ and $\epsilon_b b$ with respect to $4n$ as defined in \ref{ss.rf}. Then

	\begin{enumerate}
		\item $(\epsilon_a a)_s=(\epsilon_b b)_s$ for each $1\le s\le\ell$ if and only if $a=b$.
		\item $(\epsilon_a a)_s=(\epsilon_b b)_s$ for each $2\le s\le\ell$, $\widehat{(\epsilon_a a)}_1=\widehat{(\epsilon_b b)}_1$ and $\overline{(\epsilon_a a)}_1\ne \overline{(\epsilon_b b)}_1$ if and only if $a+b=\frac{n}{2}$.
		
	\end{enumerate}
\end{lemma}

\begin{proof}
	
	We only show $(ii)$. $(i)$ is similar. Note that $(\epsilon_b b)_s\equiv(\epsilon_b b+\frac{n}{2})_s$ for each $2\le s\le\ell$ and
	
	$\widehat{(\epsilon_b b)}_1=\widehat{(\epsilon_b b+\frac{n}{2})}_1$ and $\overline{(\epsilon_b b)}_1\ne \overline{(\epsilon_b b+\frac{n}{2})}_1$. So $(\epsilon_b b+\frac{n}{2})_s=(\epsilon_a a)_s$ for each $1\le s\le\ell$. By the Chinese Remainder Theorem, we get that $\epsilon_b b+\frac{n}{2}\equiv \epsilon_a a\operatorname{mod} n$. Hence $a\equiv b+\frac{n}{2}\operatorname{mod} n$ or $a\equiv -b+\frac{n}{2}\operatorname{mod} n$. Because $0< a, b< \frac{n}{2}$, it follows that $a+b=\frac{n}{2}$. 	
\end{proof}

\begin{emp}\label{em.sfproof4n}
	Let $n\in\mathbb{N}_{\ge 3}$ be odd, square-free and not a prime. Assume that every prime factor of $n$ is greater than $11$. Keep the notation convention in \ref{ss.sf}. Let $(x_0,x_1,x_2,x_3,x_4)$ be a solution to \eqref{eq.ne4} in $G$ with $0<x_i<\frac{\pi}{2}$ and the denominator of $x_i$ equals $4n$ for $0\le i\le 4$. Let $x_i'$ be the numerator of $x_i$ for each $0\le i\le 4$. In the following, if $v=v(n,a)$, we also denote the relative support of $v$ by $\operatorname{supp}(n,a)$. 
	
\end{emp}

\begin{lemma}\label{le.sfo}
		Let $n\in\mathbb{N}_{> 3}$ be odd, square-free and not a prime. Assume that $(x_0,x_1,x_2,x_3,x_4)$ is a solution to \eqref{eq.ne4} in $G$ with $0<x_k<\frac{\pi}{2}$ and $\operatorname{den}(x_k)=4n$ for each $0\le k\le 4$. Keep the notation in \ref{em.sfproof4n}. Assume that every prime factor of $n$ is greater than $11$ and $\ell$ is odd. Let $1\le i\le 4$. Assume that $x_0$ and $x_i$ satisfy one of the following conditions.
	
	\begin{enumerate}
		\item $\operatorname{len}(4n, x_0')<\operatorname{len}(4n, x_i')=\ell$.
		\item $\operatorname{len}(4n, x_0')=\operatorname{len}(4n, x_i')=\ell$ and  $\operatorname{max}(|\operatorname{pol}(4n,x_0,1)|,\ell-1-|\operatorname{pol}(4n,x_0,1))|<\\
		\operatorname{max}(|\operatorname{pol}(4n,x_i,1)|,\ell-1-|\operatorname{pol}(4n,x_i,1)|)$.
	\end{enumerate}

 	Then there exists $1\le j\le 4$ such that $j\ne i$ and $x_i+x_j=\frac{\pi}{2}$.
\end{lemma}

\begin{proof}

Recall that \[\operatorname{tan}x_k=v(4n,x_k')^2\]
in $\widehat{X^{4n}}$ for each $0\le k\le 4$. See Corollary \ref{cor.tarq}. Let 

$$\lambda_1:=\{0\le k\le 4\mid \operatorname{len}(4n,x_k')<\ell\}.$$
	
	Let $$\lambda_2:=\{0\le k\le 4\mid \operatorname{len}(4n,x_k')=\ell\}.$$

	By the assumption, we know that $i\in\lambda_2$. We assume that $\operatorname{max}(|\operatorname{pol}(4n,x_k,1)|,\ell-1-|\operatorname{pol}(4n,x_k,1)|)\le
	\operatorname{max}(|\operatorname{pol}(4n,x_i,1)|,\ell-1-|\operatorname{pol}(4n,x_i,1)|)$ for each $k\in\lambda_2$. See Remark \ref{rk.assum}.	We further assume that $|\operatorname{pol}(4n,x_i,1)|\ge\ell-1-|\operatorname{pol}(4n,x_i,1)|$. The case $|\operatorname{pol}(4n,x_i,1)|<\ell-1-|\operatorname{pol}(4n,x_i,1)|$ works similarly.
	
	Let

	$$\lambda_3:=\{k\in\lambda_2\mid \operatorname{pol}(4n,x_k',1)=\operatorname{pol}(4n,x_i',1)\}.$$

	By the assumption, we know that $0\notin\lambda_4$.

	\begin{claim}\label{cl.larset}
		$|\lambda_3|\ge 2$ and there exists $j\in\lambda_3$ such that $(\epsilon_{x_j'}x_j')_1\ne (\epsilon_{x_i'}x_i')_1$.
	\end{claim}

	Let $V$ be the set which consists of elements $(a_1,\dots,a_\ell)_{4n}\in \operatorname{supp}(4n,x_i')$ satisfy the following two conditions.	
	
	\begin{enumerate}
		\item $a_1=(0,1).$
		\item For $2\le s\le\ell$, if $s\in\operatorname{pol}(4n,x_i')$, then $\frac{p_s+1}{2}\le a_s\le p_s-2$; if $s\notin\operatorname{pol}(4n,x_i')$, then $a_s=1$. Moreover, if $s=\tau(x_k')$ for some $k\in\lambda_1$, then $a_s\ne (\epsilon_{x_k'} x_k')_s$ and $a_s\ne p_s- (\epsilon_{x_k'} x_k')_s$.
	\end{enumerate}

	Since we assumed that every prime factor of $n$ is greater than $11$, it follows that $V\ne\emptyset$.

	Let $v:=v(4n,c)\in V$.

	\begin{enumerate}[label=(\roman*)]
		
		\item Assume that $k\in\lambda_1$. Let $v(4n,d)\in\operatorname{supp}(4n,x_k')$. Let $s:=\tau(x_k')$. If $s\notin\operatorname{pol}(4n,x_i',1)$, then $c_s=1\ne  d_s$. If $s\in\operatorname{pol}(4n,x_i',1)$, then $c_s\ne  d_s$ by the choice of $V$.

		\item Assume that $k\in\lambda_2\backslash\lambda_3$. Let $v(4n,d)\in\operatorname{supp}(4n,x_k')$. If $v(4n,d)\in \Gamma_1(4n,x_k')\cup \Gamma_2(4n,x_k')$, then there exists $2\le s\le\ell$ such that $2\le  d_s\le\frac{p_s-1}{2}$. Since $ \frac{p_t+1}{2}\le c_t\le p_t-2$ or $c_t=1$ for each $t$, $v(4n,c)\ne v(4n,d)$ in $\widehat{X^{4n}}$. 
		
		If $v(4n,d)\in \Gamma_3(4n,x_k')\cup \Gamma_4(4n,x_k')$, then exists $s\in\operatorname{pol}(4n,x_i',1)$ such that $\frac{p_s+1}{2}\le c_s\le p_s-2$ and $ d_s=1$ by the choice of $x_i$. So $v(4n,c)\ne v(4n,d)$ in $\widehat{X^{4n}}$.

	\end{enumerate}

	Hence, if $k\notin\lambda_3$, then  $\operatorname{multi}_v(\operatorname{tan}x_k)=0$. If $|\lambda_3|=1$, or $|\lambda_3|>1$ and for each $k\in\lambda_3$, we have that  $(\epsilon_{x_k'}x_k')_1= (\epsilon_{x_i'}x_i')_1$, then $\operatorname{multi}_v(\prod\limits_{m=1}^4\operatorname{tan}x_m)\ne 0=\operatorname{multi}_v(\operatorname{tan}x_0)$. This is a contradiction. So the claim is true. By Lemma \ref{le.congsf4dn}, $x_i+x_j=\frac{\pi}{2}$.
\end{proof}

\begin{remark}\label{rk.assum}
	In the above proof, we assumed that $\operatorname{max}(|\operatorname{pol}(4n,x_k,1)|,\ell-1-|\operatorname{pol}(4n,x_k,1)|)\le
	\operatorname{max}(|\operatorname{pol}(4n,x_i,1)|,\ell-1-|\operatorname{pol}(4n,x_i,1)|)$ for each $k\in\lambda_2$. If the given $i$ fails to satisfy the condition, we continue the argument using an index $f\in\lambda_2$ such that $\operatorname{max}(|\operatorname{pol}(4n,x_k,1)|,\ell-1-|\operatorname{pol}(4n,x_k,1)|)\le
	\operatorname{max}(|\operatorname{pol}(4n,x_f,1)|,\ell-1-|\operatorname{pol}(4n,x_f,1)|)$ for each $k\in\lambda_2$. At the final step, we can conclude that there exists $g\in\lambda_3$ such that $x_f+x_g=\frac{\pi}{2}$. After cancellation, we are left with an equation of the form $	(\operatorname{tan}x_0)^2=(\operatorname{tan}x_1)(\operatorname{tan}x_2)$. Iterate the argument with this new equation. We eventually obtain $(\operatorname{tan}x_0)^2=1$ and hence $x_0=\frac{\pi}{4}$. This contradicts the given condition $\operatorname{den}(x_0)=4n$ with $n\ge 3$.  Similar assumptions will be made in the proof of Lemma \ref{le.sfe} and Lemma \ref{le.sfin}. They do not impose additional restrictions for the same reason.
\end{remark}

\begin{lemma}\label{le.sfe}

		Let $n\in\mathbb{N}_{> 3}$ be odd, square-free and not a prime. Assume that $(x_0,x_1,x_2,x_3,x_4)$ is a solution to \eqref{eq.ne4} in $G$ with $0<x_k<\frac{\pi}{2}$ and $\operatorname{den}(x_k)=4n$ for each $0\le k\le 4$. Keep the notation in \ref{em.sfproof4n}. Assume that every prime factor of $n$ is greater than $11$ and $\ell$ is even. Let $1\le i\le 4$. Assume that $x_0$ and $x_i$ satisfy one of the following conditions.
	
	\begin{enumerate}
		\item $\operatorname{len}(4n, x_0')<\operatorname{len}(4n, x_i')=\ell$.
		\item $\operatorname{len}(4n, x_0')=\operatorname{len}(4n, x_i')=\ell$ and $\operatorname{pmin}(4n,x_i')<\operatorname{pmin}(4n,x_0')$.
		\item $\operatorname{len}(4n, x_0')=\operatorname{len}(4n, x_i')=\ell$, $\operatorname{pmin}(4n,x_i')=\operatorname{pmin}(4n,x_0')$ and $|\operatorname{pol}(4n,x_0')|<|\operatorname{pol}(4n,x_i')|$.
	\end{enumerate}

	Then there exists $1\le j\le 4$ such that $j\ne i$ and $x_i+x_j=\frac{\pi}{2}$.
\end{lemma}

\begin{proof}

	Let $$\lambda_1:=\{0\le k\le 4\mid \operatorname{len}(4n,x_k')<\ell\}.$$

	Let $$\lambda_2:=\{0\le k\le 4\mid \operatorname{len}(4n,x_k')=\ell\}.$$

	By the assumption, we have that $i\in\lambda_2$. We assume that $\operatorname{pmin}(4n,x_i')\le \operatorname{pmin}(4n,x_k')$ for each $k\in\lambda_2$. See remark \ref{rk.assum}. Let

	$$\lambda_3:=\{k\in\lambda_2\mid \operatorname{pmin}(4n,x_k')=\operatorname{pmin}(4n,x_i')\}.$$ 
	
	Assume that $\operatorname{pmin}(4n,x_i')<\ell+1$. We assume that $|\operatorname{pol}(4n,x_k',1)|\le|\operatorname{pol}(4n,x_i',1)|$ for each $k\in\lambda_3$. Let

	$$\lambda_4:=\{k\in\lambda_3\mid\operatorname{pol}(4n,x_k',1)= \operatorname{pol}(4n,x_i',1)\}.$$

	Observe that $0\notin\lambda_4$ by the assumption.

\begin{claim}
	$|\lambda_4|\ge 2$ and there exists $j\in\lambda_4$ such that $(\epsilon_{x_j'}x_j')_1\ne (\epsilon_{x_i'}x_i')_1$.
\end{claim}

Let $V$ be the set which consists of elements $(a_1,\dots,a_\ell)_{4n}\in\operatorname{supp}(4n,x_i')$ satisfy the following four conditions.

	\begin{enumerate}
	\item $a_1=(0,1)$.
	\item If $2\le s<\operatorname{pmin}(4n,x_i')$, then $\frac{p_s+1}{2}\le a_s\le p_s-2$. In addition, if $s=\tau(x_k')\text{ for some } k\in\lambda_1$, then $a_s\ne (\epsilon_{x_k'}x_k')_s$ and $a_s\ne p_s-(\epsilon_{x_k'}x_k')_s$.
	\item If $\operatorname{pmin}(4n,x_i')\le s\le \ell$ and $s\in\operatorname{pol}(4n,x_i',1)$, then $2\le a_s\le \frac{p_s-1}{2}$. In addition, if $s=\tau(x_k')\text{ for some } k\in\lambda_1$, then $a_s\ne (\epsilon_{x_k'}x_k')_s$ and $a_s\ne p_s-(\epsilon_{x_k'}x_k')_s$.
	\item If $\operatorname{pmin}(4n,x_i')\le s\le\ell$ and $s\notin\operatorname{pol}(4n,x_i',1)$, then $a_s=1$.
\end{enumerate}

	Since we assumed that every prime factor of $n$ is greater than $11$, it follows that $V\ne\emptyset$. Note that $V\subset \Gamma_1(4n,x_i')$.

Let $v:=v(4n,c)\in V$. Let $r=\operatorname{pmin}(4n,x_i')$.

\begin{enumerate}[label=(\roman*)]
	\item Assume that $k\in\lambda_1$. Let $s:=\tau(x_i')$. Let $v(4n,d)\in\operatorname{supp}(4n,x_k')$. If $s\notin\operatorname{pol}(4n,x_i',1)$, then $c_s=1\ne  d_s$. If $s\in\operatorname{pol}(4n,x_i',1)$, then $c_s\ne  d_s$ by the choice of $V$.

	\item Assume that $k\in\lambda_2\backslash\lambda_3$. Let $v(4n,d)\in\operatorname{supp}(4n,x_k')$. Then $2\le c_r\le \frac{p_r-1}{2}$ and $\frac{p_r+1}{2}\le  d_r\le p_r-2$ or $ d_r=1$. So $c_r\ne  d_r$.
	
	\item Assume that $k\in\lambda_3\backslash\lambda_4$. Let $v(4n,d)\in\operatorname{supp}(4n,x_k')$. If $v(4n,d)\in \Gamma_{1,s}(4n,x_k')$ with $s>r$ or $v(4n,d)\in \Gamma_2(4n,x_k')$, then $2\le c_r\le \frac{p_r-1}{2}$ and $\frac{p_r+1}{2}\le  d_r\le p_r-2$ or $ d_r=1$. So $c_r\ne  d_r$. 
	
	Assume that $v(4n,d)\in \Gamma_{1,r}(4n,x_k)$. Then there exists $r< s\le\ell$ such that $2\le c_s\le \frac{p_s-1}{2}$ and $ d_s=1$ by the choice of $x_i$. So $c_s\ne  d_s$.

\end{enumerate}

	So in each of the above case, $v(4n,c)\ne v(4n,d)$ in $\widehat{X^{4n}}$. Hence $\operatorname{multi}_v(\operatorname{tan}x_i)=0$ if $0\le i\le 4$ and $i\notin\lambda_5$.

If $|\lambda_4|=1$, or $|\lambda_4|>1$ and for each $k\in\lambda_4$, we have that  $(\epsilon_{x_k'}x_k')_1= (\epsilon_{x_i'}x_i')_1$, then $\operatorname{multi}_v(\prod\limits_{m=1}^4\operatorname{tan}x_m)\ne 0=\operatorname{multi}_v(\operatorname{tan}x_0)$. This is a contradiction. So the claim is true. By Lemma \ref{le.congsf4dn}, we have that  $x_i+x_j=\frac{\pi}{2}$.

Next, we address the case $\operatorname{pmin}(4n,x_i')=\ell+1$. We claim that $|\lambda_3|\ge 2$ and there exists $j\in\lambda_4$ such that $(\epsilon_{x_j'}x_j')_1\ne (\epsilon_{x_i'}x_i')_1$. Define $V$ as the set consisting of the single element $v(4n,1)$. By applying the same argument as above, the claim follows. Based the claim and the Lemma \ref{le.congsf4dn},  we can deduce that $x_i+x_j=\frac{\pi}{2}$. 
\end{proof}

\begin{lemma}\label{le.sfin}
	
	Let $n\in\mathbb{N}_{> 3}$ be odd, square-free and not a prime. Assume that $(x_0,x_1,x_2,x_3,x_4)$ is a solution to \eqref{eq.ne4} in $G$ satisfying $0<x_k<\frac{\pi}{2}$, $\operatorname{den}(x_k)=4n$ and $\operatorname{len}(4n,x_k')<\ell$ for each $0\le k\le 4$. Keep the notation in \ref{em.sfproof4n}. Assume that every prime factor of $n$ is greater than $11$. Let $1\le i\le 4$. Assume that $x_0$ and $x_i$ satisfy one of the following conditions.

	\begin{enumerate}
		\item $\operatorname{len}(4n, x_0')<\operatorname{len}(4n, x_i')$.
		\item  $1<\operatorname{len}(4n,x_0')=\operatorname{len}(4n,x_i')$ and $\operatorname{pmin}(4n,x_i')<\operatorname{pmin}(4n,x_0')$.
		\item  $1<\operatorname{len}(4n,x_0')=\operatorname{len}(4n,x_i')$, $\operatorname{pmin}(4n,x_0')=\operatorname{pmin}(4n,x_i')$ and $|\overline{\operatorname{pol}}(4n,x_0',1)|<|\overline{\operatorname{pol}}(4n,x_i',1)|$.
		\item $1<\operatorname{len}(4n,x_0')=\operatorname{len}(4n,x_i')$, $\operatorname{pmin}(4n,x_0')=\operatorname{pmin}(4n,x_i')$, $\overline{\operatorname{pol}}(4n,x_0',1)=\overline{\operatorname{pol}}(4n,x_i',1)$ and $|\widehat{\operatorname{pol}}(4n,x_0')|<|\widehat{\operatorname{pol}}(4n,x_i')|$.

	\end{enumerate}

	Then there exists $1\le j\le 4$ such that $j\ne i$ and $x_i+x_j=\frac{\pi}{2}$.
\end{lemma}

\begin{proof}

	Let $$\lambda_1:=\{0\le k\le 4\mid \operatorname{len}(4n,x_k')<\operatorname{len}(4n,x_i')\}.$$
	
	Assume that $\operatorname{len}(4n, x_i')\ge \operatorname{len}(4n, x_k')$ for all $1\le k\le 4$. See remark \ref{rk.assum}. Let
	
	$$\lambda_2:=\{0\le k\le 4\mid \operatorname{len}(4n,x_k')=\operatorname{len}(4n,x_i')\}.$$
	
	Assume that $\operatorname{pmin}(4n, x_i')\le \operatorname{pmin}(4n, x_k')$ for all $k\in\lambda_2$. Let
	
	$$\lambda_3:=\{k\in\lambda_2\mid \operatorname{pmin}(4n,x_k')=\operatorname{pmin}(4n,x_i')\}.$$

	Assume that $\operatorname{pmin}(4n,x_i')<\operatorname{len}(4n,x_i')+1$. Assume that $|\overline{\operatorname{pol}}(4n,x_k',1)|\le|\overline{\operatorname{pol}}(4n,x_i',1)|$ for all $k\in\lambda_3$. Let
	
	$$\lambda_4:=\{k\in\lambda_3\mid \overline{\operatorname{pol}}(4n,x_k',1)=\overline{\operatorname{pol}}(4n,x_i',1)\}.$$
	
	Assume that $ |\widehat{\operatorname{pol}}(4n,x_k')|\le|\widehat{\operatorname{pol}}(4n,x_i')|$ for all $k\in\lambda_4$. Let
	
	$$\lambda_5:=\{k\in\lambda_4\mid \widehat{\operatorname{pol}}(4n,x_k')= \widehat{\operatorname{pol}}(4n,x_i'),\text{ and } (\epsilon_{x_k'} x_k')_s=(\epsilon_{x_i'} x_i')_s\text{ for all } s\notin\operatorname{pol}(4n,x_i',1)\}.$$

	Notice that $0\notin\lambda_4$ because $x_0$ and $x_i$ satisfy one of the four given conditions.

	\begin{claim}
	$|\lambda_5|\ge 2$ and there exists $j\in\lambda_5$ such that $(\epsilon_{x_j'}x_j')_1\ne (\epsilon_{x_i'}x_i')_1$.
	\end{claim}

	For each $k\in\lambda_4$, if $\widehat{\operatorname{pol}}(4n,x_i')\ne  \widehat{\operatorname{pol}}(4n,x_k')$, choose $\operatorname{len}(4n,x_i')\le t_k\le \ell$ such that $t_k\notin\widehat{\operatorname{pol}}(4n,x_k')$ and $t_k\in\widehat{\operatorname{pol}}(4n,x_i')$.

	Let $V$ be the set which consists of elements $(a_1,\dots,a_\ell)_{4n}\in \operatorname{supp}(4n,x_i')$ satisfy the following six conditions.

	\begin{enumerate}
		\item $a_1=(0,1)$.
		\item If $2\le s<\operatorname{pmin}(4n, x_i')$, then $\frac{p_s+1}{2}\le a_s\le p_s-2$. In addition, if $s=\tau(x_k')\text{ for some } k\in\lambda_1$, then $a_s\ne (\epsilon_{x_k'}x_k')_s$ and $a_s\ne p_s-(\epsilon_{x_k'}x_k')_s$.
		\item If $\operatorname{pmin}(4n, x_i')\le s\le \operatorname{len}(4n,x_i')$ and $s\in\overline{\operatorname{pol}}(4n,x_i',1)$ then $2\le a_s\le \frac{p_s-1}{2}$. In addition, if $s=\tau(x_k')\text{ for some } k\in\lambda_1$, then $a_s\ne (\epsilon_{x_k'}x_k')_s$ and $a_s\ne p_s-(\epsilon_{x_k'}x_k')_s$.
		\item If $\operatorname{pmin}(4n, x_i')\le s\le\operatorname{len}(4n,x_i')$ and $s\notin\overline{\operatorname{pol}}(4n,x_i',1)$, then $a_s=1$.
		\item If $\operatorname{len}(4n,x_i')< s\le\ell$ and $s\in\widehat{\operatorname{pol}}(4n,x_i')$, then $1\le a_s\le p_s-2$. In addition, if $s=t_k\text{ for some } k\in\lambda_4$, then $a_s\ne (\epsilon_{x_k'} x_k')_s$ and $a_s\ne p_s-(\epsilon_{x_k'} x_k')_s$.
		\item If $\operatorname{len}(4n,x_i')< s\le\ell$ and $s\notin\widehat{\operatorname{pol}}(4n,x_i')$, then $a_s=(\epsilon_{x_i'} x_i')_s$.
	\end{enumerate}

Since we assumed that every prime factor of $n$ is greater than $11$ and $|\lambda_1|, |\lambda_3|$ are less than $4$, it follows that $ V\ne\emptyset$.

	Let $v:=v(4n,c)\in V$.

	\begin{enumerate}[label=(\roman*)]
		\item Assume that $k\in\lambda_1$. Let $s:=\tau(x_k')$. Let $v(4n,d)\in\operatorname{supp}(4n,x_k')$. If $s\notin\overline{\operatorname{pol}}(4n,x_i',1)$, then $c_s=1\ne  d_s$. If $s\in\overline{\operatorname{pol}}(4n,x_i',1)$, then $c_s\ne  d_s$ by the choice of $V$.

		\item Assume that $k\in\lambda_2\backslash\lambda_3$. Let $v(4n,d)\in\operatorname{supp}(4n,x_k')$. Let $s=\operatorname{pmin}(4n,x_i')$. Then $2\le c_s\le \frac{p_s-1}{2}$ and, $\frac{p_s+1}{2}\le  d_s\le p_s-2$ or $ d_s=1$ by the choice of $x_i$. 
		
		\item Assume that $k\in\lambda_3\backslash\lambda_4$. Let $v(4n,d)\in\operatorname{supp}(4n,x_k')$. Then there exists $\operatorname{pmin}(4n,x_i')< s\le\operatorname{len}(4n,x_i')$ and $s\in\overline{\operatorname{pol}}(4n,x_i',1)$, such that $2\le c_s\le \frac{p_s-1}{2}$ and $ d_s=1$. 
		
		\item Assume that $x\in\lambda_4\backslash\lambda_5$. Let $v(4n,d)\in\operatorname{supp}(4n,x_k')$. Then there exists $\operatorname{len}(4n,x_i')< s\le\ell$ such that $c_s\ne  d_s$ by the choice of $V$.

	\end{enumerate}
	
	So in each of the above case, $v(4n,c)\ne v(4n,d)$ in $\widehat{X^{4n}}$. Hence $\operatorname{multi}_v(\operatorname{tan}x_k)=0$ if $0\le k\le 4$ and $k\notin\lambda_5$.
	
	If $|\lambda_5|=1$, or $|\lambda_5|>1$ and for each $k\in\lambda_5$, we have that  $(\epsilon_{x_k'}x_k')_1= (\epsilon_{x_i'}x_i')_1$, then $\operatorname{multi}_v(\prod\limits_{m=1}^4\operatorname{tan}x_m)\ne 0=\operatorname{multi}_v(\operatorname{tan}x_0)$. This is a contradiction. So the claim is true. By Lemma \ref{le.congsf4dn}, $x_i+x_j=\frac{\pi}{2}$.

To complete the proof, we need to address the case $\operatorname{pmin}(4n,x_i')=\operatorname{len}(4n,x_i')+1$. Let

	$$\lambda_4':=\{k\in\lambda_3\mid \widehat{\operatorname{pol}}(4n,x_k')= \widehat{\operatorname{pol}}(4n,x_i'),\text{ and } (\epsilon_{x_k'} x_k')_s=(\epsilon_{x_i'} x_i')_s\text{ for all } s\notin\operatorname{pol}(4n,x_i',1)\}.$$

\begin{claim}
	$|\lambda_4'|\ge 2$ and there exists $j\in\lambda_5$ such that $(\epsilon_{x_j'}x_j')_1\ne (\epsilon_{x_i'}x_i')_1$.
\end{claim}

	For each $k\in\lambda_3$, if $\widehat{\operatorname{pol}}(4n,x_i')\ne  \widehat{\operatorname{pol}}(4n,x_k')$, choose $\operatorname{len}(4n,x_i')\le t_k\le \ell$ such that $t_k\notin\widehat{\operatorname{pol}}(4n,x_k')$ and $t_k\in\widehat{\operatorname{pol}}(4n,x_i')$.

First we assume that for each $k\in\lambda_3$, we have that $\frac{p_\tau+1}{2}\le (\epsilon_{x_k'} x_k')_\tau \le p_\tau-2$ where $\tau=\tau(x_k')$. Let $V$ be the set which consists of elements $(a_1,\dots,a_\ell)_{4n}\in \operatorname{supp}(4n,x_i')$ satisfy the following conditions.

\begin{enumerate}
	\item $a_1=(0,1)$.
	\item If $2\le s\le\ell$ and $s\in\operatorname{pol}(4n,x_i')$, then $1\le a_s\le p_s-2$. In addition, if $s=t_k\text{ for some } k\in\lambda_4$, then $a_s\ne (\epsilon_{x_k'} x_k')_s$ and $a_s\ne p_s-(\epsilon_{x_k'} x_k')_s$.
	\item If $2\le s\le\ell$ and $s\notin\operatorname{pol}(4n,x_i')$, then $a_s=(\epsilon_{x_i'} x_i')_s$.

\end{enumerate}

Next we assume that $i$ satisfies the condition that $2\le  (\epsilon_{x_i'} x_i')_\tau\le\frac{p_\tau-1}{2}$ where $\tau=\tau(x_i')$. Let $V$ be the set which consists of elements $(a_1,\dots,a_\ell)_{4n}\in \operatorname{supp}(4n,x_i')$ satisfy the following conditions.

\begin{enumerate}
	\item $a_1=(0,1)$.
	\item If $2\le s\le\operatorname{len}(4n,x_i')$, then $\frac{p_s+1}{2}\le a_s\le p_s-2$. In addition, if $s=\tau(x_k')\text{ for some } k\in\lambda_1$, then $a_s\ne (\epsilon_{x_k'}x_k')_s$ and $a_s\ne p_s-(\epsilon_{x_k'}x_k')_s$. 
	\item If $\operatorname{len}(4n,x_i')< s\le\ell$ and $s\in\operatorname{pol}(4n,x_i')$, then $1\le a_s\le p_s-2$. In addition, if $s=t_k\text{ for some } k\in\lambda_3$, then $a_s\ne (\epsilon_{x_k'} x_k')_s$ and $a_s\ne p_s-(\epsilon_{x_k'} x_k')_s$.
	\item If $\operatorname{len}(4n,x_i')< s\le\ell$ and $s\notin\operatorname{pol}(4n,x_i')$, then $a_s=(\epsilon_{x_i'} x_i')_s$.

\end{enumerate}

By the same argument as above, the claim follows. Applying the claim and the Lemma \ref{le.congsf4dn},  we obtain that $x_i+x_j=\frac{\pi}{2}$. 
\end{proof}

\begin{proof} [Proof of Theorem \ref{th.4msf}]

	Assume that $\operatorname{len}(4n,x_0')=\ell$ and $\ell$ is odd. By Lemma \ref{le.sfo}, we can assume that for each $1\le k\le 4$ such that $\operatorname{len}(4n,x_k')=\ell$, we have that  $\operatorname{max}(|\operatorname{pol}(4n,x_0,1)|,\ell-1-|\operatorname{pol}(4n,x_0,1))|\ge
	\operatorname{max}(|\operatorname{pol}(4n,x_k,1)|,\ell-1-|\operatorname{pol}(4n,x_k,1)|)$.
	
	By the same argument as in the proof of Lemma \ref{le.sfo}, we can conclude that there exist $1\le i< j\le 4$ such that $x_i=x_j=x_0$. By relabeling, assume that $i=1$ and $j=2$. After cancellation, we get that $\operatorname{tan}x_3\operatorname{tan}x_4=1$. So $x_3+x_4=\frac{\pi}{2}$. Hence $(x_0,x_1,x_2,x_3,x_4)\in\Phi_{1,1}$.
	
	The case $\operatorname{len}(4n,x_0')=\ell$ with $\ell$ even and the case $\operatorname{len}(4n,x_0')<\ell$ follows similarly from Lemma \ref{le.sfo}, Lemma \ref{le.sfe} and Lemma \ref{le.sfin}.
\end{proof}

\section{The square-free case}\label{sec.sf2}

	In this section, we analyze the solution of Equation \eqref{eq.ne4} in $G$ under the condition that the denominators of the solution are either $n$ or $2n$, where $n$ is an odd, square-free number that is not a prime. We prove the following theorem.

\begin{theorem}\label{th.sfn2n}
	Let $n\in\mathbb{N}_{> 3}$ be odd, square-free and not a prime. Assume that each prime factor of $n$ is  greater than $7$. Assume that $(x_0,x_1,x_2,x_3,x_4)$ is a solution to \eqref{eq.ne4} in $G$ with $0<x_i<\frac{\pi}{2}$ and $\operatorname{den}(x_i)\in\{n,2n\}$ for each $0\le i\le 4$. Then $(x_0,x_1,x_2,x_3,x_4)\in\Phi_{1,1}$.
\end{theorem}

We postpone the proof of Theorem \ref{th.sfn2n} to the end of this section.

\begin{lemma}\label{le.congsfn}
	Let $n\in\mathbb{N}_{\ge 3}$ be odd. Let $a,b\in\mathbb{Z}$ be such that $0< a, b< n$ and $\operatorname{gcd}(n,ab)=1$. Let $(\epsilon_a a)_s$ and $(\epsilon_a a)_s$ be the elements associated to $\epsilon_a a$ and $\epsilon_b b$ with respect to $n$ as defined in \ref{ss.rf}. Assume that $(\epsilon_a a)_s=(\epsilon_b b)_s$ for each $1\le s\le\ell$. Then the folowings are true.

	\begin{enumerate}
		\item $\epsilon_a=\epsilon_{b}$ if and only if $a=b$.
		\item $\epsilon_a\ne\epsilon_{b}$ if and only if $a+b=n$.
	\end{enumerate}
\end{lemma}

\begin{proof}
	
	Note that $(\epsilon_a a)_s=(\epsilon_b b)_s$ for each $1\le s\le\ell$ if and only if $\epsilon_a a\equiv \epsilon_b b\operatorname{mod} n$ by the Chinese Remainder Theorem. Since $0< a, b< n$, the claims follow.
\end{proof}

\begin{emp}\label{em.sfproof}
	Let $n\in\mathbb{N}_{\ge 3}$ be odd, square-free and not a prime. Assume that every prime factor of $n$ is greater than $7$. Let $n=p_1\cdots p_\ell$ be the prime factorization of $n$ with $p_1<\cdots<p_\ell$. Let $(x_0,x_1,x_2,x_3,x_4)$ be a solution to \eqref{eq.ne4} in $G$ with $0<x_i<\frac{\pi}{2}$ and the denominator of $x_i$ equals $n$ or $2n$ for each $0\le i\le 4$. Write $x_i=\frac{x_i'}{2n}$ for each $0\le i\le 4$. 
	
\end{emp}

\begin{lemma}\label{le.sfred}
	Let $n\in\mathbb{N}_{> 3}$ be odd, square-free and not a prime. Assume that each prime factor of $n$ is  greater than $7$. Assume that $(x_0,x_1,x_2,x_3,x_4)$ is a solution in $G$ to Equation \eqref{eq.ne4} with $0<x_k<\frac{\pi}{2}$ and $\operatorname{den}(x_k)\in\{n,2n\}$ for $1\le k\le 4$. Then, up to reordering, the following is true:
	\begin{enumerate}
		\item $x_1=x_2$ or $x_1+x_2=\frac{\pi}{2}$.
		\item $x_3=x_4$ or $x_3+x_4=\frac{\pi}{2}$.
	\end{enumerate}

\end{lemma}

\begin{proof}
	Keep the notations in \ref{em.sfproof}. We need to deal with the following three cases separately.
	
	\begin{enumerate}[label=(\alph*)]
		\item $\ell$ is odd and there exists $1\le i\le 4$ such that $\operatorname{len}(n,x_i')=\ell$.
		\item $\ell$ is even and there exists $1\le i\le 4$ such that $\operatorname{len}(n,x_i')=\ell$.
		\item For each $1\le k\le 4$, we have that  $\operatorname{len}(n,x_k')<\ell$.
	\end{enumerate}

	The proof for the three cases can be directly adapted from the proofs of Lemma \ref{le.sfo}, Lemma \ref{le.sfe}, and Lemma \ref{le.sfin}, respectively. In the following, we point out the required modifications to the proof of Lemma \ref{le.sfo} in order to establish Lemma \ref{le.sfred} in case (a). The other two cases can be approached similarly.

	By Corollary \ref{cor.tarq}, if $x_i'$ is even, then $$\operatorname{tan}x_i=v(n,x_i')^{-1}v(n,2^{-1}x_i')^2$$ 

in $\widehat{X^n}$.
If $x_i'$ is odd, then $$\operatorname{tan}x_i=v(n,x_i')v(n,2^{-1}x_i')^{-2}$$

in $\widehat{X^n}$. 

We define the sets $\lambda_1$ and $\lambda_2$ as in the proof of Lemma \ref{le.sfo}. Let $$\lambda_1:=\{1\le k\le 4\mid \operatorname{len}(n,x_k')<\ell\},$$

and $$\lambda_2:=\{1\le k\le 4\mid \operatorname{len}(n,x_k')=\ell\}.$$

 Let $i\in\lambda_2$ be such that

	$$\operatorname{max}(|\operatorname{pol}(n,x_k')|,\ell-|\operatorname{pol}(n,x_k')|)\le
	\operatorname{max}(|\operatorname{pol}(n,x_i')|,\ell-|\operatorname{pol}(n,x_i')|)$$ 
	
	for each $k\in\lambda_2$. We further assume that $|\operatorname{pol}(n,x_i')|\ge\ell-|\operatorname{pol}(n,x_i')|$ as in the proof of Lemma \ref{le.sfo}. Define  
	
	$$\lambda_3:=\{k\in\lambda_2\mid \operatorname{pol}(n,x_k')=\operatorname{pol}(n,x_i')\}.$$

	We claim that $|\lambda_3|\ge 2$. This claim is proved similarly as the claim \ref{cl.larset}. Note that since $0\notin\lambda_1$, the set $V$ is nonempty under the condition that every prime factor of $n$ is greater than $7$. The concluding paragraph of the proof of the claim \ref{cl.larset} should be adjusted to the following.
	
	If $k\notin\lambda_3$ and $k\ne 0$, then  $\operatorname{multi}_v(\operatorname{tan}x_k)$ is even. If $|\lambda_3|=1$, then $\operatorname{multi}_v(\prod\limits_{i=1}^4\operatorname{tan}x_i)$ is odd. Note that $\operatorname{multi}_v((\operatorname{tan}x_0)^2)$ is even. So  $\operatorname{multi}_v(\prod\limits_{i=1}^4\operatorname{tan}x_i)\ne\operatorname{multi}_v((\operatorname{tan}x_0)^2)$. This is a contradiction. Hence the claim holds. Then by Lemma \ref{le.congsfn}, we obtain that there exists $j\in\lambda_3$ such that $x_i=x_j$ or $x_i+x_j=\frac{\pi}{2}$. 
	
	Repeat the argument, we arrive at the following: up to reordering, we have that (i) $x_1=x_2$ or $x_1+x_2=\frac{\pi}{2}$ and (ii) $x_3=x_4$ or $x_3+x_4=\frac{\pi}{2}$. This finishes the proof of the case $\ell$ is odd and there exists $1\le i\le 4$ such that $\operatorname{len}(n,x_i')=\ell$.
\end{proof}

Consider the following four equations.

\begin{equation}\label{eq.red20}
	(\operatorname{tan}x_0)^2=1.
\end{equation}

\begin{equation}\label{eq.red22}
	(\operatorname{tan}x_0)^2=(\operatorname{tan}x_1)^2.
\end{equation}

\begin{equation}\label{eq.red3}
	(\operatorname{tan}x_0)^2=(\operatorname{tan}x_1)^4.
\end{equation}

\begin{equation}\label{eq.red2}
	(\operatorname{tan}x_0)^2=(\operatorname{tan}x_1)^2(\operatorname{tan}x_2)^2.
\end{equation}

\begin{lemma}\label{le.sfred220}
	Let $n\in\mathbb{N}_{> 3}$ be odd, square-free and not a prime. Assume that each prime factor of $n$ is  greater than $7$. Then the following is true.
	
	\begin{enumerate}
		\item There is no solution in $G$ to Equation \eqref{eq.red20} of the form $(x_0,x_1)$ with $0<x_i<\frac{\pi}{2}$ and $\operatorname{den}(x_i)\in\{n,2n\}$ for $0\le i\le 1$.
		\item Assume that $(x_0,x_1)$ is a solution in $G$ to Equation \eqref{eq.red22} with $0<x_i<\frac{\pi}{2}$ and $\operatorname{den}(x_i)\in\{n,2n\}$ for $1\le i\le 4$. Then $x_0=x_1$.
		\item There is no solution in $G$ to Equation \eqref{eq.red3} of the form $(x_0,x_1)$ with $0<x_i<\frac{\pi}{2}$ and $\operatorname{den}(x_i)\in\{n,2n\}$ for $0\le i\le 1$.
		\item There is no solution in $G$ to Equation \eqref{eq.red2} of the form $(x_0,x_1,x_2)$ with $0<x_i<\frac{\pi}{2}$ and $\operatorname{den}(x_i)\in\{n,2n\}$ for $0\le i\le 1$.
	\end{enumerate}

\end{lemma}

\begin{proof}
	We only discuss the proof of last claim. The others work similarly. Assume that $(x_0,x_1,x_2)$ is a solution to Equation \eqref{eq.red2} with $0<x_i<\frac{\pi}{2}$ and $\operatorname{den}(x_i)\in\{n,2n\}$ for $0\le i\le 1$. Then 
	$\operatorname{tan}x_0=\operatorname{tan}x_1\operatorname{tan}x_2.$ Similar to the argument in the proof of Lemma \ref{le.sfred}, one can assume that $x_1=x_2$ or $x_1+x_2=\frac{\pi}{2}$. In both cases, apply the argument of the proof of Lemma \ref{le.sfred} one more time. We get a contradiction. So the claim follows.
\end{proof}

\begin{proof} [Proof of Theorem \ref{th.sfn2n}]
	The conclusion follows from Lemma \ref{le.sfred} and Lemma \ref{le.sfred220}.
\end{proof}

\section{The non square-free case}\label{sec.nonsf}

In this section, we investigate the solution to Equation \eqref{eq.ne4} under the condition that the denominator of the solution is a non-square-free integer. Our main results are Theorem \ref{th.4mnsf}, Theorem \ref{th.mnsf} and Theorem \ref{th.8nsf}.

\begin{theorem}\label{th.4mnsf}
	Let $n\in\mathbb{N}_{\ge 9 }$ be odd and non square-free.  Assume that every prime factor of $n$ is greater than $5$. Assume that $(x_0,x_1,x_2,x_3,x_4)$ is a solution to \eqref{eq.ne4} in $G$ with $0<x_i<\frac{\pi}{2}$ and $\operatorname{den}(x_i)=4n$ for each $0\le i\le 4$. Then $(x_0,x_1,x_2,x_3,x_4)\in\Phi_{1,1}$.
\end{theorem}

Before proving Theorem \ref{th.4mnsf}, we need some preparations.

\begin{emp}\label{ss.brnf}

	Let $n\in\mathbb{N}_{\ge 9 }$ be odd and non square-free. Let $4n=p_1^{e_1}\cdots p_\ell^{e_\ell}$ be the prime factorization of $4n$ with $p_1<\dots<p_\ell$, and $e_i>0$ for $1\le i\le \ell$. Notice that $p_1=2$ and $e_1=2$. Recall that $\mu:=\operatorname{min}\{2\le s\le\ell\mid e_s\ge 2\}$.

	Let $a\in\mathbb{Z}$ be such that $\operatorname{gcd}(4n,a)=1$.  Let $a_s, \overline{a}_s$, $\widehat{a}_s$ and $\widetilde{a}_s$ be the elements associated to $a$ with respect to $4n$ defined in \ref{ss.rf}. In particular, the residue form of $v(4n,a)$ is $(a_1,\dots,a_\ell)_{4n}$. Similarly, let $(-a)_s, \overline{(-a)}_s$, $\widehat{(-a)}_s$ and $\widetilde{(-a)}_s$ be the elements associated to $-a$ with respect to $4n$. In particular, the residue form of $v(4n,-a)$ is $((-a)_1,\dots,(-a)_\ell)_{4n}$.
	
\end{emp}

\begin{lemma}\label{le.min}
Keep the assumptions in \ref{ss.brnf}. Then the following is true.	
	\begin{enumerate}
		\item If $e_s=1$ for some $2\le s\le\ell$, then $(-a)_s=p_s-a_s$.
		\item If $e_s\ge 2$ for some $1\le s\le\ell$, then $\overline{(-a)}_s=p_s-\overline{a}_s-1$ and $\widehat{(-a)}_s=p_s^{e_s-1}-\widehat{a}_s$.
	\end{enumerate}
\end{lemma}
\begin{proof}
	\begin{enumerate}
		\item $e_s=1$. Note that $p_s-a_s\equiv -a\operatorname{mod} p_s$ and $1\le p_s-a_s\le p_s-1$. By definition, $(-a)_s=p_s-a_s$.
		\item $e_s\ge 2$. In this case, $\widetilde{(-a)}_s={p_s}^{e_s}-\widetilde{a}_s=(p_s-\overline{a}_s-1)p_s^{e_s-1}+(p_s^{e_s-1}-\widehat{a}_s)$ where $0\le p_s-\overline{a}_s-1\le p_s-1$ and $1	\le (p_s^{e_s-1}-\widehat{a}_s)\le p_s^{e_s-1}-1$. By definition, $\overline{(-a)}_s=p_s-\overline{a}_s-1$ and $\widehat{(-a)}_s=p_s^{e_s-1}-\widehat{a}_s$.
	\end{enumerate}
It is clear that the statement also holds true when $n$ is even.
\end{proof}

\begin{lemma}\label{le.excl}
	Keep the assumptions in \ref{ss.brnf}. If $e_s\ge 2$ for some $1\le s\le\ell$, then exactly one of the two inequalities, $\widehat{a}_s< \frac{p_s^{e_s-1}}{2}$ and $\widehat{(-a)}_s< \frac{p_s^{e_s-1}}{2}$, holds.   
\end{lemma}

\begin{proof}
	We prove the statement without requiring $n$ to be odd. First we show $\widehat{a}_s\ne \frac{p_s^{e_s-1}}{2}$. 
	Recall that $\widetilde{a}_s=\overline{a}_s p_s^{e_s-1}+\widehat{a}_s$ where $\widetilde{a}_s$ is the integer satisfying $\widetilde{a}_s\equiv a \operatorname{mod} p_s^{e_s}$ and $1\le \widetilde{a}_s\le p_s^{e_s}-1$. If $\widehat{a}_s= \frac{p_s^{e_s-1}}{2}$, then $p_s=2$, $s=1$ and $e_s\ge 3$. Therefore $2\mid \widetilde{a}_s$ and so $2\mid a$. Hence $\operatorname{gcd}(4n,a)\ne 1$. A contradiction. So $\widehat{a}_s\ne \frac{p_s^{e_s-1}}{2}$. By Lemma \ref{le.min}, $\widehat{a}_s+\widehat{(-a)}_s=p_s^{e_s-1}$. The conclusion follows.
\end{proof}

\begin{emp}\label{def.nonsfepsilon}
	Keep the assumptions in \ref{ss.brnf}. If $\widehat{a}_\mu< \frac{p_\mu^{e_\mu-1}}{2}$, let $\epsilon_a=1$; if $\widehat{(-a)}_\mu< \frac{p_\mu^{e_\mu-1}}{2}$, let $\epsilon_a=-1$. Then $\widehat{(\epsilon_a a)}_\mu< \frac{p_\mu^{e_\mu-1}}{2}$. By Lemma \ref{le.excl}, the number $\epsilon_a$ is well defined.

		 Let $1\le s\le\ell$. The index $s$ is called a \textit{pole} of $v(4n,a)\in X^{4n}$ if $\overline{(\epsilon_a a)}_s=p_s-1$. Let

	$$\operatorname{pol}(4n,a):=\{1\le s\le \ell\mid \overline{(\epsilon_a a)}_s=p_s-1 \},$$ 
	
	and for $1\le r<\ell$, let
	
	$$\operatorname{pol}(4n,a,r):=\{r< s\le \ell\mid \overline{(\epsilon_a a)}_s=p_s-1 \}.$$ 
	
	Let 
	
	$$e(a):=|\operatorname{pol}(4n,a)|.$$

	Let 
	\begin{equation*}
		\aligned
		\Gamma(4n,a):=&\{(b_1,\dots,b_\ell)_{4n}\in X^n\mid b_1=(0,1). \text{ For }2\le s\le \ell, \\
		&\text{ if } s\in\operatorname{pol}(4n,a,1)\text{ and } e_s=1, \text{ then }  1\le b_s\le p_s-2;\\
		&\text{ if } s\in\operatorname{pol}(4n,a,1)\text{ and } e_s\ge 2,\text{ then }b_s=(\overline{r}_s,\widehat{ r}_s)\in\mathbb{Z}^2 \\
		& \text{ with }  0\le \overline{r}_s\le p_s-2 \text{ and }\widehat{ r}_s=\widehat{(\epsilon_a a)}_s;\text{ if } s\notin\operatorname{pol}(4n,a,1), \\
		&  \text{ then } b_s=(\epsilon_a a)_s\}.
		\endaligned
	\end{equation*}

	Here $(b_1,\dots,b_\ell)_{4n}$ refers to the corresponding element $v(4n,b)$ as discussed in Remark \ref{re.rf}.
	
\end{emp}

\begin{remark}
	If $e(a)=0$, then $\Gamma(4n,a)=\{v(4n,a)\}$. This is the case when $v(4n,a)$ has no poles.
\end{remark}

\begin{lemma}\label{le.ba2}
		Keep the assumptions in \ref{ss.brnf}. Then the following is true.
	\begin{enumerate}
		\item $\Gamma(4n,a)\subseteq B_{4n}$, where $B_{4n}$ the set that induces a basis of $\widehat{X^{4n}}$ under the quotient map given in Theorem \ref{th.ba1}.
		\item Let $\Gamma(4n,a)$ be as in \ref{def.nonsfepsilon}. Then \[v(4n,a)=\prod\limits_{v\in \Gamma(4n,a)}v^{(-1)^{e(a)}},\]
		in $\widehat{X^{4n}}$, where both sides are viewed as elements in $\widehat{X^{4n}}$ under the quotient map.
	\end{enumerate}
\end{lemma}

\begin{proof}
	Let $(b_1,\dots,b_\ell)_{4n}\in \Gamma(4n,a)$. By definition, $1\le \widehat{ b}_k< \frac{p_k^{e_k-1}}{2}.$  Since for all $ 1\le i\le \ell$, $\overline{ b}_i\ne p_i-1$, we have that $(b_1,\dots,b_\ell)_{4n}\in B_{4n}$. The second part follows from Lemma \ref{le.nor1}.
\end{proof}

\begin{remark}

	If the number $4n$ in Lemma \ref{le.ba2} is replaced by a number $m$ satisfying either of the following two conditions, the Lemma still holds after minor adjustments.
	
	\begin{enumerate}
		\item $m \in \mathbb{N}{\ge 9}$, which is odd and non-square-free.
		\item $m \in \mathbb{N}{\ge 8}$ such that $8 \mid m$.
	\end{enumerate}
	
	This is due to the similarity in the basis representation in these cases, as shown in the Theorem \ref{th.ba1}.

\end{remark}

\begin{emp}\label{em.nonsfproof}
	Let $n$ be as in \ref{ss.brnf}. Let $(x_0,x_1,x_2,x_3,x_4)$ be a solution to \eqref{eq.ne4} in $G$ with $0<x_i<\frac{\pi}{2}$ and the denominator of $x_i$ equals $4n$ for each $0\le i\le 4$. Let $x_i'$ be the numerator of $x_i$.
	
\end{emp}

\begin{lemma}\label{le.ns1}
	
	Let $n\in\mathbb{N}_{\ge 3}$ be is odd and not square-free. Assume that every prime factor of $m$ is greater than $5$. Assume that $(x_0,x_1,x_2,x_3,x_4)$ is a solution to \eqref{eq.ne4} in $G$ with $0<x_k<\frac{\pi}{2}$ and $\operatorname{den}(x_k)=4n$ for each $0\le k\le 4$. Keep the notation in \ref{em.nonsfproof}. Assume that $|\operatorname{pol}(4n, x_0',1)|<|\operatorname{pol}(4n, x_i',1)|$ for some $1\le i\le 4$. Then there exists $1\le j\le 4$ such that $j\ne i$ and $x_i+x_j=\frac{\pi}{2}$.

\end{lemma}

\begin{proof}
	 Assume that $|\operatorname{pol}(4n, x_i')|\ge|\operatorname{pol}(4n, x_k')|$ and for each $0\le k\le 4$. See Remark
	
	Let 
	$$\lambda_1:=\{0\le k\le 4\mid \operatorname{pol}(4n, x_k',1)\ne \operatorname{pol}(4n, x_i',1)\}.$$

	Let $$\lambda_2:=\{0\le k\le 4\mid \operatorname{pol}(4n, x_k',1)=\operatorname{pol}(4n, x_i',1)\}$$
	
	and

	$$\lambda_3:=\{k\in\lambda_2\mid (\epsilon_{x_k'}x_k')_s=(\epsilon_{x_i'}x_i')_s\text{ for each } 2\le s\le\ell\}.$$
	
	By the assumption, we get that $0\notin\lambda_3$.

	\begin{claim}
		$|\lambda_3|\ge 2$ and there exists $j\in\lambda_3$ such that $\overline{(\epsilon_{x_j'}x_j')}_1\ne\overline{(\epsilon_{x_i'}x_i')}_1$.
	\end{claim}
	
	For each $k\in\lambda_1$, choose $2\le t_k\le\ell$ such that $t_k\in\operatorname{pol}(4n,x_i',1)$ and $t_k\notin\operatorname{pol}(4n,x_k',1)$.

	Let $V$ be the set which consists of elements $(a_1,\dots,a_\ell)_{4n}\in X^{4n}$ satisfy the following four conditions.

	\begin{enumerate}

		\item $a_1=(0,1)$.

		\item If $2\le s\le\ell$, $s\in\operatorname{pol}(4n,x_i')$ and $e_s=1$, then $1\le a_s\le p_s-2$. In addition, if $s=t_k$ for some $k\in\lambda_1,$ then $a_s\ne (\epsilon_{x_k'}x_k')_s$.
		
		\item If $2\le s\le\ell$, $s\in\operatorname{pol}(4n,x_i')$ and $e_s\ge 2$, then $a_s=(\overline{a}_s,\widehat{a}_s)$ where $0\le \overline{a}_s\le p_s-2,$ and $ \widehat{a}_s=\widehat{(\epsilon_{x_i'}x_i')}_s$. In addition, if $s=t_k$ for some $k\in\lambda_1,$ then $\overline{a}_s\ne \overline{(\epsilon_{x_k'}x_k')}_s$.

		\item $2\le s\le\ell$ and $s\notin\operatorname{pol}(4n,x_i')$. Then $a_s=(\epsilon_{x_i'}x_i')_s$.

	\end{enumerate}

	Note that $V\subseteq B_{4n}$. Since we assumed that every prime factor of $n$ is greater than $5$ and $|\lambda_1|$ is less than $5$, it follows that $ V\ne\emptyset$.

	Let $v(4n,c)\in V$.

	\begin{enumerate}[label=(\roman*)]
	\item Assume that $k\in\lambda_1$. Let $v(4n,d)\in\operatorname{supp}(4n,x_k')$. Then there exists $s\in\operatorname{pol}(4n,x_k',1)$ such that $ d_s\ne c_s$ by the choice of $x_i$ and $V$.

	\item Assume that $k\in\lambda_2\backslash\lambda_3$.  Let $v(4n,d)\in\operatorname{supp}(4n,x_k')$. Then there exists $s\notin\operatorname{pol}(4n,x_k',1)$ such that $s\ne1$ and $ d_s\ne c_s$.

\end{enumerate}

So in each of the above case, $v(4n,c)\ne v(4n,d)$ in $\widehat{X^{4n}}$. Hence $\operatorname{multi}_v(\operatorname{tan}x_k)=0$ if $0\le k\le 4$ and $k\notin\lambda_3$.

If $|\lambda_3|=1$, or $|\lambda_3|>1$ and for each $k\in\lambda_3$, we have that  $\overline{(\epsilon_{x_k'}x_k')}_1=\overline{(\epsilon_{x_i'}x_i')}_1$, then $\operatorname{multi}_v(\prod\limits_{m=1}^4\operatorname{tan}x_m)\ne 0=\operatorname{multi}_v(\operatorname{tan}x_0)$. This is a contradiction. So the claim is true. By Lemma \ref{le.congsf4dn}, $x_i+x_j=\frac{\pi}{2}$.
\end{proof}

\begin{proof} [Proof of Theorem \ref{th.4mnsf}]

By Lemma \ref{le.ns1}, we can assume that $|\operatorname{pol}(4n, x_0',1)|\ge|\operatorname{pol}(4n, x_k',1)|$ for each $1\le k\le 4$. By the same argument as in the proof of Lemma \ref{le.ns1}, we can conclude that there exist $1\le i< j\le 4$ such that $x_i=x_j=x_0$. By relabeling, assume that $i=1$ and $j=2$. After cancellation, we get that $\operatorname{tan}x_3\operatorname{tan}x_4=1$. So $x_3+x_4=\frac{\pi}{2}$. Hence $(x_0,x_1,x_2,x_3,x_4)\in\Phi_{1,1}$.
\end{proof}

\begin{theorem}\label{th.mnsf}
	Let $n\in\mathbb{N}_{\ge 9 }$ be odd and not square-free. Assume that every prime factor of $n$ is greater than $5$.	Assume that $(x_0,x_1,x_2,x_3,x_4)$ is a solution to \eqref{eq.ne4} in $G$ with $0<x_i<\frac{\pi}{2}$ and $\operatorname{den}(x_i)\in\{n,2n\}$ for each $0\le i\le 4$. Then $(x_0,x_1,x_2,x_3,x_4)\in\Phi_{1,1}$.

\end{theorem}

\begin{proof} [Proof of Theorem \ref{th.mnsf}]

	The conclusion follows from Lemma \ref{le.nsfred} and Lemma \ref{le.nsfred220}.
\end{proof}

\begin{lemma}\label{le.nsfred}
	Let $n\in\mathbb{N}_{\ge 9 }$ be odd and not square-free. Assume that every prime factor of $n$ is greater than $5$. 	 Assume that $(x_0,x_1,x_2,x_3,x_4)$ is a solution in $G$ to Equation \eqref{eq.ne4} with $0<x_i<\frac{\pi}{2}$ and $\operatorname{den}(x_i)\in\{n,2n\}$ for $1\le i\le 4$. Then, up to reordering, the following is true:
	\begin{enumerate}
		\item $x_1=x_2$ or $x_1+x_2=\frac{\pi}{2}$.
		\item $x_3=x_4$ or $x_3+x_4=\frac{\pi}{2}$.
	\end{enumerate}

\end{lemma}

\begin{proof}
	The claim follows by a combination of the proofs of Lemma \ref{le.sfred} and Lemma \ref{le.ns1}.
\end{proof}

\begin{lemma}\label{le.nsfred220}
	Let $n\in\mathbb{N}_{\ge 9 }$ be odd and not square-free. Assume that every prime factor of $n$ is greater than $5$. Then the following is true.
	
	\begin{enumerate}
		\item There is no solution in $G$ to Equation \eqref{eq.red20} of the form $(x_0,x_1)$ with $0<x_i<\frac{\pi}{2}$ and $\operatorname{den}(x_i)\in\{n,2n\}$ for $0\le i\le 1$.
		\item Assume that $(x_0,x_1)$ is a solution in $G$ to Equation \eqref{eq.red22} with $0<x_i<\frac{\pi}{2}$ and $\operatorname{den}(x_i)\in\{n,2n\}$ for $1\le i\le 4$. Then $x_0=x_1$.
		\item There is no solution in $G$ to Equation \eqref{eq.red3} of the form $(x_0,x_1)$ with $0<x_i<\frac{\pi}{2}$ and $\operatorname{den}(x_i)\in\{n,2n\}$ for $0\le i\le 1$.
		\item There is no solution in $G$ to Equation \eqref{eq.red2} of the form $(x_0,x_1,x_2)$ with $0<x_i<\frac{\pi}{2}$ and $\operatorname{den}(x_i)\in\{n,2n\}$ for $0\le i\le 1$.
	\end{enumerate}

\end{lemma}

\begin{proof}
	The claims follows by a combination of the proofs of Lemma \ref{le.sfred220} and Lemma \ref{le.nsfred}.
\end{proof}

\begin{theorem}\label{th.8nsf}
	Let $n\in\mathbb{N}_{\ge 8}$ satisfying $8\mid n$. Assume that each prime factor of $n$ is either $2$ or greater than $5$. Assume that $(x_0,x_1,x_2,x_3,x_4)$ is a solution to \eqref{eq.ne4} in $G$ with $0<x_i<\frac{\pi}{2}$ and $\operatorname{den}(x_i)=n$ for each $0\le i\le 4$. Then $(x_0,x_1,x_2,x_3,x_4)\in\Phi_{1,1}$.
\end{theorem}

\begin{proof}
	The proof is the same as Theorem \ref{th.4mnsf}. 
\end{proof}

\section{A case with small prime divisors}\label{sec.smal}

Up until now, the statements we have proved frequently assumed that the denominators exclude certain small divisors. This section investigates a special case when we allow some small divisors. To be more precise, let $n=2^r 5$ with $r$ a positive integer. In Theorem \ref{th.25}, we provide a description of the solutions to Equation \eqref{eq.ne4} $G$ when the denominators of the solutions are divisors of $n$.

\begin{theorem}\label{th.25}
	Let $n=2^r5$ with $r$ a positive integer. Assume that $(x_0,x_1,x_2,x_3,x_4)$ is a solution to \eqref{eq.ne4} in $G$ with $0<x_i<\frac{\pi}{2}$ for each $0\le i\le 4$. Assume that $\operatorname{max}\{\operatorname{den}(x_i)\mid 0\le i\le2\}=n$ and $\operatorname{den}(x_i)\mid n$ for each $0\le i\le 4$. Then either the tuple $(x_0,x_1,x_2,x_3,x_4)\in\Phi_{1,1}\cup\Phi_{1,2}$ or, up to reordering $x_1,x_2,x_3,x_4$, the tuple $(x_0,x_1,x_2,x_3,x_4)\in \{(\frac{1}{8},\frac{1}{40},\frac{7}{40},\frac{9}{40},\frac{17}{40})\pi,(\frac{3}{8},\frac{23}{40},\frac{31}{40},\frac{33}{40},\frac{39}{40})\pi\}$.
	
\end{theorem}

We need the following Lemmas to prove Theorem \ref{th.25}.

\begin{lemma}\label{le.hom}
	Let $n:=2^r5$. Fix $1\le \widehat{a}< 2^{r-2}$. For each $0\le i\le 4$, let $\overline{a}_i=0$ or $1$; let $a_i:=(\overline{a}_i,\widehat{a})$ and $w_i:=(a_i,i)_n\in X^n$. Assume that $\prod\limits_{i=1}^4w_i=1$ in $\widehat{X}^n$. Then $\overline{a}_1=\cdots=\overline{a}_4$.
\end{lemma}

\begin{proof}
		Prove by contradiction. Assume that for some $1\le j\le 3$, $\overline{a}_j\ne \overline{a}_4$. Let $b:=(0,\widehat{a})$ and $v:=(b,j)_n$. Then $v$ is a basis element and $\operatorname{multi}_v(\prod\limits_{i=0}^4 w_i)=\pm 2$. This is a contradiction. 
\end{proof}

\begin{lemma}\label{le.nosq}
	Assume that $r\ge 4$. Let $n:=2^r5$ and $m:=2^r$. Fix $a_1:=(\overline{a}_1,\widehat{a}_1)\in\mathbb{Z}^2$ where $0\le \overline{a}_1\le 1$ and $1\le \widehat{a}_1< 2^{r-2}$. For each $1\le i\le 4$, let $w_i:=(a_1,i)_n\in X^n$. Then for some basis element $u$ of level $m$, we have that $\operatorname{multi}_{u}(\prod\limits_{i=1}^4 w_i)=\pm 1$.
\end{lemma}

\begin{proof}

	Assume that $v_1=v(m,b)$ and $v_2=v(m,c)$ where $c\equiv 5^{-1}b \operatorname{mod} m$ and $1\le b, c<m$. Note that $\prod\limits_{i=1}^4w_i=v_1v_2^{-1}$ by the Formula \ref{eq.nor}. Let $u_1:=(b_1)_m$ where  $b_1=(0,\widehat{ b}_1)\in\mathbb{Z}^2$
	and $u_2:=(c_1)_m$ where  $c_1=(0,\widehat{c}_1)\in\mathbb{Z}^2$.	Note that $u_1,u_2\in B_n$, $v_1=u_1^{\pm 1}$ and $v_2=u_2^{\pm 1}$.

	\begin{claim}
		$u_1\ne u_2$.
	\end{claim}
	
	If $v(m,5^{-1}b)\ne v(m,b)$ in $X^m$ and $v(m,5^{-1}b)\ne  v(m,b+2^{r-1})$ in $X^m$, then $\widehat{c}_1\ne \widehat{ b}_1$. We are done. So we can assume that $v(m,5^{-1}b)= v(m,b)$ or $v(m,5^{-1}b)=  v(m,b+2^{r-1})$.\\

		(i) The case $v(m,5^{-1}b)=v(m,b)$. 
			It follows that $5^{-1}b\equiv b\operatorname{mod} 2^r$ or $5^{-1}b\equiv -b\operatorname{mod} 2^r$, so $5\equiv 1 \operatorname{mod} 2^r$ or $5\equiv -1 \operatorname{mod} 2^r$. Since we assumed that $r\ge 4$, this is a contradiction.\\
			
		(ii) The case $v(m,5^{-1}b)= v(m,b+2^{r-1})$.
			It follows that, $5^{-1}b\equiv b+2^{r-1} \operatorname{mod} 2^r$ or $5^{-1}b\equiv -b-2^{r-1} \operatorname{mod} 2^r$.

		In the first case, $5^{-1}\equiv 1+b^{-1}2^{r-1} \operatorname{mod} 2^r$. So $(1+b^{-1}2^{r-1})5\equiv 1\operatorname{mod} 2^r.$ So $(5b^{-1}-1)2^{r-1}+2^{r-1}+4\equiv 0 \operatorname{mod} 2^r$. Note that both $5$ and $b^{-1}$ are odd. Thus $2^{r-1}+4\equiv 0 \operatorname{mod} 2^r$. If $r\ge 4$, then $0<2^{r-1}+4< 2^r$. This is a contradiction. 
		
		In the second case, $5^{-1}\equiv -1-b^{-1}2^{r-1} \operatorname{mod} 2^r$. So $(1+b^{-1}2^{r-1})5\equiv -1\operatorname{mod} 2^r.$ So $(5b^{-1}-1)2^{r-1}+2^{r-1}+6\equiv 0\operatorname{mod} 2^r.$ Thus $2^{r-1}+6\equiv 0 \operatorname{mod} 2^r$. 	If $r\ge 4$, then $0<2^{r-1}+6< 2^r$. This is a contradiction.

	From the claim, it follows that $\operatorname{multi}_{u_1}(\prod\limits_{i=1}^4 w_i)=\pm 1$ and $\operatorname{multi}_{u_2}(\prod\limits_{i=1}^4 w_i)=\pm 1$.
\end{proof}

Consider the equation

\begin{equation}\label{eq.eq2ter}
	(\operatorname{tan}x_0)^2=(\operatorname{tan}x_1)(\operatorname{tan}x_2).
\end{equation}

Let 

	\begin{equation*}
	\aligned
	&\Phi_1:=\{(s,s,s)\pi\mid s,t\in\mathbb{Q}\text{ and } 0<s< \frac{1}{2}\},\\
	&\Phi_2:=\{(\frac{1}{4},s,\frac{1}{2}-s)\pi\mid s,t\in\mathbb{Q}\text{ and } \;0<s< \frac{1}{2}\}.\\
	\endaligned
\end{equation*}

\begin{lemma}\label{le.shoeq}
	Let $n=2^r5$ with $r$ a positive integer. Assume that $(x_0,x_1,x_2)$ is a solution to \eqref{eq.eq2ter} in $G$ with $0<x_i<\frac{\pi}{2}$ and $\operatorname{den}(x_i)\mid n$ for each $0\le i\le 4$ Then the tuple $(x_0,x_1,x_2)$ is in the set $\Phi_1\cup\Phi_2.$
\end{lemma}

\begin{proof}
	Let $m:=\operatorname{max}\{\operatorname{den}(x_i)\mid 0\le i\le2\}$. Assume that $8\mid m$. Let 
	$$\lambda_1:=\{0\le i\le2\mid \operatorname{den}(x_i)=m \}.$$
	
	Assume that $(x_0,x_1,x_2)\notin\Phi_2$. We claim that $|\lambda_1|=3$. Assume that $|\lambda_1|\le 2$. If $0\in\lambda_1$, then for some basis element $v$ of level $m$, we have that  $\operatorname{multi}_v((\operatorname{tan}x_0)^2)=\pm 4$ and $|\operatorname{multi}_v((\operatorname{tan}x_1)(\operatorname{tan}x_2))|\le 2$, a contradiction. Assume that $0\notin\lambda_1$. By the same argument as in the proof of Lemma \ref{le.ns1}, we have that $x_1+x_2=\frac{\pi}{2}$. Then $x_0=\frac{\pi}{4}$. Hence $(x_0,x_1,x_2)\notin\Phi_2$, a contradiction. So the claim follows.

	By applying the same proof as for Theorem \ref{th.4mnsf}, we can conclude that $(x_0, x_1, x_2)\in\Phi_1$. If $8\nmid m$, there are only a few cases to check.
\end{proof}

\begin{proof}[Proof of Theorem \ref{th.25}]
	
	Assume that $r\ge 4$. When $r<4$ we can examine case by case. Assume that $(x_0,x_1,x_2,x_3,x_4)\notin\Phi_{1,1}\cup\Phi_{1,2}$. We derive a contradiction in the following. Let
	 $$\lambda_1:=\{0\le k\le 4\mid \operatorname{den}(x_k)=n \}.$$

	It is clear that for some $1\le i\le 4$, we have that  $i\in\lambda_1$. For each $k\in\lambda_1$, let $x_k'$ be the numerator of $x_k$. Let $$\lambda_2:=\{k\in\lambda_1\mid \widehat{(\epsilon_{x_k'}x_k')}_1=\widehat{(\epsilon_{x_i'}x_i')}_1 \},$$

	and 
	$$\xi_1:=\{(\epsilon_{x_k'}x_k')_2\mid k\in\lambda_2\backslash\{0\}\}.$$

	\begin{claim}\label{cl.1234}
		$|\xi_1|=4$, i.e. $\xi_1=\{1,2,3,4\}$.
	\end{claim}
	
	Assume for the sake of contradiction that $|\xi_1|< 4$. We analyze the cases $4\notin\xi_1$ and $4\in\xi_1$ separately.\\
	
	(i) The case $4\notin\xi_1$. Consider the element $v:=(a_1,a_2)_n\in X^n$ where $a_1=(\overline{a}_1,\widehat{a}_1)\in\mathbb{Z}^2$ and $a_2\in\mathbb{Z}$ with $\overline{a}_1=0$,  $\widehat{a}_1=\widehat{(\epsilon_{x_i'}x_i')}_1$ and $a_2=(\epsilon_{x_i'}x_i')_2$. Note that $v\in B_n$. Then there exists $j\in \lambda_2\backslash\{0\}$ such that $j\ne i$, $(\epsilon_{x_j'}x_j')_2=(\epsilon_{x_i'}x_i')_2$ and $\overline{(\epsilon_{x_j'}x_j')}_1\ne \overline{(\epsilon_{x_i'}x_i')}_1$. Otherwise $\operatorname{multi}_v(\prod\limits_{k=1}^4\operatorname{tan}x_k)\equiv 2\not\equiv 0\equiv \operatorname{multi}_v((\operatorname{tan}x_0)^2)\operatorname{mod} 4$, a contradiction.	By Lemma \ref{le.congsf4dn}, we get that $x_i+x_j=\frac{\pi}{2}$. Applying Lemma \ref{le.shoeq}, we get that $(x_0,x_1,x_2,x_3,x_4)\in\Phi_{1,1}\cup\Phi_{1,2}$. This contradicts the assumption.\\

	(ii) The case $4\in\xi_1$. Let $j\in\lambda_2\backslash\{0\}$ be such that $(\epsilon_{x_j'}x_j')_2=4$. If there exists $k\in\lambda_2\backslash\{0\}$ such that $k\ne j$, $(\epsilon_{x_k'}x_k')_2=4$ and $\overline{(\epsilon_{x_k'}x_k')}_1\ne \overline{(\epsilon_{x_j'}x_j')}_1$, then $x_j+x_k=\frac{\pi}{2}$. By Lemma \ref{le.shoeq}, we get that $(x_0,x_1,x_2,x_3,x_4)\in\Phi_{1,1}\cup\Phi_{1,2}$. This contradicts the assumption.

		 So for each $k\in\lambda_2\backslash\{0\}$ with $(\epsilon_{x_k'}x_k')_2=4$, we have that $\overline{(\epsilon_{x_k'}x_k')}_1= \overline{(\epsilon_{x_i'}x_i')}_1$, thereby implying that $x_k=x_j$. Let $$\lambda_3:=\{k\in\lambda_2\backslash\{0\}\mid x_k=x_j \}.$$
		 
		We claim that $|\lambda_3|\ge 2$. If $|\lambda_3|=1$, let $a_2\in\{1,2,3,4\}\backslash\xi_1$. Let $a_1:=(0,\widehat{(\epsilon_{x_i'}x_i')}_1)$ and let $v:=(a_1,a_2)_n\in X^n$. It is clear that $v\in B_n$. Then $\operatorname{multi}_v(\prod\limits_{k=1}^4\operatorname{tan}x_k)\equiv 2\not\equiv 0\equiv \operatorname{multi}_v((\operatorname{tan}x_0)^2)\operatorname{mod} 4$, a contradiction. Hence $|\lambda_3|\ge 2$. The claim is true.
		
		Now we proceed to claim that $x_0=x_j$. Assume that $x_0\ne x_j$. Choose $1\le a_2\le 3$ such that it satisfies the following two conditions: (a) If $\operatorname{den}(x_0)=n$ and $(\epsilon_{x_0'}x_0')_2\ne 4$, then let $a_2\ne (\epsilon_{x_0'}x_0')_2$. (b) If $|\{1,2,3\}\backslash\xi_1|=1$, then let $a_2\notin\{1,2,3\}\backslash\xi_1$. Let $a_1:=(0,\widehat{(\epsilon_{x_i'}x_i')}_1)$ and let $v:=(a_1,a_2)_n\in X^n$. It is clear that $v\in B_n$.
		  
		  Assume that $\operatorname{den}(x_0)=n$, $(\epsilon_{x_0'}x_0')_2= 4$
		  and $\widehat{(\epsilon_{x_0'}x_0')}_1=\widehat{(\epsilon_{x_j'}x_j')}_1$. Because $x_0\ne x_j$, we get that $\overline{(\epsilon_{x_0'}x_0')}_1\ne  \overline{(\epsilon_{x_j'}x_j')}_1$. Then $|\operatorname{multi}_v((\operatorname{tan}x_0)^2)-\operatorname{multi}_v(\prod\limits_{k=1}^4\operatorname{tan}x_k)|>0$, a contradiction.
		  
		  Assume that $\operatorname{den}(x_0)< n$, $(\epsilon_{x_0'}x_0')_2\ne 4$
		  or $\widehat{(\epsilon_{x_0'}x_0')}_1\ne \widehat{(\epsilon_{x_j'}x_j')}_1$. Then $|\operatorname{multi}_v(\prod\limits_{k=1}^4\operatorname{tan}x_k)|\ge 2>0= \operatorname{multi}_v((\operatorname{tan}x_0)^2)$, a contradiction.
		Hence $x_0=x_j$. The claim holds.

		Next, we can simplify both sides of \eqref{eq.ne4} by canceling terms, resulting in the relation of the form $(\operatorname{tan}x_f)(\operatorname{tan}x_g)=1$. Consequently, we have that $x_f+x_g=\frac{\pi}{2}$. Therefore, $(x_0,x_1,x_2,x_3,x_4)\in\Phi_{1,1}$. This contradicts the assumption. Hence Claim \ref{cl.1234} is true.

	We need to deal with the cases $0\notin\lambda_2$ and $0\in\lambda_2$ separately.\\
	
	(i) The case $0\notin\lambda_2$. We claim that $0\notin\lambda_1$. Assume that $0\in\lambda_1$. If $(\epsilon_{x_0'}x_0')_2=4$, let $a_2=1$; otherwise, let $a_2:=(\epsilon_{x_0'}x_0')_2$. Let $a_1:=(0,\widehat{(\epsilon_{x_0'}x_0')}_1)$ and let $v:=(a_1,a_2)_n\in X^n$. It is clear that $v\in B_n$. It follows that $\operatorname{multi}_v((\operatorname{tan}x_0)^2)\ne 0=\operatorname{multi}_v(\prod\limits_{k=1}^4\operatorname{tan}x_k)$, a contradiction. Thus the claim is true. Let 
	$$\xi_2:=\{\overline{(\epsilon_{x_i'}x_i')}_1\mid i\in\lambda_2\}.$$ 
	
	By Lemma \ref{le.hom}, we know that $|\xi_2|=1$. Then to every basis element $v$ of level $m$, $\operatorname{multi}_v((\operatorname{tan}x_0)^2)$ is even. By Lemma \ref{le.nosq}, there exists a basis element $u$ of level $m$ such that $\operatorname{multi}_v(\prod\limits_{i=1}^4\operatorname{tan}x_i)=\pm 1$. This is a contradiction. This finishes the proof of the Theorem in the case $0\notin\lambda_2$.\\

	(ii) The case $0\in\lambda_2$. We claim that there exists $1\le i\le 4$ such that $x_i=x_0$. From Claim \ref{cl.1234}, there exists $i\in\lambda_2\backslash\{0\}$ such that $(\epsilon_{x_i'}x_i')_2=(\epsilon_{x_0'}x_0')_2$. If $\overline{(\epsilon_{x_0'}x_0')}_1=  \overline{(\epsilon_{x_j'}x_j')}_1$, then we are done. Assume that $\overline{(\epsilon_{x_0'}x_0')}_1\ne  \overline{(\epsilon_{x_j'}x_j')}_1$. If $1\le (\epsilon_{x_0'}x_0')_2\le 3$, let $a_2=(\epsilon_{x_0'}x_0')_2$. If $(\epsilon_{x_0'}x_0')_2=4$, let $a_2=1$. Let $a_1:=(0,\widehat{(\epsilon_{x_0'}x_0')}_1)$ and let $v:=(a_1,a_2)_n\in X^n$. Then $|\operatorname{multi}_v((\operatorname{tan}x_0)^2)-\operatorname{multi}_v(\prod\limits_{k=1}^4\operatorname{tan}x_k)|>0$, a contradiction. So the claim holds. By relabeling, assume that $i=4$. After simplifying \eqref{eq.ne4}, we get the following relation:
	 
	 \begin{equation*}
	 	(\operatorname{tan}(\frac{\pi}{2}-x_0))(\operatorname{tan}x_1)(\operatorname{tan}x_2)(\operatorname{tan}x_3)=1.
	 \end{equation*}
	 
	 By the same argument as in the case $0\notin\lambda_2$, we get a contradiction. This completes the proof.
\end{proof}

\section{A Further Generalization}\label{sec.gen}

In this section, we consider the Equation \eqref{eq.ne9} which generalizes Equation \eqref{eq.ne4}. We will show in Proposition \ref{pro.6var} that the solutions of the two Equations are related.

Consider the following equation:
\begin{equation}\label{eq.ne9}
	(\operatorname{tan}x_0)^2=(\operatorname{tan}x_1)(\operatorname{tan}x_2)(\operatorname{tan}x_3)(\operatorname{tan}x_4)(\operatorname{tan}x_5).
\end{equation}

\begin{definition}
	
	Let $n\in \mathbb{N}_{\ge 3}$. Assume that $v\in X^n$ and $v=\prod\limits_{i=1}^{k}v_i^{f_i}$ in $\widehat{X^n}$ with $v_i\in B_n$. We define the {\it degree of level $n$} of $v$, denoted by $\operatorname{deg}_n(v)$, to be $\sum\limits_{i=1}^{k}f_i$.

\end{definition}

\begin{proposition}\label{pro.6var}
	Let $n$ be an odd prime. If $(x_0,x_1,x_2,x_3,x_4,x_5)$ is a solution to \eqref{eq.ne9} in $G$ with $0<x_i<\frac{\pi}{2}$ and $\operatorname{den}(x_i)\in\{4,n,2n,4n\}$ for each $0\le i\le 5$, then $x_j=\frac{\pi}{4}$ for some $1\le j\le 5$.

\end{proposition}

\begin{proof}
	We can assume that $x_i=\frac{x_i'}{4n}\pi$ where  $x_i'\in \mathbb{Z}$ and $0<x_i'<2n$ for each $0\le i\le 5$. Assume that $x_0\ne \frac{\pi}{4}$. We consider the cases $x_0'$ is odd and even separately.

		 (i) The case $x_0'$ is odd. By the same argument as in Theorem \ref{th.4pr2}, we can assume that \[\operatorname{tan}x_3\operatorname{tan}x_4\operatorname{tan}x_5=1.\] 
		
		Let $C:=\{3\le i\le 5\mid \operatorname{den}(x_i)=4n\}$.
		Let $c:=|C|$. 
		
		Assume that $c=0$. If $x_i\ne \frac{\pi}{4}$ for each $3\le i\le 5$, then $\operatorname{deg}_n(\prod\limits_{i=3}^5\operatorname{tan}x_i)\ne 0$ as it is odd. This is a contradiction. So $x_i=\frac{\pi}{4}$ for some $3\le i\le 5$.

		Assume that $c=2$ and $x_3'$ and $x_4'$ are odd. Similar to the argument in Lemma \ref{le.02}, we can show $\operatorname{tan}x_3\operatorname{tan}x_4=1$. So $\operatorname{tan}x_5=1$. Hence $x_5=\frac{\pi}{4}$.
		
		Assume that $c=1$ or $3$. Then $\operatorname{deg}_{4n}(\prod\limits_{i=3}^5\operatorname{tan}x_i)\equiv 2\operatorname{mod}4$. This is a contradiction.

	(ii) The case $x_0'$ is even. Let $C:=\{0\le i\le 5\mid \operatorname{den}(x_i)=4n\}$.
		Let $c:=|C|$. Similar to the argument in Lemma \ref{le.02}, we know that $c$ is even and $\prod\limits_{\substack{i=1\\i\in C}}^5\operatorname{tan}x_i=1$ in $\mathbb{C}$. If $x_i\ne\frac{\pi}{4}$ for each $1\le i\le 5$, then $\operatorname{deg}_n(\prod\limits_{\substack{i=1\\i\notin C}}^5\operatorname{tan}x_i)$ is odd. This is a contradiction. So $x_i=\frac{\pi}{4}$ for some $1\le i\le 5$.

	Assume that $x_0=\frac{\pi}{4}$. Then 
	\begin{equation*}
		(\operatorname{tan}x_1)(\operatorname{tan}x_2)(\operatorname{tan}x_3)(\operatorname{tan}x_4)(\operatorname{tan}x_5)=1.
	\end{equation*}
	
	Similar to the argument in (i), we get that $x_i=\frac{\pi}{4}$ for some $1\le i\le 5$.
\end{proof}

\bibliographystyle{plain}

\bibliography{bibli.bib}

\end{document}